\documentclass[11pt]{article}
\usepackage{amsmath, amsthm, amssymb,  color, graphicx, amsfonts, sectsty, caption, subcaption, tikz, bbm, algorithm, url, capt-of}
\RequirePackage{bm} 

 \sectionfont{\rmfamily\mdseries\scshape\normalsize}
\subsectionfont{\rmfamily\mdseries\normalsize}
\subsubsectionfont{\rmfamily\mdseries\normalsize}
\usepackage[margin=1in]{geometry}
\usepackage{color}
\usepackage{graphicx}
\usepackage{amsfonts}
\usepackage{fancyhdr}
\usepackage{booktabs}

\usepackage{jbradic_definitions}
 \usepackage{wrapfig}
 
\usepackage{multirow}
 \usepackage{tabularx}
\usepackage{array}
\usepackage[square,sort,comma]{natbib}
\setcitestyle{authoryear}

\usepackage{color}

\title{Robustness  in  sparse  linear models: relative efficiency \\
 based on robust approximate message passing}
\author{Jelena Bradic  \\ \footnotesize  Department of Mathematics \\ \footnotesize University of California, San Diego\footnote{\today}}
\date{ }
\begin{document}
\maketitle

\begin{abstract}
{\small

Understanding efficiency in high dimensional linear models is a longstanding problem of interest. Classical work with smaller dimensional problems dating back to Huber and Bickel has illustrated the clear benefits  of efficient loss functions. When the number of parameters $p$ is of the same order as the sample size $n$, $p \approx n$,  an efficiency pattern different from the one of Huber was recently established.
 In this work, we  consider the effects of model selection on the estimation efficiency of penalized methods. 
 In particular, we explore 
whether sparsity results in new efficiency patterns    
 when $p > n$.
In the interest of deriving the   asymptotic mean squared error for regularized M-estimators, we use the powerful framework of approximate message passing.
We propose a novel, robust and sparse approximate message passing algorithm (RAMP), that is adaptive to the error distribution. Our  algorithm includes many non-quadratic and non-differentiable loss functions, therefore extending previous work that mostly concentrates on the least square  loss. 
 We derive its asymptotic mean  squared error and show its convergence, while allowing $p, n, s \to \infty$, with $n/p \in (0,1)$ and $n/s \in (1,\infty)$.
We identify new patterns of relative efficiency regarding a number of penalized $M$ estimators,
when  $p$ is much larger than   $n$.    We show that the classical information bound is no longer reachable,  even for  light--tailed error distributions.  We show that the penalized least absolute deviation estimator  dominates the penalized least square estimator, in cases of heavy tailed distributions. We observe this pattern for all choices of the    number of non-zero parameters $s$, both $s \leq n$ and $s \approx n$. In non-penalized problems where  there is no sparsity, i.e., $s =p \approx n$,  the opposite regime holds. Therefore, we discover that the presence of model selection   significantly changes the efficiency patterns.

}

\end{abstract}

\section{Introduction}
 

In recent years,  scientific communities  face major challenge with the  size and complexity  of the data analyzed. The size of such contemporary datasets and the number of variables collected
makes the search for, and exploitation of sparsity  vital to their statistical analysis. Moreover, the presence of heterogeneity, outliers and anomalous data in such samples is  very common.
   %
 However,
 statistical estimators that are not  designed for  both sparsity and  data irregularities simultaneously will give biased results, depending on the ``magnitude" of the deviation and  the ``sensitivity" of the method.

An example of an early work on robust statistics is  \cite{BA55,B53}. Specifically, they argue that a good statistical procedure should be insensitive to changes not involving the parameters, but should be effective in being sensitive to the changes of parameters to be estimated. 
Estimators based on a minimization of  non-differentiable loss functions are one common example of such estimators; in particular, the maximum likelihood  loss for generalized Laplace density with parameter $\alpha \in(0,1)$, takes the form $-\alpha |Y-A x|^{\alpha-1} \mbox{sign}(Y-Ax)$. 
Subsequently, \cite{T60}  discussed a telling example in which
very low frequency events could utterly destroy the average performance of
 optimal statistical estimators.
%
%
%
%
%
%
%
%
These observations led to a number of papers by  \cite{H60}, \cite{H68} and \cite{B75} who laid the comprehensive foundations of a theory of robust statistics.
 In particular, Huber's seminal 
work on M-estimators \citep{H73} established asymptotic properties of a class of  M-estimators in  the situation where the number of parameters, $p$, is fixed and the number of samples, $n$, tends to infinity.
Since   then,    numerous important
steps have been taken toward analyzing and quantifying robust statistical methods --
notably in the work of \cite{R84,DL88,Y87},   among others.
 Even today, there exist several (related) mathematical concepts of robustness (see \cite{M06}). 
More recently the work of \cite{K13} and \cite{DM13} illuminated surprising and novel robustness  properties of the least squares estimator, when the number of parameters is very close to the number of samples. 
 This illustrates diverse and rich  aspects of robustness and its intricate dependence on the dimensionality of the parameter space. 
%

Classical M-estimation theory ignored model selection out of necessity. Modern computational power allows statisticians to  deal with model-selection problems more realistically. Hence, statisticians have moved away from the M-estimators and started working on the penalized M-estimators; moreover, they     allow the number of parameters, $p$, to grow with the sample size, $n$. 
To further the focus on penalized M-estimators, we 
consider a  linear regression model: 
\begin{equation}\label{w}
  Y={A}  x_o+ W
\end{equation}
with $  Y=(Y_1, ... ,Y_n)^T \in \mathbb{R}^n$ a vector of responses, $A \in\RR^{n\times p}$ a known design matrix, ${ x}_o \in \RR^p$ a vector of parameters;  the noise vector ${ W}=(W_1, ..., W_n)^T \in \RR^n$   having i.i.d, zero-mean components each with distribution $F=F_w$  and a density function $f_w$. 
When  $p$, overcomes $n$ -- in particular when  $p \geq n$ --
a form of    sparsity   is imposed on the model parameters $x_o$, i.e., it is imposed that $\mbox{supp}({ x}_o) =\{ 1 \leq j \leq p: {x_o}_{j} \neq 0\} $ with $|\mbox{supp}({ x}_o)| = s$. 
Early work on sparsity inducing estimators,   includes  penalized   least squares (LS) estimators with various penalties  including $l_1$-penalty, Lasso,  \citep{T96}, concave penalty, SCAD \citep{FL01} and MCP \citep{Z10}, adaptive $l_1$ penalty \citep{Z06}, elastic net penalty \citep{ZH05}, and many more.  However,  when the error distribution $F_w$ deviates from the normal distribution, the $l_2$ loss function is typically changed to the $- \log f_w$. Unfortunately, in real life situations the error distribution $F_w$ is unknown and a method that adapts to many different distributions is needed. 
Following classical literature on M-estimators, penalized robust methods such as  penalized Quantile regression \citep{BC11},  penalized Least Absolute Deviation estimator \citep{WL13}, AR-Lasso estimator \citep{FFB14}, robust adaptive Lasso \citep{AMR14} and many more, have been proposed. These methods  penalize a convex loss function $\rho$ in the following manner
\begin{equation}\label{eqDef}
\hat{x}(\lambda) \equiv \underset{x\in\mathbf{R^p}}{\operatorname{arg  min \   }}  \mathcal{L}(x) =  \underset{x\in\mathbf{R^p}}{\operatorname{arg  min \   }}\sum_{i = 1}^n \rho (Y_i - A_i^Tx )+\lambda \sum_{j=1}^p P (x_j),
\end{equation}
for a suitable penalty function $P$. In the above display, $A_i$ denotes the $i$-th row-vector of the matrix $A$.

Despite the substantial
body of   work on robust M-estimators, 
there is very little work on robust properties of penalized M-estimators. 
Robust assessments of penalized statistical estimators customarily are made ignoring model selection.
 Typical properties discussed are   model selection consistency  or  tight  upper bounds on the statistical estimation error (e.g., \cite{BFW11,N12,FFB14,FLW14,LMTZ15,L15, L11,W13,C14}). 
In particular, the existing work 
  has  been primarily reduced to   the tools that are intrinsic to Huber's 
  M-estimators.  In order to do that, the authors establish  a model selection consistency and then reduce the analysis to this selected model assuming that the selected model is the true model.
However,  this analysis is dissatisfactory, as the necessary assumptions for the model selection consistency are far too restrictive. Hence, departures from such considerations  are highly desirable. They are also  difficult to achieve because  the analysis needs to  factor in    the model selection bias.    This is where our work makes progress. By   including the bias of the model selection in the analysis 
 we are able to answer question like: in high dimensional regime, which estimator is preferred?    In the
low-dimensional setting, several independent lines of work provide reasons for using  
 distributionally robust estimators over their least-squares  alternatives \citep{H81}. 
 However, in high dimensional setting, 
it remains an open question, what are the advantages of using a complicated
loss function over a simple loss function such as the squared loss? Can we better understand  how differences between probability distributions  affect penalized M-estimators? 
One powerful justification exists, using
the  point of view of statistical efficiency.  
\cite{H73} introduced the concept of minimax asymptotic variance estimator that achieves the minimal asymptotic variance for the least favorable distribution; the smaller the variance the more robust the estimator is.

Huber's proposed measure of robustness allows a  comparison of  estimators by comparing  their asymptotic variance; one caveat is that the two estimators need to be consistent    up to the same order.  
For cases with $p \geq n$ little or nothing is known about the asymptotic variance of the robust estimator \eqref{eqDef} as $p \to \infty$ whenever $n \to \infty$. Moreover, the
penalized M-estimator  is biased as it shrinks many coefficients to zero. For such estimators, the set of parameters for which Hodge's super--efficiency  occurs is not of measure zero.  
Hence, asymptotic variance may not be the most optimal criterion for comparison.   
This suggest that a different criterion for comparison needs to be considered  in the high dimensional asymptotic regime where $n \to \infty$, $p \to \infty$ and $n/p \to \delta \in (0,1)$. We  examine the asymptotic mean squared error (denoted with AMSE from hereon). AMSE
 is an effective  measure of efficiency as it combines both the effect of the bias and of the variance \citep{DL88}.  However, in $p \gg n$ regime, it is not obvious that the asymptotic  mean squared error  will satisfy the classical formula. 
 
 AMSE was studied in \cite{B13a,K13}
 for the case of ridge regularization, with $P(x_i)=x_i^2$, and when $p \leq n$ but $p \approx n$. In this setting AMSE is equal to the asymptotic  variance of $\hat \xb(\lambda)$. They discovered   a new Gaussian  component in the AMSE of $\hat x(\lambda)$ that cannot be explained by the traditional Fisher Information Matrix.  To analyze AMSE for the case of no-penalization, with $p \approx n$, \cite{DM13}  utilized the  techniques of Approximate Massage Passing (AMP) and discovered the same  Gaussian component,  in the $\hat x(\lambda)$.    The advantage of the AMP framework is that it provides an exact asymptotic expression of the asymptotic mean squared error of the   estimator   instead of an upper bound. Here, we theoretically   investigate the  applicability of the AMP techniques when $p \geq n$ and the loss function is not-necessarily least squares or differentiable. For the case of the least squares loss
 with  $p \geq n$,
\cite{BM12}  make a strong connection between the  penalized least squares and    the  AMP algorithm of \cite{DMM09}. 
    However, the AMP algorithm of \cite{BM12} cannot recover the signal when the distribution of the noise is arbitrary.  For this settings, we design  a new, robust and sparse Approximate Message Passing (RAMP) algorithm.

 The proposed RAMP algorithm is not the first algorithm to consider improving the AMP framework as a means of  adapting to the problem of different loss function; however, it is the first that simultaneously  allows  shrinkage in estimation. \cite{DM13} propose  a three-stage AMP algorithm  that matches  the classical M-estimators; however,  it merely applies to the  $p  \leq n$ case. When $p>n$, the second step of their algorithm fails to iterate and the other two stages do not match with \eqref{eqDef}.   Our proposed algorithm belongs to the general class of   first-order approximate massage passing algorithms.  However, in contrast to the existing methods  it  has three-steps. It has iterations that are  based on gradient descent with  an objective that is scaled and min regularized  version of the original loss function $\rho$.  Moreover, it allows non-differentiable loss functions.  The three--step estimation method of RAMP is no longer a simple  proxy for the one-step M estimation. Due to high dimensionality with $p \geq n$, such a step is no longer adequate. Our proof technique leverages 
 the powerful technique of the AMP proposed  in \cite{BM11}; however, we
require a more refined analysis here in order to extend the results to one involving  non-differentiable and robust  loss functions while simultaneously allowing  $p \geq n$. We relate the proposed algorithm to the penalized M estimators when $p \gg n$ and show that a solution to one may lead to the solution to the other.
 We show its convergence while allowing non-differentiable loss functions and  $p, n, s \to \infty$, with $n/p \to \delta \in (0,1)$ and $n/s \to a  \in (1,\infty)$. This enabled us to derive the  AMSE of a general class of $l_1$ penalized M-estimators and to study their relative efficiency.
 
  We show that the AMSE depends on the distribution of the {\it effective score} and that it takes a form much different than the classical one, in that it  also depends on the sparsity parameter $s$. Moreover, we present a detailed  study of the relative efficiency of the penalized least squares method  and  the penalized least absolute deviation method. We discover regimes where one is more preferred than the other and that do not match classical  findings of Huber. 
Several important insights follow immediately:  relative efficiency is considerably affected by the model selection step; the most optimal loss function may no longer be the negative log likelihood function;  even  the sparsest high dimensional estimators have an additional Gaussian component in their asymptotic mean squared error that does not disappear asymptotically. 

We briefly describe the notation used in the paper.
We use   $\langle u\rangle \equiv \frac{\sum^{m}_{i=1}u_i}{m} $ to denote the average of the vector $u\in \mathbb R^m$.  Moreover,
if  given $f: \mathbb{R}\to\mathbb{R}$ and $v = (v_1,...,v_n)^T \in \mathbb{R}^m$, we define $f(v)\in\mathbb{R}^m\equiv (f(v_1),...,f(v_m))$. Moreover, its subgradient $f'(v)$ is taken coordinate-wise and is  $(f'(v_1),\dots,f'(v_m))$.
For   bivariate function $ f(u,v)$, we define $\partial_1 f(u,v) $ to be the partial derivate with respect to the first argument; similarly  $\partial_2 f(u,v)$, is the partial derivate with respect to the second argument.
We use $\| \cdot\|_1$ to denote $l_1$ and $\|\cdot\|_2$ to denote the $l_2$ norm. We define the sign function as sign$(v) =\mathbbm{1}\{v >0\}  -\mathbbm{1}\{v <0\}$, and zero whenever $v=0$. We use $\Phi$ and $\phi$ to denote the cumulative distribution function and density function of the standard normal random variable.

This paper investigates the effects of   the $l_1$ penalization  on robustness properties of the penalized estimators, in particular, how to incorporate bias induced by the penalization in the exploration of   robustness. We present a scaled, $\min$-regularized, not necessarily differentiable, robust loss functions  for  penalized M-estimation such that the corresponding approximate massage passing algorithm (RAMP) is adaptable to different loss functions and sparsity simultaneously.  Four examples of $\min$-regularized losses, that include the one of Least Absolute Deviation (LAD) and Quantile loss, are introduced in Section \ref{sec:loss}. The corresponding  algorithm is a modified form of AMP for robust losses which  offers offers a more general framework  over the standard AMP method and is discussed in Section \ref{sec:RAMP}. Section \ref{sec:theory} studies a number of important theoretical results concerning the RAMP algorithm as well as its convergence properties and its connections to the penalized M-estimators. Section \ref{sec:RE} studies Relative Efficiency and establishes lower bounds for the AMSE. Moreover, this section also presents  results on relative efficiency of penalized least absolute deviations (P-LAD) estimator with respect to the penalized least squares (P-LS) estimator. We find that P-LS is preferred over P-LAD when the error distribution is ``light-tailed" with a new breakdown point for which the two methods are indistinguishable; furthermore, we find that P-LS is never preferred over P-LAD when the error distribution is ``heavy-tailed". Section \ref{sec:examples} contains detailed numerical experiments on a number of RAMP losses, including LS, LAD, Huber and a number of Quantiles, and a number of error distribution, including normal, mixture of normals, student.    In \ref{sec:tuning}- \ref{sec:con}, we demonstrate both how to use RAMP method in practice, and analyze its finite sample convergence properties.  The second subsection involves study of state-evolution equation whereas the third subsection involves the study of the AMSE.  In both studies, we find that the RAMP works extremely well.  In \ref{sec:error}-\ref{sec:design}, we demonstrate good properties of the RAMP algorithm with varying error distribution and the distribution of the design matrix $A$. Lastly, in \ref{sec:RE1} we present analysis of relative efficiency between P-LS and P-LAD estimators where we consider both $p \leq n$ and $p \geq n$. We demonstrate that the results of RAMP for $p\leq n$ match those of \cite{B13a} and \cite{K13} whereas for $p \geq n$ the results establish new patterns according to Section \ref{sec:RE}.

\section{$\min$-Regularized Robust Loss Functions}\label{sec:loss}

We consider the loss function $\rho : \RR \rightarrow \RR_{+}$  to be a non-negative  convex function with subgradients $\rho'$, defined as 
\[
\rho'(x) = \left\{ y| \rho(z) \geq \rho(x) + y (z-x), \mbox{ for all } z\in \RR \right\}.
\] If $\rho$ is differentiable, $\rho'$ represents the first derivative of $\rho$. 
%
Our assumption includes some interesting cases, such as  least squares loss, Huber loss, quantile loss and  least absolute deviation loss.

Similarly to \cite{DM13} and \cite{K13} we use $\min$ regularization to regularize the squared loss with the robust loss $\rho$.
This introduces the family of regularizations  of  the robust loss $\rho$ as  follows:
\begin{equation}
\rho(b,z) = \underset{x\in\mathbb{R}}{\operatorname{min \   }}\left\{b\rho(x)+\frac{1}{2}(x-z)^2\right\}.
\end{equation}
This family is often named a Moreau envelope or Moreau-Yosida regularization. The Moreau envelope is continuously
differentiable, even when $\rho$ is not. In addition, the sets of minimizers
of $\rho$ and $\rho(b,z)$ are the same.
Related to the family of the regularized
  loss functions $\rho(b,z)$ is  
 the proximal mapping operator of the functions $\rho(b,z)$,   defined as:
\begin{equation}\label{eq:prox}
Prox(z,b) = \underset{ x \in \RR } { \mbox{arg  min    }}
\left\{b\rho(x)+\frac{1}{2}(x-z)^2\right\}.
 \end{equation}
 For all convex and closed losses $\rho$,  the operator $Prox(z,b)$ exists for all $b$ and  is unique for big enough $b$ and all $z$.
 Moreover, it 
  admits the subgradient characterization; if   $Prox(z,b)=u$ then 
  $$ z-u \in b\rho'(u). $$
The proximal mapping operator is widely used in  non-differentiable  convex optimization   in defining proximal-gradient methods. The parameter $ b$ controls
the extent to which the proximal operator maps points towards the
minimum of $\rho$, with smaller values of $b$ providing a smaller movement towards
the minimum. Finally, the fixed points of the proximal operator of $\rho$ are precisely
the minimizers of $\rho$; for appropriate choice of  $b$,  the proximal
minimization scheme  converges to the   optimum of $\rho$,  with least geometric and possibly superlinear
rates (\cite{BT99}; \cite{LT95}).
For each $\rho(b,z)$,   we define a corresponding {\it effective score} function as:
 \begin{equation}
\Phi(z; b)= b\rho'(Prox(z,b)). 
\end{equation}

The score functions   $\Phi(z;b)$  were used in \cite{DM13}, for two times differentiable losses $\rho$,
 to define  a new iteration step in the robust Approximate Message  Passing algorithm   therein. 
 In this paper, we extend their method  for sparse estimation and discuss  loss functions $\rho$ that are not    necessarily differentiable  and those that do not necessarily satisfy restricted strong convexity condition \citep{N12}.  Important examples of such loss functions $\rho$ are  absolute deviation and quantile  loss, as they are neither differentiable nor do they satisfy restricted strong convexity condition. The extension is significantly complicated, as the set of fixed points of the proximal operator is no longer necessarily sparse; moreover,  trivial inclusion of the $l_1$ norm in  the definition \eqref{eq:prox} does not provide an algorithm that belongs to the Approximate Message Passing family or that converges to the penalized M-estimator \eqref{eqDef}.


In the rest  of this section, we provide examples of the effective score function above, with a number of different losses $\rho$. 

 \begin{example} \rm
[Least Squares Loss]
The least squares loss function is  defined as $
\rho(x)=\frac{1}{2}x^2$.
Setting the first derivative to be 0, the proximal map \eqref{eq:prox} satisfies
$bProx(z,b)+Prox(z,b)-z=0,
$
which simultaneously  provides  the proximal map and the effective score function
$$Prox(z;b) = \frac{z}{1+b}, \qquad \Phi(z;b) = \frac{b}{1+b}z.$$ 
\end{example}

\begin{example} \rm
[Huber Loss]
Let $\gamma>0$ be a fixed positive constant.
Huber's  loss function is defined as
\begin{equation}
\rho_H(x,\gamma)= \left\{\begin{array}{cc}
\frac{x^2}{2} &\text{if }|x|\leq \gamma\\
\gamma|x|-\frac{\gamma^2}{2}& \mbox{otherwise};
\end{array}
\right.
\end{equation}
hence, the family of loss functions depends on the new tuning parameter $\gamma$ and is defined as
\begin{equation}
\rho(z,b,\gamma)\equiv \underset{x \in\RR}{\operatorname{min \   }} \left\{b\rho_H(x,\gamma)+\frac{1}{2}(x-z)^2 \right\}.
\end{equation}
Thus, the effective score function  depends on  new parameter $\gamma$, so we use $\Phi(z;b,\gamma)$ to denote its value. Moreover,
we notice that whenever $| Prox(z,b)| \leq \gamma$, the proximal mapping operator takes on the same form as in the least squares case, i.e., it is equal to $z/(b+1)$.
In more general form, we conclude
\[
\rho'_H(Prox(z,b,\gamma)) = \text{ min } ( \text{max } (-\gamma,Prox (z,b,\gamma)),\gamma),
\]
 and with it that 
 \[
 Prox (z,b)= \left\{
 \begin{array}{cc}
 \frac{z}{1+b}, & |z| \leq (1+b)\gamma\\
 z-b\gamma, & z>(1+b)\gamma \\
 z+b\gamma, & z < - (1+b)\gamma
 \end{array}
 \right. .
 \]
Hence, the Huber effective score function
\begin{align}
\Phi(z;b,\gamma)&= 
 \left\{\begin{array}{cc}
\frac{bz}{1+b},&   |z| \leq (1+b)\gamma\\
b \gamma, &  z > (1+b)\gamma\\ 
-b \gamma & z<-(1+b)\gamma
\end{array}
\right..
\end{align}
 
\end{example}


\begin{example} \rm\rm
[Absolute Deviation Loss] 
The Absolute Deviation  loss function is defined as 
$\rho(x)=|x| $.
According to \eqref{eq:prox}, we observe that proximal mapping operator satisfies
$
b\rho'(Prox(z,b))+Prox(z,b)-z \in 0
$. We consider $Prox(z,b) \neq 0$ first. We observe that  $Prox(z,b)<0 $, when $z<-b$ and  $Prox(z,b)>0 $ when $z>b$.
This indicates that $\mbox{sign}(Prox(z,b))=\mbox{sign}(z)$. Substituting it in the  previous equation, we get $Prox(z,b)=z-b\mbox{ sign}(z).$
Next, we observe that when $Prox(z,b)=0$ we have ${\partial(b|x|)}/{\partial x}=b\xi$, where $\xi\in(-1,1).$ Substituting it in the proximal mapping equation, we get $z\in (-b,b)$.
Above all, we obtain
\begin{equation}
Prox(z,b)=\left\{\begin{array}{cc}
0,&z\in(-b,b)\\
z-b\mbox{ sign}(z),& \mbox{otherwise}
\end{array}
\right..
\end{equation}
Observe that the form above is equivalent to the soft thresholding operator.  Moreover, the Absolute Deviation effective  score function becomes, 
\begin{align*}
\Phi(z;b)&=
             \left\{\begin{array}{cc}
             z, &z\in(-b,b)\\
             b\mbox{{ sign}}(z),& \mbox{otherwise}
             \end{array}
             \right. .
\end{align*}
\end{example}

\begin{example} \rm
[Quantile Loss]
Let $\tau$ be a fixed quantile value and such that $\tau \in (0,1)$.
The quantile loss function is defined as 
$$\rho_\tau (x) = |x| \biggl( (1-\tau) 1\{ x<0\} + \tau 1\{x>0\} \biggl) = \tau x_+ +(1-\tau) x_-,$$
 for  $x_+ = \max\{x,0\}$ and $x_-=\min\{x,0\}$.  
  The family of $\min$ regularized loss function is  then defined as follows
 \[
 \rho(z,b,\tau)\equiv \underset{ x \in\RR}{\operatorname{min \   }} \left\{b \rho_\tau(x)+\frac{1}{2}(x-z)^2 \right\}.
 \]
Similarly, as before,
$
b\rho'(Prox(z,b))+Prox(z,b)-z \in 0
$.
Now, we first consider $Prox(z,b) \neq  0$, in which case we obtain 
\[
  \rho'(Prox(z,b))=  \mbox{sign} \left( Prox(z,b)  \right)   \biggl( (1-\tau) 1\{ Prox(z,b) <0\} + \tau 1\{Prox(z,b) >0\} \biggl).
\]
Next, we observe that when $Prox(z,b)=0$ we have ${\partial( \rho_\tau (x))}/{\partial x}=b\xi ((1-\tau) 1\{ x<0\} + \tau 1\{x>0\} )$, where $\xi\in(-1,1).$ Analyzing the positive and negative parts separately, we see that ${\partial( \rho_\tau (x))}/{\partial x} = b \tau \xi$ and ${\partial( \rho_\tau (x))}/{\partial x} = b (1- \tau) \xi$, respectively.
 Hence,  
\[
Prox(z,b)=  \left\{ 
\begin{array}{cc}
z-b\tau,& z>b\tau  \\
z-b( \tau-1),&z<b(\tau-1) \\
0,& \mbox{otherwise}
\end{array}
\right.  ,
\]
%
and with it that the Quantile score function becomes, 
\begin{align*}
\Phi(z; b,\tau)&=
             \left\{\begin{array}{cc}
             z,&z\in(b\tau-b,b\tau)\\
             b\tau ,& z > b\tau \\
             b(\tau-1) ,& z < b\tau-b 
             \end{array}
             \right. .
\end{align*}
 
\end{example}

In order to establish   theoretical properties, we will impose a number of conditions on the  density of the error term $W$, a class of robust loss functions $\rho$ and a design matrix $A$.   More precisely we impose the following conditions.\\

 {\bf Condition (R)}: Let $i=1,\dots, n$. The loss function $\rho$ is convex with sub-differential $\rho'$. It satisfies:
 \begin{itemize}
 \item[(i)] For all $u \in \RR$, $ \rho'(u)$ is an absolutely continuous function which can be decomposed as  
 \[
 \rho'(u)=\upsilon_1(u) + \upsilon_2(u) + \upsilon_3(u)
 \]
 where $\upsilon_1$ has an absolutely continuous derivative $\upsilon_1'$, $\upsilon_2$ is a continuous, piecewise linear continuous function, constant outside a bounded interval and $\upsilon_3$ is a nondecreasing step function. In more details, 
\begin{itemize}
\item 
$
 \upsilon_2'(u) = \kappa_\nu,  \qquad  \mbox{for } q_\nu < u \leq q_{\nu+1}, \qquad \nu =1,\cdots,k,
$

 for  $\kappa_0,\dots,\kappa_k  \in \RR$ with $\kappa_0=\kappa_k =0$ and $-\infty =q_0 < q_1 < \cdots<q_k <q_{k+1}=\infty,$ and $-\infty =\kappa_0 < \kappa_1 < \cdots<\kappa_k <\kappa_{k+1}=\infty.$
 \item
 $
 \upsilon_3(u) = \alpha_\nu,   \qquad  \mbox{for } r_\nu < u \leq r_{\nu+1}, \qquad \nu =1,\cdots,k,
$

 for  $\alpha_0,\dots,\alpha_k,  \in \RR$ with $\alpha_0=\alpha_k = 0$ and $-\infty =r_0 < r_1 < \cdots<r_k <r_{k+1}=\infty,$ and $-\infty =\alpha_0 < \alpha_1 < \cdots<\alpha_k <\alpha_{k+1}=\infty.$
 \end{itemize}
 \item[(ii)]  For all $u \in \RR$, $|\rho'(u)| \leq k_0$, where $k_0$ is   positive and bounded constant.

 \item[(iii)] The functional $h(t)=\int \rho(z-t)dF(z)$ has unique minimum at $t=0$.
 
 \item[(iv)] For some $\delta>0$ and $\eta >1$, $ \EE \left[ \sup_{|u|\leq \delta} | v_1^{''} (z+u)|  \right]^\eta  $ is finite; where, $v_1^{'}(z)=(d/dz) v_1(z)$ and $v_1^{''}(z)=(d^2/dz^2) v_1(z)$.

%
\end{itemize}
 The first Condition (i) depict explicitly the trade--off between the smoothness of $\rho$ and smoothness of $F$. This assumption covers the classical  Huber's and Hampel's  loss functions.
 Although we allow for not necessarily differentiable loss functions, we consider  a class of loss functions for which the sub-differential $\rho'$ is bounded. This lessens the effect of gross outliers and  in turn leads to many good robust properties of the resulting estimator. Least squares loss does not  satisfy this property but the AMP iteration with least squares loss has been studied in \cite{BM12} and its asymptotic mean squared error derived therein. For all other losses discussed above, this property holds.
The third Condition (iii), is to assure uniqueness of the population parameter that we wish to estimate.
The fourth Condition (iv)  is essentially a moment condition that holds, for example, if $v_1^{''}$ is bounded  and either $v_1^{''}(z)=0$ for $z<a$ or $z>b$ with $- \infty <a <b <\infty$, or $\EE |W|^{2 +\epsilon} < \infty$ for some $\epsilon >0$.\\

 {\bf Condition (D)}:   Let $W_1,\dots,W_n$ be i.i.d. random variables with the distribution function $F$.  
Let $F$ have two bounded derivatives $f$ and $f'$ and $f>0$ in a neighborhood of either $q_1,\cdots, q_k$ or  $r_1,\cdots, r_k$  appearing in Condition ({\bf R})(i) above. 
 \\
 
Although we assume  that the error terms $W_i$'s have bounded density, we allow for densities with possibly unbounded moments and we do not assume any a-priori knowledge of the density $f$.\\

 {\bf Condition (A)}: The design matrix $A$ is such that  $A_{ij}$ are i.i.d and follow Normal distribution $\mathcal{N}(0,1/n)$ for all $1 \leq i \leq n$ and $1 \leq j \leq p$. Moreover, the vector $x_0$ is  such that  its  empirical cumulative distribution function  converges weakly to a distribution $\PP_{x_0}$ as $p \to \infty$. Additionally $\PP_{x_0} \in \mathcal{F}_{\omega}$,    where $\mathcal{F}_{\omega}$ denotes the set of distributions whose mass at zero is greater than or  equal to $1-\omega$, for $\omega \in (0,1)$. \\

The class of distributions $\mathcal{F}_{\omega}$ has been studied in many papers \citep{BM12,DMM09,Z15}, it implies $\PP(x_0 \neq 0) \leq \omega$ and is considered a good model for exactly sparse signals.
 While
this setting is admittedly specific, the careful study of such matrix ensembles has a long tradition
both in statistics and communications theory and    is borrowed from the AMP formulation   \citep{BM12}. It simplifies the analysis significantly and can be relaxed if needed.  In particular,  it implies the Restricted Eigenvalue condition  \citep{B09}; that is, the design matrix $A$ is such that 
 \[
\kappa(s,c) = \min_{J \subset \{1,\dots p\}, |J| \leq s} \min_{v \neq 0, \|v_{J^c}\|_1 \leq c \|v_J\|_1} \frac{ \| A v\|_2}{ \sqrt{n} \|v_j \|_2} >0
 \]
with high probability, as long as the sample size $n$ satisfies $n > c' (1+ 8c)^2 s \log p/ \kappa(s,c)^2  $, for some universal constant $c'$.
 The integer $s$ here plays the role of an upper bound on the sparsity   of a
vector of coefficients $x_0$. Note that,  with $c \geq 1$, 
   the square submatrices of size $\leq 2s$ of the  matrix  $\frac{1}{n} \sum_{i=1}^n A_i A_i^T$ are necessarily
positive definite.  

\section{Robust Sparse Approximate Message Passing}\label{sec:RAMP}

We propose an algorithm called RAMP, for ``robust approximate message passing."  Our proposed algorithm is iterative and starts  from the  initial estimate $x^0=0\in\RR^p$ and guarantees a sparse estimator at its final iteration.  During iterations $t=1,2,3,\dots$ our algorithm applies a three-step procedure to update its estimate $  x^t\in\RR^p$, resulting in a new estimate $x^{t+1}\in\RR^p$.  We name the iteration steps as the Adjusted Residuals, the Effective Score and the Estimation Step.

We use the following notation  $\delta = n/p <1$ and  $\omega = \EE \| x_0 \|_0$ with $x_0\sim \PP_{x_0}$.  
We set  $\omega = s/p$.  We use
$G(z;b)$ to denote the rescaled, min regularized effective score function, i.e., $$ G(z;b)= \frac{\delta}{\omega} \Phi(z;b). $$
Let $\theta$ denote the nonnegative thresholding parameter  and let $\eta: \RR \times \RR_+ \to \RR $ be the soft thresholding function  
\[
 \eta(x,\theta) =  \left\{\begin{array}{ll}
                      x-\theta & \mbox{if $x>\theta$};\\
                       0          &  \mbox{if $-\theta \leq x \leq \theta$};\\
                    x+\theta &  \mbox{if $x<-\theta$}.
\end{array}\right.
\]

\vskip 5pt
\textbf{Adjusted Residuals}: Using the previous estimate $x^{t-1}$ and a current estimate $x^t$,  compute the adjusted residuals $z^t\in\mathbb{R}^n$
\begin{equation}\label{alg:alg3}
z^t =  Y-  A x^t+\frac{1}{\delta}G(z^{t-1};b_{t-1}) \left\langle \partial_1\eta \bigl(x^{t-1} +  
A^TG(z^{t-1};b_{t-1});\theta_{t-1}\bigl) \right\rangle.
\end{equation}
 We add a rescaled product to the ordinary residuals $Y-A x^t$, that explicitly depends on $n$, $p$ and $s$.   This step can be
recognized as proximal gradient descent \citep{BT09} in the   variable $x$ of the function $\rho$
using the step size  $\left\langle \partial_1\eta \bigl(x^{t-1} +  
A^TG(z^{t-1};b_{t-1});\theta_{t-1}\bigl) \right\rangle/\omega$.
\vskip 5pt
\textbf{Effective Score}: Choose the scalar $b_t$ from the following equation, such that the empirical average of the effective score $\Phi(z ;b )$ has the slope $s/n$,

\begin{equation}\label{alg:alg2}
\frac{\omega}{\delta}= \frac{1}{n}\sum_{i=1}^{n}\partial_1 \Phi(z^t_i;b_t).
\end{equation}

As $n/s >1$, for differentiable losses $\rho$ previous equation has at least one solution, as $ \Phi(z ;b )$ is continuous in $b$ 
and takes values of both $0$ and $\infty$. Whenever, $\partial_1 \Phi$ is not continuous, the solution can be 
  defined uniquely in the form
\[
b_t = \frac{1}{2} (b_t^+ + b_t^-)
\]
where 
$b_t^+ = \sup \{d >0:   \frac{1}{n}\sum_{i=1}^{n}\partial_1 \Phi(z^t_i;d) > \frac{\omega}{\delta} \}$
and 
$b_t^- = \sup \{d <0:   \frac{1}{n}\sum_{i=1}^{n}\partial_1 \Phi(z^t_i;d) < \frac{\omega}{\delta} \}$.
For non-differentiable losses $\rho$, we consider two adaptations. First, we allow parameter $b_t$, which controls  the amount of $\min$ regularization of the robust loss $\rho$ function, to be adaptive with each iteration $t$. Second, we consider a population equivalent of the \eqref{alg:alg2} first,   then  design an estimator of it and solve the fixed point equation.
In more details,

\begin{equation}\label{alg:alg2a}
\frac{\omega}{\delta}=  \hat \nu (b_t),
\end{equation}
for  a consistent estimator $\hat \nu = \hat \nu (b_t)$ of a population parameter $\nu$ defined as 
\[
 \nu (b_t)=
\partial_1\EE \left[ \Phi(z^t;b_t) \right].
\]
The advantage of this method is to avoid numerical challenges arising from solving a fixed point equation of a  non-continuous function. A particular form of $\hat \nu$ depends on the choice of the loss function $\rho$ and the density of the error term $f_W$. We discuss examples in the Section~\ref{sec:2}.


\textbf{Estimation}: Using the regularization parameter $b_t$ determined by the previous step,   update the estimate $x^t$ as follows,
\begin{equation}\label{alg:alg1}
x^{t+1}=\eta\left(x^t+A^TG(z^t;b_t);\theta_t \right),
\end{equation}
with the soft thresholding function $\eta$.

The estimation step of the algorithm   introduces the necessary thresholding step needed for inducing sparsity in the estimator \citep{BM12}. However, in contrast to the existing methods it is adjusted with the appropriately scaled, min regularized robust score function $ \delta \Phi/ \omega$. The three--step estimation method of RAMP is no longer a simple  proxy for the one-step M estimation. Due to high dimensionality with $p \geq n$, such a step is no longer adequate. Instead, we work with its soft thresholded alternative to ensure approximate sparsity of each iterate. Furthermore,  the residuals require additional scaling, i.e., we multiply the  scaling  of \cite{DM13} with a factor proportional to the fraction of sparse elements of the current iterate, in other words,   $\left\langle \partial_1\eta \bigl(x^{t-1} + A^T \frac{\delta}{\omega} \Phi (z^{t-1};b_{t-1});\theta_{t-1}\bigl) \right\rangle$ (see Lemma \ref{lemma:lemma1} below). 
Rescaling of $\delta/\omega$  in the above term  is absolutely necessary, is an effect of the regularization $b_{t-1}$ and will be absorbed by $\theta_{t-1}$.
This rescaling  is needed to prove connection with the general AMP algorithms of \cite{BM11}.
In  existing AMP algorithms,  this scaling  does not appear as a special case  of least squares loss for which it gets canceled with a constant in $\Phi$.  
 
 In the following we present a few examples of RAMP algorithm for different choices of the loss function $\rho$.
 
\begin{example}
[Least squares Loss (continued)]
Notice that  $\Phi'(z;b)=\frac{b}{1+b}$. Then, \eqref{alg:alg2} returns  the value of the $\min$-regularizer  $b$
of  $\frac{s}{n-s}$. For the particualar choice of $b$, we see that the rescaled effective score $G(z;b) =z$, in which case the RAMP algorithm is equivalent to the AMP of \cite{BM12}.
\end{example}

 \begin{example}
[Huber Loss (continued)] 

Following the definition of $\Phi(z;b,\tau)$ we obtain
$$ \EE  \Phi(z;b,\tau) = \frac{b}{1+b}  \EE \left[ z \mathbbm{1}(|z| \leq (1+b)\gamma  )\right]  + b \gamma  \PP(z > (1+b) \gamma)  -b\gamma \PP(z<  -(1+b)\gamma).$$  Additionally, Condition $(\Rb)$, guarantees that  $\nu = \partial_1 \EE  \Phi(z;b,\gamma)$. 
Therefore, 
\[
\nu
=  \frac{b}{1+b} F_z((1+b)\gamma ) - \frac{b}{1+b} F_z(-(1+b)\gamma )
-b \gamma f_z((1+b)\gamma ) -b\gamma f_z(-(1+b)\gamma ), 
\]
for $F_z, f_z$ denoting the distribution and density functions of the residuals $z$. Given a  sample of adjusted residuals $z_1^t,\dots, z_n^t$, provided by \eqref{alg:alg3}  at any iteration $t$, we can easily formulate an empirical distribution function $ \hat F _z^t$ and a density estimator $\hat f_z^t$, using any of the standard non-parametric tools.  Then, for any fixed $ \gamma >0$, $b_t$ is a solution to an implicit function equation \eqref{alg:alg2a}
\[
s/n =   \frac{b}{1+b}  \hat F _z^t((1+b)\gamma ) - \frac{b}{1+b}  \hat F _z^t(-(1+b)\gamma )
-b \gamma \hat f_z^t((1+b)\gamma ) -b\gamma \hat f_z^t(-(1+b)\gamma ).
\]

%
%
%

\end{example}
 
\begin{example}
[Absolute Deviation Loss (continued)] 

Since $ \EE  \Phi(z;b) =   \EE \left[ z \mathbbm{1}(|z| \leq b)\right] + b \PP(|z| >b)  $, Condition $(\Rb)$, guarantees that     $\nu = \partial_1 \EE  \Phi(z;b ) $
\[
\nu
=  F_z(b) - F_z(-b)
-bf_z(b) +bf_z(-b),
\]
for $F_z, f_z$ denoting the distribution and density functions of $z$. Given  a set of adjusted residuals $z_1^t,\dots, z_n^t$, provided by \eqref{alg:alg3}  at any iteration $t$, 
$b_t$ is a solution to an implicit function equation \eqref{alg:alg2a}
\[
s/n =  \hat F_z^t(b) - \hat F_z^t(-b)
-b \hat f_z^t(b) +b \hat f_z^t(-b).
\]

\end{example}

\begin{example}
[Quantile Loss (continued)] 

For the case of the quantile loss $\nu =  \EE\partial_1  \Phi(z;b,\tau)$.  Adding Condition {\rm ({\bf R})} to the setup, we obtain $\nu = \partial_1 \EE  \Phi(z;b,\tau)$. Narrowing the focus to $\EE  \Phi(z;b,\tau) $ we obtain
 $ \EE  \Phi(z;b,\tau) =   \EE \left[ z \mathbbm{1}(z \leq b \tau)\right] + \EE \left[ z \mathbbm{1}(z \geq b (\tau-1))\right]  + b \tau \PP(z >b\tau)   +b (\tau-1) \PP(z< b(\tau-1))$. Now, refining the equation for $\nu$ we obtain
\[
\nu
=  F_z(b \tau) - F_z(b(\tau-1))
-b\tau f_z(b \tau) +b(\tau-1) f_z(b(\tau-1)), 
\]
for $F_z, f_z$ denoting the distribution and density functions of $z$. Given  a set of adjusted residuals $z_1^t,\dots, z_n^t$, provided by \eqref{alg:alg3}  at any iteration $t$, 
and  a fixed $\tau \in(0,1)$, $b_t$ is a solution to an implicit function equation \eqref{alg:alg2a}
\[
s/n =  \hat F_z^t(b\tau ) - \hat F_z^t\bigl(b(\tau-1)\bigl)
-b\tau \hat f_z^t(b\tau) +b(\tau-1) \hat f_z^t\bigl(b(\tau-1)\bigl).
\]
\end{example}

In practice,
$\hat F_z^t(b\tau )$
typically takes the form of  an empirical cumulative distribution function
$
\hat F_z^t(b\tau )=\frac{1}{n} \sum_{i=1}^n \mathbbm{1}\{z_i^t \leq b \tau\}.
$
In contrast, there are numerous consistent estimators of $f_z(b \tau)$. For instance, by the asymptotic linearity results of Lemma \ref{lem:uniform-rho}, we consider
\begin{equation}\label{eq:nuhat}
\frac{1}{h\sqrt{n}} \sum_{i=1}^n \left[ \Phi(z_i^t+n^{-1/2}h;b\tau) - \Phi(z_i^t-n^{-1/2}h;b\tau) \right],
\end{equation}
for a bandwidth parameter $h>0$.
In practice, it is difficult to obtain estimators $\hat F_z^t(b\tau )$ and $\hat f_z^t(b\tau )$ that are continuous functions of $b$.
Hence, to solve the fixed point equations we implement a simple grid search and set $b$ to be the average of the  the first value of $b$ on the grid for which the estimated function  is bellow $s/n$ and the the last value of $b$ on the grid for which the estimated function is above $s/n$.

\section{Theoretical Considerations}\label{sec:theory}

In this section, we offer   theoretical analysis and  prove how is RAMP related to the $l_1$ penalized M-estimators and show the convergence property of the RAMP estimator.

\subsection{Relationship to penalized M-estimation}\label{sec:2}
The last term $\langle \partial_1\eta(A^TG(z^{t-1};b_{t-1})+x^{t-1};\theta)\rangle$
 in step 1 of RAMP iteration (equation (\ref{alg:alg3})) is  a correction of the residual, called Onsager reaction term. This term is generated from the theory of belief propagation in factor graphical models and the procedure of generation is shown in \cite{DMM11}. Adding the Onsager reaction term in each iteration is the main difference from AMP iteration and soft thresholding iteration. The intuition of this term in each step is considering undersampling and sparsity simultaneously. The following Lemma \ref{lemma:lemma1} shows the relationship between the Onsager reaction term and in Donoho's \citep{DMM09} term the undersampling--sparsity.

 \begin{lemma}\label{lemma:lemma1}
 
Let  $(z_*,b_*,\hat x_*)$ be a fixed point of the RAMP equations \eqref{alg:alg3}, \eqref{alg:alg2} and \eqref{alg:alg1}   having $b_*>0$.
According to the definition of $\eta(x)$, the   correction term $\langle \partial_1\eta(A^TG(z;b)+x;\theta)\rangle$ evaluated at the fixed point $(z_*,b_*,\hat x_*)$ is equal to ${\| \hat x_*\|_0}/{p}$, i.e.,to $\omega$.
\end{lemma}

The following lemma shows the reason behind the  use of the effective score $\Phi(z;b)$ in the RAMP algorithm -- it connects the RAMP iteration with the penalized M-estimation. The penalized M-estimator, which is the optimal solution $\hat{x}(\lambda)$ of problem (\ref{eqDef}), satisfies the KKT condition:
\begin{equation*}\label{kkt}
  -\sum_{i=1}^n\rho'(Y_i-A_i^T\hat x(\lambda)) A_i - \lambda v = -X^T\rho'(Y-A \hat x (\lambda)) -\lambda v = 0
\end{equation*}
where $\rho'$ is applied component-wise. We will show in the following lemma that the estimator in the RAMP iteration with proper thresholding level  $\theta$ also satisfies the KKT condition  above  with the help of the rescaled, effective score function $G(z;b)$.

\begin{lemma}\label{the:thethe}

Let  $(z_*,b_*,\hat x_*)$ be a fixed point of the RAMP equations \eqref{alg:alg3}, \eqref{alg:alg2} and \eqref{alg:alg1},   having $b_*>0$. Then, $\hat x_*$ is a solution to the penalized M-estimator problem \eqref{eqDef} with  $ \lambda  =\frac{ \theta_* \omega }{b_*\delta}$. Vice versa, any minimizer $\hat x(\lambda)$ of the problem \eqref{eqDef} corresponds to one (or more) RAMP fixed points of the form $\left(z_*,\frac{ \theta_* \omega }{\lambda \delta},\hat x_* \right)$.
\end{lemma}

From the lemma above, we offer a relationship of tuning parameter $\lambda$ in penalized M-estimation with the threshold  parameter $\theta$ of  the RAMP iteration:
\[
\lambda = \frac{\theta\omega}{b\delta}.
\]

\subsection{State Evolution of RAMP}

State Evolution (SE) formalism introduced in  \cite{DMM10} and \cite{DMM09} is used to predict the dynamical behavior of numerous observables of the approximate message passing algorithms. In SE formalism, the asymptotic distribution of the residual and the asymptotic performance of the estimator can be measured while allowing $p \to \infty$. The parameter $\bar{\tau}_t^2$  can be considered as the state of the algorithm and it predicts whether the algorithm  converges or not. In more details,  the asymptotic mean squared error (AMSE), defined as 
\[
\mbox{AMSE} = \underset{p\to\infty}{\operatorname{lim }}\frac{1}{p}\sum_{i=1}^p(x_i^t-x_{0,i})^2,\
\]
 is a function of a state evolution parameter $\bar{\tau}_t^2. $
 We will show  that the proposed RAMP algorithm, which contains three steps,  belongs to a very general class of message passing algorithms. We will offer how to compute  $\bar{\tau}_t$, through a novel iteration scheme  that is adjusted for $p \geq n$
 and robust,  not necessarily differentiable losses $\rho$.

\begin{lemma}\label{lem:lem22}
Let Conditions $(${\bf R}$)$, $(${\bf D}$)$ and $(${\bf A}$)$ hold. Then,
the RAMP algorithm defined by the equations (\ref{alg:alg3}), (\ref{alg:alg2}) and (\ref{alg:alg1}) belongs to the general recursion of \cite{BM11}. 
Let  $\bar{\sigma}^2_0 = \frac{1}{\delta}\EE X^2_0$ and let  $x_0$ and  $W$ follow   density $p_{x_0}$ and $f_W$ respectively, where $\EE W^2=\sigma^2$. Let $Z$ be a standard normal random variable. Then, for all $t\geq0$ the state evolution sequence $\{\bar{\tau}^2_t\}_{t\geq0}$ of the RAMP algorithm is obtained by the following   iterative system of equations:
\begin{equation}\label{tau}
\bar{\tau}_{t}^2=\mathbb{E} \left[G(W+\bar{\sigma}_tZ;b_t)\right]^2,
\end{equation}
where
\begin{equation}\label{sigma}
\bar{\sigma}_t^2 = \frac{1}{\delta} \mathbb{E}\left[\eta(x_0+\bar{\tau}_{t-1}Z,\theta)-x_0)\right]^2.
\end{equation}
\end{lemma}

Notice that the function of $\bar{\sigma}_t$ and $\bar{\tau}_t$ depends on the distribution of true signal $p_{x_0}$, error distribution $F_W$ and a loss function; however, $\bar{\tau}_t$ and  $\bar{\sigma}_t$  do  not  depend on the design matrix $\Ab$. Therefore,   we believe that  the assumptions of the Gaussian design can be released.  

In more details, define the sequence $\bar \tau_t^2 $ by setting $\bar \sigma_0^2 = \frac{1}{\delta} \EE[x_0^2]$ for $x_0 \sim p_{x_0}$ and with it $\bar \tau_0^2 =  \frac{\delta^2}{\omega^2} \EE [\Phi(W - \bar \sigma_0 Z; b(\bar \tau_0^2))]$;   let $\bar \tau_t^2$ be defined as the solution to the iterative equations \eqref{tau} and \eqref{sigma}, i.e.
\[
\bar \tau_{t+1}^2 = \mathbb{V}(\bar \tau_t^2, b(\bar \tau_t^2),\theta (\bar \tau_t^2))
\]
for 
\[
\mathbb{V}(  \tau ^2, b ,\theta  ) = \frac{\delta^2}{\omega^2} \EE [\Phi(W +\sigma  Z; b )], \qquad 
  {\sigma}^2  = \frac{1}{\delta} \mathbb{E}\left[\eta(x_0+\tau Z,\theta)-x_0)\right]^2.
\]

\begin{lemma}\label{lem:uniquetau}
Let $\rho$ be a convex function and let  Conditions $(${\bf R}$)$, $(${\bf D}$)$ and $(${\bf A}$)$ hold. 
For any $\sigma^2 >0$ and $\alpha > \alpha_{\min}$, the fixed point equation 
\[
  \tau ^2 = \mathbb{V}\left(  \tau ^2, b(\tau^2),\alpha \tau\right)
\]
admits a unique solution $\tau^*=\tau^*(\alpha)$ for all smooth loss functions $\rho$. Moreover, 
$\lim_{t \to \infty} \tau_t = \tau^*(\alpha)$. Further, the convergence takes   place at any initial solution and is monotone.
Additionaly, for all non-smooth loss functions the fixed point equation above, admits multiple solutions $\tau^*=\tau^*(\alpha)$.  In such cases,  the convergence take place but it  depends on the initial solution and   is monotone for each initialization.
\end{lemma}

 The display above offers an explicit expression of how the additional Gaussian variable $Z$ effects  the fixed points $\tau^*$ and $\sigma^*$ and the  sequence $\{x^t\}$. In the case of a simple Lasso estimator, with $G$ being a rescaled least squares loss, $\bar \tau_t^2$ becomes $\sigma^2 + \delta^{-1}\mathbb{E} \left[\eta(x_0+\bar{\tau}_{t-1}Z,\theta)-x_0\right]^2$, similar to the result of \cite{BM11}. The paper \cite{DM13} considers non-penalized $M$ estimates with strongly convex loss functions -- this excludes   Least Absolute Deviation and Quantile loss, in particular. We provide  further details of the behavior of the fixed point $\tau^*$ in Section \ref{sec:examples} for the cases of non-differentiable loss functions. Moreover, we  relate  the properties of $\sigma^*$ to the  relative efficiency of $l_1$-penalized M-estimators in Section \ref{sec:RE}.

Next, we show  that at each iteration $t$, $x^t+A^TG(z^t;b_t)$ has the same distribution as $x_0+\bar{\tau}_{t-1}Z$. This enables us to provide the characterization of the {\it effective slope} of the algorithm. It measures the value of the $\min$-regularization parameter $b$, which satisfies the population analog of the Step 2 of the RAMP algorithm.

\begin{lemma}\label{lem:delta}
Let Conditions Let Conditions $(${\bf R}$)$, $(${\bf D}$)$ and $(${\bf A}$)$ hold. Let   $\bar \sigma$ be   a  stationary point  of the recursion \eqref{tau}-\eqref{sigma}  of the RAMP algorithm defined by equations (\ref{alg:alg3}), (\ref{alg:alg2}) and (\ref{alg:alg1}).  For all twice differentiable losses $\rho$,
\[
\frac{\omega}{\delta} = \EE \left[ \partial_1 \Phi \left( W+ \bar \sigma Z; b\right) \right],
\]
where $W$ and $Z$ have $F_W$ and $\mathcal{N}(0,1)$ distributions, respectively. 
Let $f_{C -\Phi(C;b)}$ denote the density of the random variable $C -\Phi(C;b)$ for $C=W - \bar \sigma Z$.  
Let the bandwidth, $h$, for the consistent  estimator $\hat \nu$, \eqref{eq:nuhat}, be such that $h \to 0$ and $n h \to \infty$. Then, for the non-necessarily differentiable losses $\rho$, 
\[
\frac{\omega}{\delta} = \EE \left[ \partial_1 v_1\left( W+ \bar \sigma Z; b\right) \right] + \EE \left[ \partial_1 v_2\left( W+ \bar \sigma Z; b\right) \right] + \sum_{\nu=1}^{k-1} \alpha_\nu b \left(f_{C -\Phi(C;b)} (r_{\nu+1})  - f_{C -\Phi(C;b)} (r_{\nu}) \right),
\]
where $v_1,v_2$ are defined in Condition~$\mathbf{(R)}$.
\end{lemma}

\subsection{Asymptotic mean squared error } \label{sec:amse}

In this section, we relate the state evolution properties of $\bar \tau_t$ and  $\bar \sigma_t$
with a    distance measure of $x_i^t$ and $x_0$.  Similarly to the existing literature on approximate message passing, the measure of distance is done through  a  pseudo-Lipschitz function $\psi$.   We say a function  $\psi:\RR^2 \to \RR$ is pseudo-Lipschitz if there exist a constant $L >0$ such that for
all $x,y \in \RR^2$: $|\psi(x) - \psi(y)|\leq L (1+\|x\|_2 + \| y \|_2) \| x-y\|_2$.

\begin{theorem}\label{the:the2}
Let   Conditions $(${\bf R}$)$, $(${\bf D}$)$ and $(${\bf A}$)$ hold and let $\psi: R \times R \rightarrow R$  be a pseudo-Lipschitz function.  Let $\{x^t \}_{t \geq 0}$ be a sequence of RAMP estimates, indexed by the iteration number $t$. Then, almost surely
\[
\underset{p\to\infty}{\operatorname{lim }}\frac{1}{p}\sum_{i=1}^{p}\psi(x^{t+1}_i,x_{0,i})= \EE [\psi(\eta(x_0+\bar{\tau}_tZ;\theta),x_0)],
\]
for all $\bar \tau _t$ and $\bar \sigma _t$ defined by the recursion \eqref{tau}-\eqref{sigma}.
\end{theorem}
Choosing $\psi(x,y) =(x-y)^2,$ we have the AMSE map, which can predict the success of recovering signals: 
\begin{equation}
\label{AMSE}
\mbox{AMSE}(x^t,x_0)=\underset{p\to\infty}{\operatorname{lim }}\frac{1}{p}\sum_{i=1}^{p}(x_i^t-x_{0,i})^2=\EE [\eta(x_0-\bar{\tau}^*_t Z;\theta)-x_0]^2.
\end{equation}
The display above presents the asymptotic mean squared error of the sequence of solutions to the RAMP algorithm. 
Next we connect this sequence of the RAMP algorithm    to the $l_1$-penalized M-estimator \eqref{eqDef}.  
We demonstrated that the estimator of RAMP is one of the optimal solution in  Lemma \ref{the:thethe}. In turn, we  measure the distance between the RAMP iteration and the penalized estimator. We use $L_2$ norm as the measurement of distance. 
\begin{theorem}\label{the:the1}
Let Conditions   $(${\bf R}$)$, $(${\bf D}$)$ and $(${\bf A}$)$ hold.
Let $\hat{x}$ be the penalized M-estimator  and  let $\{x^t\}$  be the sequence of estimates produced by the RAMP algorithm. Then, 
\[
\underset{t\to\infty}{\operatorname{lim }}\underset{p\to\infty}{\operatorname{lim }}\frac{1}{p}\|x^{t}-\hat{x}(\lambda)\|_2^2= 0,
\]
for all $\lambda>0$   for which $\|\hat{x}(\lambda)\|_2^2 /p$ is finite.
\end{theorem}

Based on Theorem $\ref{the:the1}$, we can further prove the following theorem to show the distance of penalized M-estimator $\hat{x}$ and the true parameter $x_0$. 
\begin{theorem}\label{thm:thm10}
Let Conditions $(${\bf R}$)$, $(${\bf D}$)$ and $(${\bf A}$)$ hold. Let $\hat{x}$ be the penalized M-estimator. Let $\psi: \mathbb{R} \times \mathbb{R} \to \mathbb{R}$ be a pseudo-Lipschitz function. Then,  
\[
\underset{p\to\infty}{\operatorname{lim }}\frac{1}{p}\sum_{i=1}^{p}\psi(\hat{x}_i(\lambda),x_{0,i})= \EE [\psi(\eta(x_0+\bar{\tau}^*Z;\theta),x_0)],
\]
for all $\lambda$ for which   $\| \hat{x}(\lambda)\|_2^2/p$ is finite and for $\tau^*$   a stationary point of the recursion \eqref{tau}-\eqref{sigma}.
\end{theorem}
The right hand sides of Theorem \ref{thm:thm10} and Theorem \ref{the:the1} are equal. This guarantees that the AMSE$(x^t,x_0)$ and the AMSE$(\hat{x},x_0)$ are asymptotically the same, even if expressions for most of the loss functions are too complex to simplify. This  offers   not only an upper bound on AMSE$(\hat{x},x_0)$, but also an exact expression of it. 

\section{Relative Efficiency} \label{sec:RE}

The robustness properties of sparse, high-dimensional  estimators are difficult to quantify due to  shrinkage effects and subsequent  bias in estimation.  Whenever efficiency is defined though asymptotic variance, shrinkage is known to lead to super-efficiency phenomena. Relative efficiency can capture both the size of the bias and the variance together leading to a relevant robustness evaluation. We can say that one estimator    dominates  the the other, if its asymptotic mean squared error is smaller.
 
 State evolution of the RAMP algorithm provides a useful iterative scheme for computing the value of the Asymptotic Mean Squared Error.  According to Theorem  \ref{thm:thm10}, the asymptotic mean squared error of penalized $M$-estimators is 
\begin{equation*}\label{sigmaa}
 \mbox{AMSE} ({\bar \tau}_{t+1}^2,b({\bar \tau}_{t+1}^2),\theta_{t+1})=\delta \mathbb{E}\biggl[\eta\left(x_0+\bar{\tau}_{t}Z, \lambda\frac{\delta}{\omega}b_t \right)-x_0\biggl]^2,
\end{equation*}
  with the expectation taken with respect to $x_0$ and $Z$ and 
 \begin{align}\label{eq:tau-tau}
 \bar{\tau}_{t} ^2
=
 \frac{   {\mathbb{E} \biggl[\Phi ^2\left(W +\bar \sigma_{t } Z ;b_t\right)\biggl]}}{ \biggl[ \EE \left[ \partial_1  v \left( W+ \bar \sigma_t Z; b_t\right) \right] + \sum_{\nu=1}^{k-1} \alpha_\nu b_t \left(f_{C_t -\Phi(C_t;b_t)} (r_{\nu+1})  - f_{C_t -\Phi(C_t;b_t)} (r_{\nu}) \right) \biggl]^2},
\end{align}
where  $v$ denotes the continuous part of  $\rho'$, i.e., $v=v_1+v_2$; moreover, $f_{C_t -\Phi(C_t;b_t)}$ denotes the density of the random variable $C_t -\Phi(C_t;b_t)$ with $C_t=W - \bar \sigma_t Z$.  
Hence the high dimensional  asymptotic mean squared error mapping, allowing $p \geq n$,  is a sequence  $\{  \mathbb{AMSE} ({\bar \tau}_t^2,b({\bar \tau}_t^2),\theta_t)\}_{t \geq 0}$  produced by the above iterative scheme.

 Observe that $1/\delta = O(1/n)$ for all $p \leq n$ and $p$   does not grow  with $n$. In this setting, $\eta$ is the identity function and $\bar \sigma_t^2 = \bar \tau^2_{t-1}$ for all twice-differentiable losses $\rho$. Also,  in this setting,  the  asymptotic mean squared error  mapping above takes the form of variance mapping presented in  \cite{DM13}, after observing that the bias in estimation disappears. Specifically, when $p=o(n)$, we recover the result of  the above mentioned paper and identify the additional Gaussian component in the variance mapping.
 
 Cases  of $p \geq n$, are significantly  more complicated. We see that $Z$ component will never disappear as $1/\delta = p/n \geq 1$.  Moreover, bias in estimation will not disappear asymptotically. This indicates that studies of efficiency in high dimensions with $p \geq n$  never converge to the low-dimensional case, as was previously believed.  
 Even when the true model is truly sparse with $s \ll n$, the additional $Z$ component does not disappear; it has a substantial role in both the size of the asymptotic variance and the asymptotic bias.

\begin{theorem} \label{thm:information}
Suppose that $W$ has a well-defined Fisher information matrix $I(F_W)$.  
Let $\tau_t$ and $\sigma_t$ be the state evolution parameters  following equations \eqref{tau} and \eqref{sigma}, respectively.
  Then,  under conditions $(${\bf R}$)$, $(${\bf D}$)$ and $(${\bf A}$)$ 
\begin{enumerate}
\item[(i)] for every iteration $t$ of the RAMP algorithm  \eqref{alg:alg3}, \eqref{alg:alg2}, \eqref{alg:alg1}, state variable $\tau_t$ satisfies
\[
\tau_t^2\geq \frac{\omega}{\delta}  \frac{1 + \sigma_t^2 I(F_W) }{I( F_W ) },
\]
\item[(ii)] for the stationary solution $(\tau^*,\sigma^*)$ of the RAMP algorithm with all $\alpha \geq \alpha_{\min} >0$
\[
{\tau^*}^2\geq   \frac{s}{n-s}\frac{ 1}{I( F_W ) },
\]
\item[(iii)]   for fixed values of $\alpha$ and $x_0$, with $\theta=\alpha \tau$, there exist  functions  $\nu_1$, $\nu_2$ that are convex and increasing, respectively, and are such that the asymptotic mean squared error  mapping for high dimensional problems satisfies:
\[
\mbox{\rm AMSE} ({\tau^*}^2,b({\tau^*}^2),\alpha \tau^*)=  \nu_1(\tau) \tau^2  + \nu_2(\tau),
\]
with
\begin{align*}
\nu_1(\tau) = 1+ \alpha^2 &- \EE_{x_0}\left[\alpha^2 \left( \Phi(\alpha -   \frac{x_0}{\tau})  - \Phi(-\alpha -   \frac{x_0}{\tau})\right) \right.\\
 & \left. -   \left( \alpha +  \frac{x_0}{\tau}\right) \phi(\alpha -   \frac{x_0}{\tau})   - \left( \alpha -  \frac{x_0}{\tau}\right) \phi(-\alpha -   \frac{x_0}{\tau}) \right]
 \end{align*}
and
$\nu_2(\tau) =  \EE_{x_0} \left[ x_0^2 \left(  \Phi(\alpha -   \frac{x_0}{\tau})  - \Phi(-\alpha -   \frac{x_0}{\tau}) \right) \right]
$.

\end{enumerate}
\end{theorem}
  
 Recall that traditional lower bound of $M$-estimators with $p \leq n$ is $1/I(F_W)$ and is such that asymptotic mean squares error is equal to the variance and   is achievable for fixed $p$ and $n\to \infty$ asymptotics. From the display above, we observe that under diverging $p$ and $s$ and  $n$,  such that $p \gg n \geq s$, traditional lower  bound is not achievable for all $s \geq n/2$, i.e., for all   "dense" high dimensional problems. Hence, we observe a new phase transition regarding robustness in high dimensional and sparse problems. In the inequality above, the effect of sparsity is extremely clear. If the problem is significantly sparse, with $n /s < \infty$, then the traditional information  bound may be achieved, whereas for all  other problems  the traditional  information bound cannot be achieved, as there is inflation in the  variance.\\

{ \bf{Relative Efficiency of Penalized Least Squares

 and Penalized Absolute Deviations}}\\
 
 Next, we study the relative efficiency of the penalized least squares (P-LS from hereon)  estimator, with respect to the penalized least absolute deviation (P-LAD from hereon) estimator. From the results above, we can clearly compute the asymptotic mean squared error of the penalized methods as the recursive equations 
 \begin{align}
 \tau_{\mbox{\tiny P-LS}}^2 &= \sigma_W^2 + \sigma_{\mbox{\tiny P-LS}}^2,\\
  \tau^2_{\mbox{\tiny P-LAD}} &= \frac{\EE \left[ \left( W + \sigma_{\mbox{\tiny P-LAD}} Z\right)^2 \mathbbm{1} \left\{ |W + \sigma_{\mbox{\tiny P-LAD}} Z| \leq  b\right\}\right] + b^2 \PP \left( |W + \sigma_{\mbox{\tiny P-LAD}} Z| > b\right) }{\PP^2 \left( |W + \sigma_{\mbox{\tiny P-LAD}} Z| \leq  b\right)}.
 \end{align}
 In the above display, both $\sigma_{\mbox{\tiny P-LAD}}$ and $\sigma_{\mbox{\tiny P-LS}}$ satisfy the equation of \eqref{sigma} with $\tau_{\mbox{\tiny P-LAD}}$ and $\tau_{\mbox{\tiny P-LS}}$, respectively.

Notice that in, sparse, high dimensional setting,  the distribution of the $x_0$ can be represented as a convex combination of the Dirac measure at 0 and a measure that doesn't have mass at zero. Let us denote with $\Delta$ and $U$ two random variables, each having the two measures above. Then, the asymptotic mean squared error satisfies 
\[
\delta \sigma^2 = \delta \tau^2 \left( (1- {\omega} ) (\EE_{Z} \eta(Z,\alpha) )^2 + \omega \EE_{(U,Z)} \left[ \eta\left(\frac{U}{\tau}+Z;\alpha\right)- \frac{U}{\tau}\right]^2\right).
\]
 We will explore this representation to study the relative efficiency of P-LS and P-LAD estimators. 
The relative efficiency of  P-LS w.r.t. P-LAD is defined as the quotient of their asymptotic mean squared errors. 
By results of previous sections, this amounts to the quotient of $\sigma_{\mbox{\tiny P-LS}}^2/\sigma_{\mbox{\tiny P-LAD}}^2$. To evaluate this quotient, we study the behavior of $\sigma_{\mbox{\tiny P-LS}}^2/ \sigma_{\mbox{\tiny $W$}}^2 $ and $\sigma_{\mbox{\tiny P-LAD}}^2/ \sigma_{\mbox{\tiny $W$}}^2 $ independently. In order to do so, we need a preparatory lemma below.

\begin{lemma}\label{lem:prep1}
Let  Conditions $(${\bf R}$)$, $(${\bf D}$)$ and $(${\bf A}$)$  hold.
Let $\bar\sigma_{\mbox{\tiny P-LAD}}^2$   be a fixed point solution to the state-evolution system of equations \eqref{tau} and \eqref{sigma}, with a loss  $\rho(x)=|x|$ .
Let $\sigma_{\mbox{\tiny $W$}}^2 $ be a variance of the error term $W$ \eqref{w}.    Then, $\tau_{\mbox{\tiny P-LAD}}^2 \to 0$ and $\sigma_{\mbox{\tiny P-LAD}}^2 \to 0$, whenever $\sigma_{\mbox{\tiny $W$}}^2 \to 0$ and $\tau_{\mbox{\tiny P-LAD}}^2 \to \infty$ and $\sigma_{\mbox{\tiny P-LAD}}^2 \to \infty$, whenever $\sigma_{\mbox{\tiny $W$}}^2 \to \infty$.
\end{lemma}

 Next, we consider a class of distributions $f_W$ such that $\sigma_W^2$ exists  and consider state variable $\sigma_{\mbox{\tiny P-LAD}}^2$ as a function of $\sigma_W^2$. We provide limiting behavior of both P-LS and P-LAD in cases where $\sigma_W^2 \to 0$, that is the case of ``light tailed distributions."

\begin{lemma}\label{lem:Plad}
Let  Conditions $(${\bf R}$)$, $(${\bf D}$)$ and $(${\bf A}$)$  hold.
Let $\bar\sigma_{\mbox{\tiny P-LAD}}^2$ and $\bar\sigma_{\mbox{\tiny P-LS}}^2$  be a fixed point solution to the state-evolution system of equations \eqref{tau} and \eqref{sigma} with a loss  $\rho(x)=|x|$ and a loss  $\rho(x)=(x)^2$, respectively. 
Let $\sigma_{\mbox{\tiny $W$}}^2 $ be a variance of the error term $W$ \eqref{w}. In turn,  if $M(\omega)<\delta$,
\[
\lim _{ \sigma_{\mbox{\tiny $W$}}^2 \to 0} \frac{\bar\sigma_{\mbox{\tiny P-LS}}^2}{\sigma_{\mbox{\tiny $W$}}^2}  \to \frac{1}{1-M(\omega)/\delta}, \qquad \lim _{ \sigma_{\mbox{\tiny $W$}}^2 \to 0} \frac{\bar\sigma_{\mbox{\tiny P-LAD}}^2}{\sigma_{\mbox{\tiny $W$}}^2}  \to \infty,
\]
with
$
M(\omega) = \inf_\tau \left\{ (1-\omega) \EE \eta^2(Z;\tau) + \omega \sup_{\mu \geq 0} \EE \left( \eta(\mu +Z;\tau) - \mu\right)^2 \right\}
$, where $\EE$ is with respect to $Z$.
\end{lemma}

A recent work \cite{Z15} proved that $M(\omega)/\delta \geq \omega/\delta$. Together with  the results of Lemma \ref{lem:Plad}, we can see that the P-LAD method is less efficient than the P-LS method for all of  $\omega<\delta$. In other situations where $\omega \to \delta$, both limits  on the right hand side of Lemma \ref{lem:Plad} are infinity and the two methods are inseparable. 
Classical Huber's results state that  the LS method is more efficient than the LAD method  only for the class of Normal distributions. However, with high dimensional asymptotic, where $s \to n$ we do not see this pattern. The result above identifies new the breakdown point, where $M(\omega)=\delta$, that is,
\[
\sup \left\{\omega:  \inf_\tau \Bigl [ (1-\omega) \EE \eta^2(Z;\tau) + \omega \sup_{\mu \geq 0} \EE \left( \eta(\mu +Z;\tau) - \mu\right)^2 \Bigl] < \delta\right\}.
\]
The implication is that when the sparsity $s$ approaches $n$ the P-LAD and P-LS method have efficiency of the same order.
Next, we provide limiting behavior of both P-LS and P-LAD in cases where $\sigma_W^2 \to \infty$; that is, in the case of ``heavy tailed distributions."

\begin{lemma}\label{lem:Plad2}
Let  Conditions $(${\bf R}$)$, $(${\bf D}$)$ and $(${\bf A}$)$  hold.
Let $\bar\sigma_{\mbox{\tiny P-LAD}}^2$ and $\bar\sigma_{\mbox{\tiny P-LS}}^2$  be a fixed point solution to the state-evolution system of equations \eqref{tau} and \eqref{sigma} with a loss  $\rho(x)=|x|$ and a loss  $\rho(x)=(x)^2$, respectively. 
Let $\sigma_{\mbox{\tiny $W$}}^2 $ be a variance of the error term $W$ \eqref{w}. Then,  if $\Gamma <\delta$,
\[
\lim _{ \sigma_{\mbox{\tiny $W$}}^2 \to \infty} \frac{\bar\sigma_{\mbox{\tiny P-LS}}^2}{\sigma_{\mbox{\tiny $W$}}^2}  \to \frac{1}{1-\Gamma /\delta}, \qquad \lim _{ \sigma_{\mbox{\tiny $W$}}^2 \to \infty} \frac{\bar\sigma_{\mbox{\tiny P-LAD}}^2}{\sigma_{\mbox{\tiny $W$}}^2}  \to \frac{\Gamma}{\delta},
\]
with
$
\Gamma  =  \EE \eta^2(Z;\alpha)   
$, where $\EE$ is with respect to $Z$.
\end{lemma}

We observe that the result above does not depend on the sparsity $s$. Moreover,  as $ {(1-\Gamma /\delta)^{-1}}$ is larger than or equal to $\Gamma/\delta$, it displays a universally better efficiency of P-LAD over P-LS for  all ``heavy-tailed distributions" $f_W$.  In \cite{DM13}, for the unpenalized LAD and LS, such universal guarantees do not exist  and  are also dimensionality dependent. However, in the presence of  model selection we obtain a new behavior, where P-LAD achieves better asymptotic efficiency for every  $s$ and $p$ and $n$ and $n,p \to \infty$ with $n/p \in(0,1)$.
  
\section{Numerical Simulation} \label{sec:examples}

Within this section, we'd like to show the finite sample performance of RAMP from the following five aspects. First, we   discuss how to select the tuning parameter and show the existence and uniqueness of the state evolution parameters  while allowing different loss functions. Second, we   show the limit behaviors of iterative parameters of RAMP with different loss functions. Third, we  compare the performance of RAMP algorithm with different error distribution settings, which includes light--tailed and heavy--tailed. Fourth, we   release the assumption of the Gaussian design matrix and show that the distribution of design matrix does not effect the asymptotic performance of the distribution. Finally, we   discuss the relative efficiency of the RAMP estimators with different undersampling and sparsity setting.

\subsection{Tuning Parameter Selection \& Implementation}\label{sec:tuning}
The policy to choose for thresholds $\theta_t$ is based on \cite{DMM09},  which sets $\theta_t = \alpha \bar{\tau}_t$, where $\alpha$ is taken to be fixed.  In \cite{DMM09},   authors choose a grid of $\alpha$ starting from $\alpha_{min},$ so as to get a grid of iterative parameters. We mimic the same approach as that which offers a set of $\alpha$ within an interval $[\alpha_{min}, \alpha_{max}]$. For each $\alpha$, we get the RAMP estimator $x^t$ and SE iterative parameters $\bar{\tau}_t$ and $\bar{\sigma}_t$. We use these parameters to evaluate the AMSE$(x^t,x_0)$ and then tune the optimal $\alpha$ by minimizing AMSE$(x^t,x_0)$. In other words, $\bar{\tau}_t$ is calculated by the recursion $\bar{\tau}_t^2 = \mathbb{V}(\bar{\sigma}^2, \alpha\bar{\tau}_t)$, where $\mathbb{V}$ is the right hand side of equation (\ref{tau}) and $\bar{\sigma}$ is calculated from equation (\ref{sigma}).  The following simulation sections substitute $\theta=\theta(\alpha)$ to be $\lambda$ as a tuning parameter based on Lemma \ref{the:thethe}, in order to do an easy comparison  between the huber loss,  the least squares loss and the quantile loss.
In our simulation examples, we implement equations  \eqref{tau},\eqref{sigma} and \eqref{AMSE} for different cases  of the loss functions $\rho$. When $p>n$, it is hard to simplify the expression of these equations, except when the error is normal (simplified  as an equation (1.7) in \cite{DMM09}).

\subsection{Existence and Uniqueness of State Evolution parameters}\label{sec:existence}

In this subsection, we offer plots of the recursion $\bar{\tau}_t^2 = \mathbb{V}(\bar{\sigma}^2, \alpha\bar{\tau}_t)$ to show that $\bar{\tau}^*$ exists and is unique as iteration goes with differentiable and non-differentiable loss functions. We  choose $\alpha = 2$  to illustrate the worst case behavior. We fix $\delta = 0.64$, with $p = 500$ and $x_0$ that follows $P(x_0 = 1) = P(x_0 = -1) = 0.064$ with t $P(x_0 = 0) = 0.872$.

\begin{figure}[H]
\centering
\vspace{-20pt}
\includegraphics[trim = 50mm 60mm 15mm 40mm, width=0.4\textwidth, height=0.25\textheight]{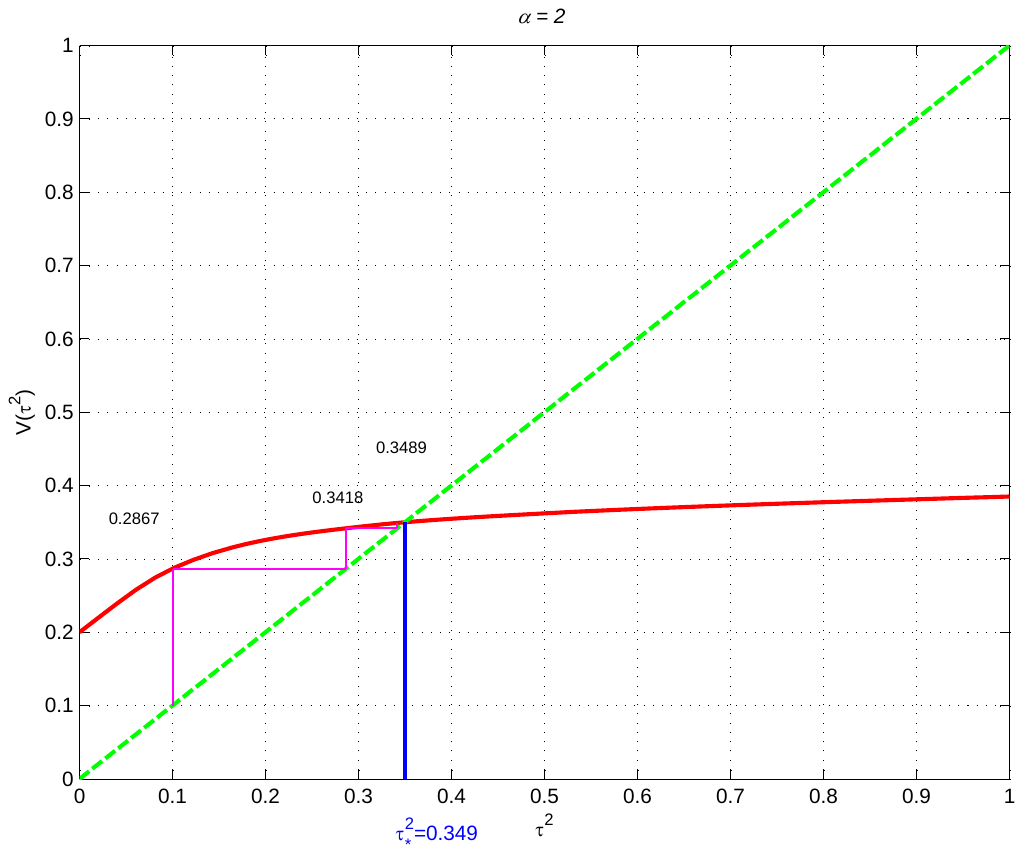}
\includegraphics[trim = 20mm 60mm 45mm 40mm, width=0.4\textwidth, height=0.25\textheight]{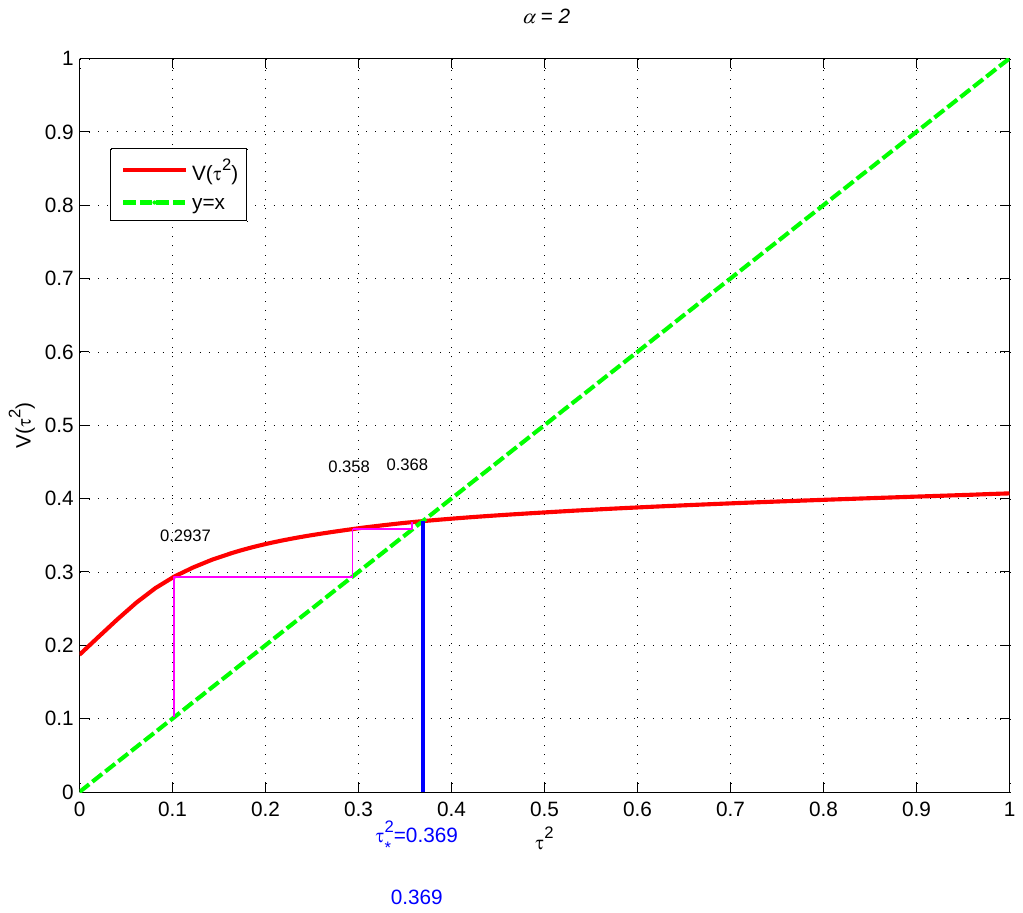}
\caption{Existence and uniqueness of $\bar{\tau}^2$ with different loss function: the square loss and the huber loss} \label{fig:1}
\end{figure}

\begin{figure}[H]
\centering
\vspace{-20pt}
\includegraphics[trim = 50mm 60mm 15mm 40mm, width=0.4\textwidth, height=0.25\textheight]{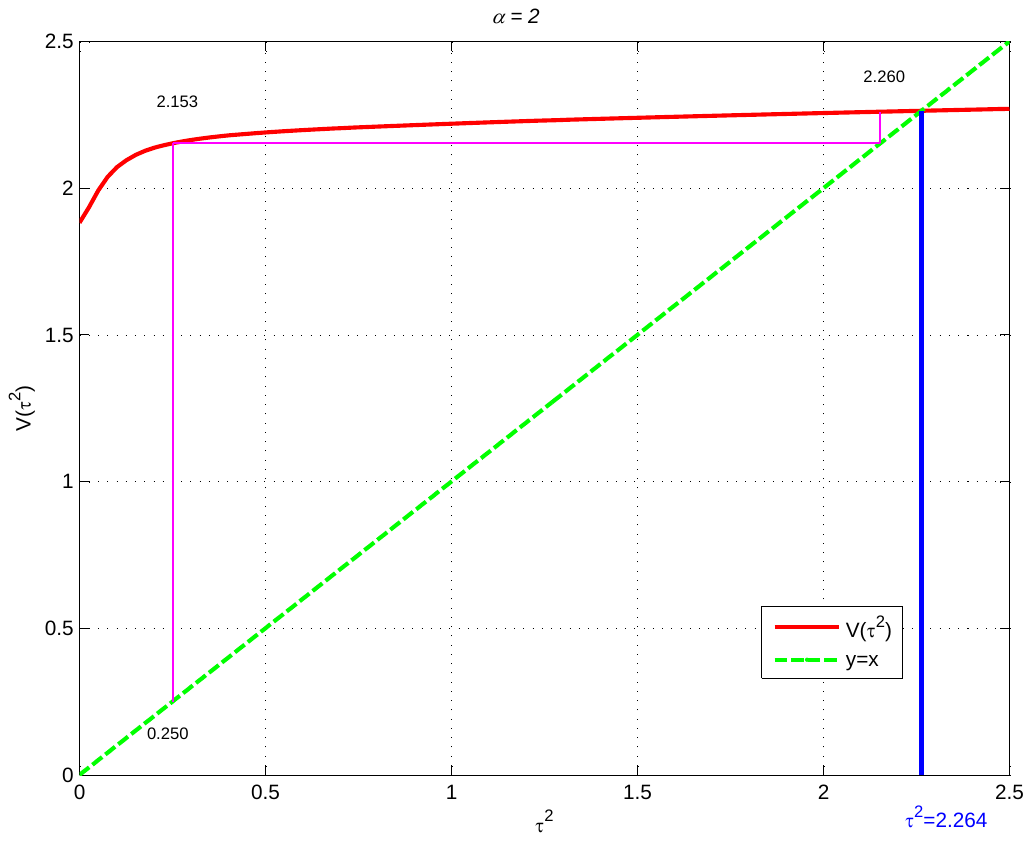}
\includegraphics[trim = 0mm 58mm 62mm 40mm, width=0.4\textwidth, height=0.25\textheight]{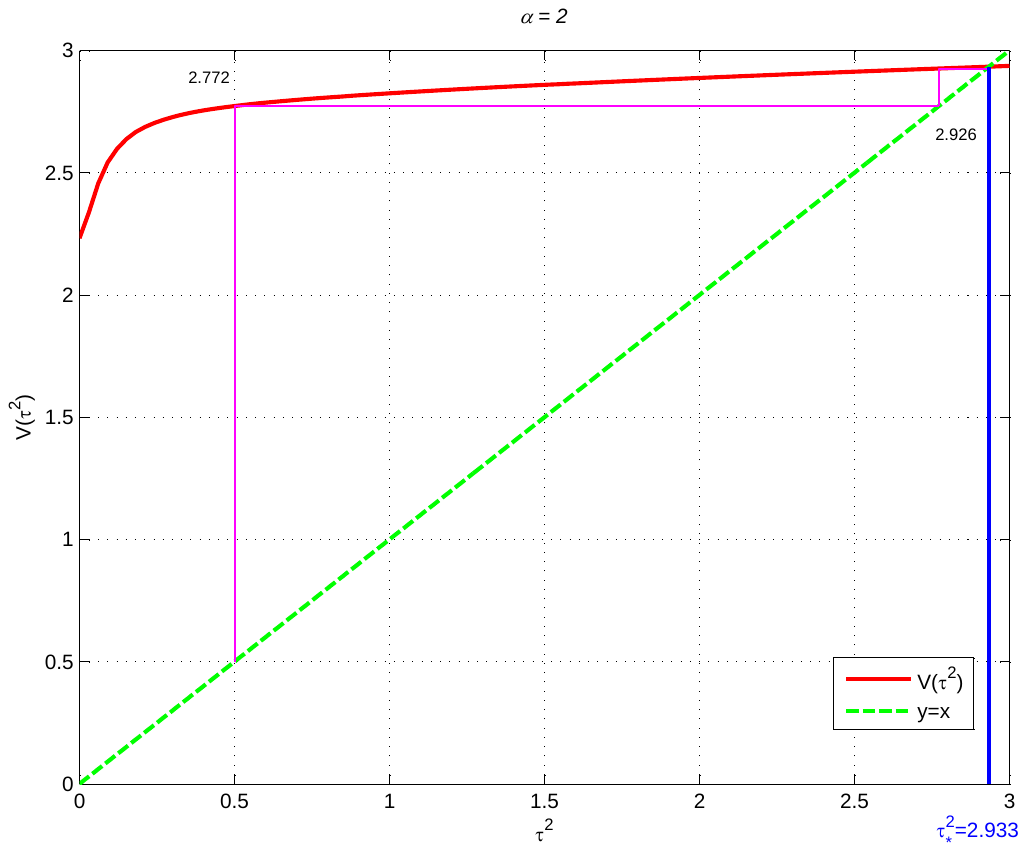}
\includegraphics[trim = 50mm 60mm 15mm 40mm, width=0.4\textwidth, height=0.25\textheight]{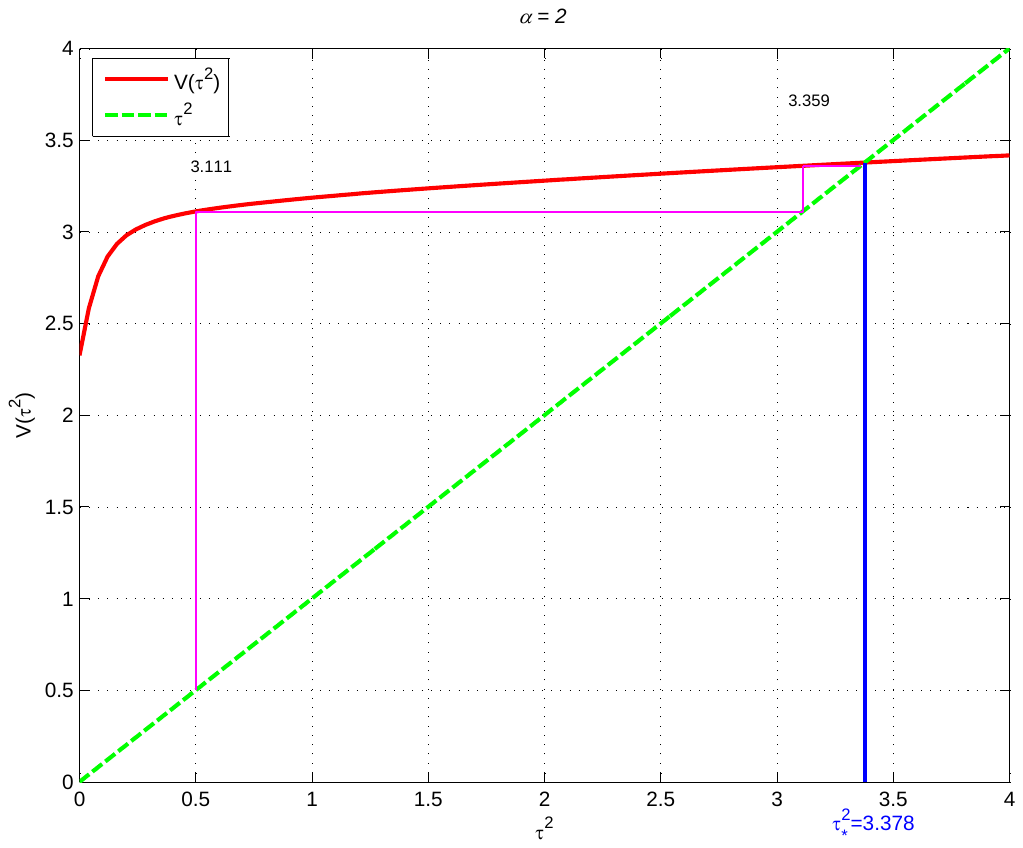}
\caption{Existence and uniqueness of $\bar{\tau}^2$ with non-differentiable  loss functions: the absolute loss and the quantile losses with $\tau = 0.7$ and   $\tau = 0.3$, respectively.}\label{fig:2}
\end{figure}

We focus on Gaussian distribution $\mathcal N (0,0.2)$ for the errors $W$ and show loss of efficiency when other than least squares loss is considered (see Figure \ref{fig:efficiency} below).
Results of the state evolution equations are presented in  Figures \ref{fig:1}-\ref{fig:efficiency}  below, where in the Gaussian setting above, we consider the least squares loss, the huber loss with $\gamma=1$, the least absolute deviation loss and the quantile losses with $\tau=0.7$ and $\tau=0.3$. We observe that the unique value of the state-evolution recursions is easily found even for the non-differentiable losses, under the recommendations of Section \ref{sec:RAMP}.
Figures \ref{fig:1} and \ref{fig:2}, right panel, shows how $\bar \tau_t^2$ evolves to the fixed point near $0.349$ starting from  $0.1$ for the case of the least squares loss and to the fixed point near $0.369$, $2.264$, $2.933$, $3.378$ for the case of the huber, least absolute deviations and quantile losses, respectively. Simultaneously the mapping $\mathbb{V}(\bar \tau^2,b,\theta)$ evolves to the fixed points near $0.348$, $0.368$, $2.260$, $2.926$, $3.359$ for all five losses considered - including non-differentiable losses.
Moreover, Figure \ref{fig:efficiency} illustrate that the loss is not great, even when we start from the randomly chosen starting $\alpha$ value. We perform further efficiency study in the subsection \ref{sec:RE1}.

\begin{figure}[H]
\centering
\vspace{-20pt}
\includegraphics[trim = 50mm 62mm 30mm 40mm, width=0.4\textwidth, height=0.25\textheight]{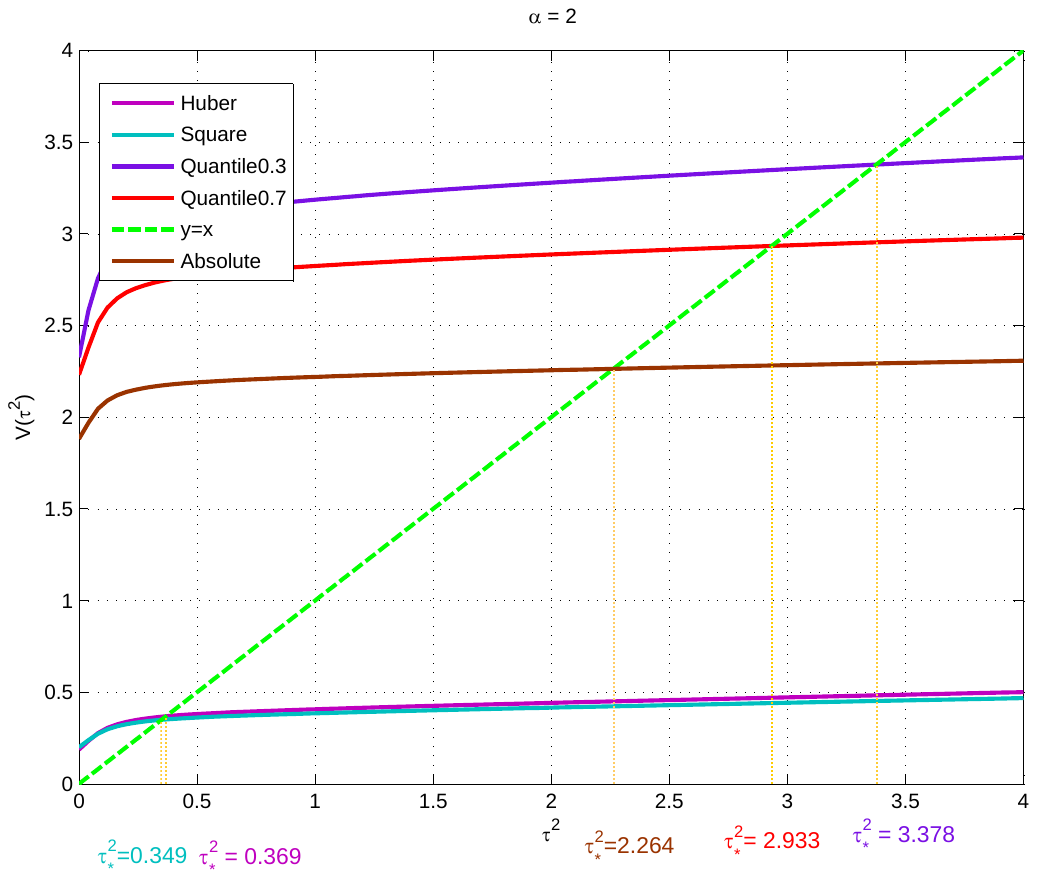}
\caption{Comparative plot of the mapping  $\mathbb{V}(\bar{\tau}^2)$ with different  loss functions with the  error $W$ following standard normal distribution and with fixed $\delta = 0.64$ and $p=500$.}\label{fig:efficiency}
\end{figure}

\subsection{Limit behavior of the parameters of  RAMP} \label{sec:con}
We assess the limit behaviors of parameters of different loss functions to express the iterations of the  RAMP algorithm. We are interested in the linear regression model
\[
   Y = A^Tx_0 + W,
\] 
where each element of $A$ is i.i.d. and  follows $N(0, 1/n)$. The error $W$ follows $N(0, 0.2)$ and the sample size is 320. We consider   fixed ratio $\delta = 0.64$. The distribution of the true parameter is set as $\mathbb{P}(x_0 = 1) = \mathbb{P}(x_0 = -1) = 0.064$ and $\mathbb{P}(x_0 = 0) = 0.872.$

The simulation step is as follows.
 We use $\omega = s/p = 0.128$ based on the setting of the $p_{x_0}$ into equation (\ref{alg:alg3}) to generate b.
We generate a series of $\alpha$, and regard the threshold $\theta_t = \alpha *\bar{\tau}_t$.
Then,
 we use the iteration of $\bar{\sigma}_t$, $\bar{\tau}_t$ from Lemma \ref{lem:lem22} to find the stable point $\bar{\tau}^*$ with stopping at $|\bar{\tau}_t-\bar{\tau}_{t-1}|<tol$, where $tol$ is a small positive number and is taken to be $10^{-6}$ here.
Lastly, we use the expression of $\lambda = \frac{\alpha \bar{\tau}^*\omega}{b\delta}$ and the expression of AMSE in Theorem \ref{the:the1} to find the AMSE$(x^t,x_0)$.
The penalized $M$-estimators theory suggest  cross-validation for the optimal values of $\lambda$. For such value we find 
 its corresponding AMSE$(x^t,x_0)$ and present it in Table \ref{tab:1} below.

\begin{table}[!ht]
\caption{Convergence of RAMP iteration with different loss function }
\label{tab:Table1}
\centering
\begin{tabular}{l l l l l l}
Loss Function &b&optimal $\lambda$& iteration steps&$\bar{\tau}^{*^2}$ & AMSE\\
\hline\hline
Square Loss& 0.2711864&0.6970546&8&0.3265822&0.0810126\\
Huber Loss&0.2714135&0.6261463&12& 0.3431436&0.09150527 \\
Absolute Loss& 0.4990769&1.91523&8&2.0276825 &0.0943257\\ 
Quantile Loss &0.7319994&1.402867&11&2.821827&0.1177329\\
\hline\hline
\end{tabular} \label{tab:1}
\end{table}

  Table  \ref{tab:1} compares several necessary parameters in the iteration of the RAMP algorithm. We contrast   four different loss functions: Least Squares loss, Huber loss with $\gamma=1$, Least Absolute Deviation loss and Quantile Loss with $\tau=0.7$.  The results presented in the table are averages over $100$ repetitions. We notice that within only twenty iteration steps, the RAMP algorithm becomes stable no matter of the loss function considered. Furthermore, we present values of a number of   parameters of the RAMP algorithm: $\min$-regularization $b$, regularization $\lambda$ and state evolution $\bar \tau^*$.  We observe that they all differ according to the loss function considered, illustrating that there is no universal choice of the above parameters that works uniformly well for all loss function.\\

Additionally, we present Figure \ref{fig:correct} and show the empirical convergence of AMSE$(x^t,x_0)$ with respect to the optimal tuning parameter $\lambda$ and  different loss functions. The plots illustrate  that when $\lambda$ becomes larger, the AMSE$(x^t, x_0)$ decreases dramatically and further   stabilizes around $0.12$. The reason AMSE$(x^t,x_0)$ becomes fixed on $0.12$ is  because the RAMP algorithm shrinks the estimator $x^t$ to be the zero vector;  hence,  the AMSE$(x^t, x_0) = ||x^t||_2^2 = 0.064 +0.064 = 0.128$, when $\lambda$ is large enough.  Moreover, we notice that for each  of the loss functions, the RAMP algorithm  chooses  the optimal $\lambda$, which will offer the minimum AMSE$(x^t, x_0)$. Therefore, the RAMP algorithm maintains the advantage of the AMP, which in turns, offers  an optimal  solution to the problem \eqref{eqDef}.

\begin{figure}[h]
\hskip 65pt
\includegraphics[trim = 60mm 80mm 0mm 40mm, width=0.4\textwidth, height=0.25\textheight]{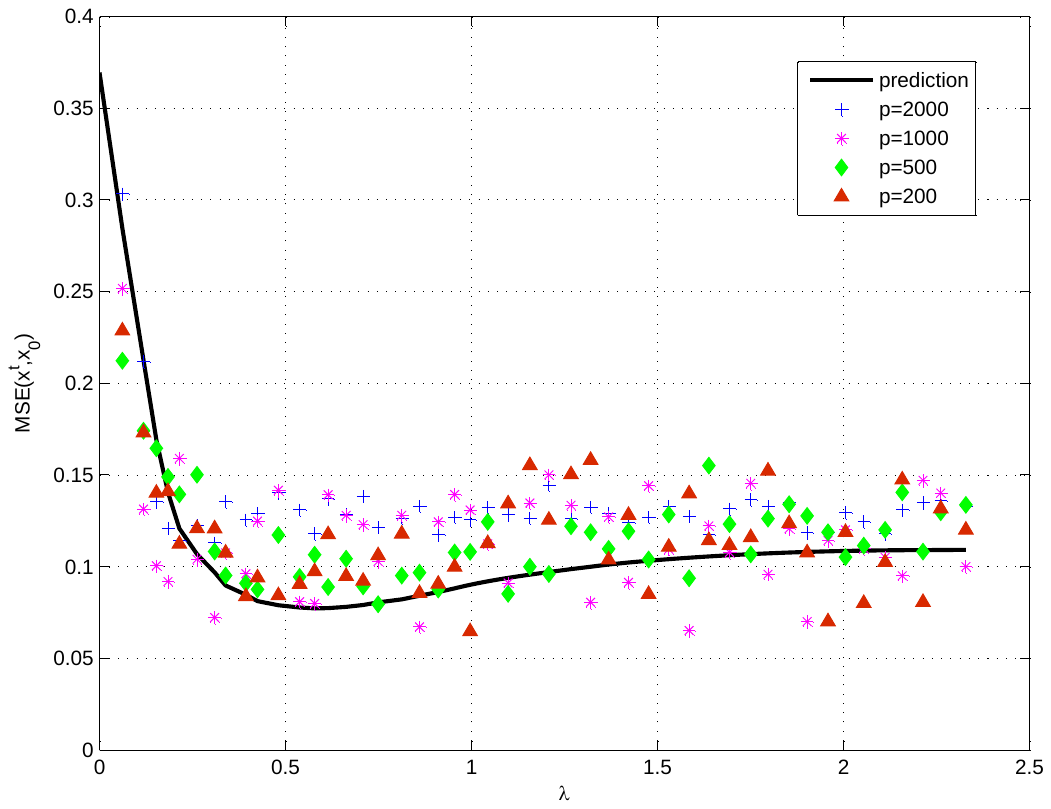}
\includegraphics[trim = 30mm 80mm 30mm 40mm, width=0.4\textwidth, height=0.25\textheight]{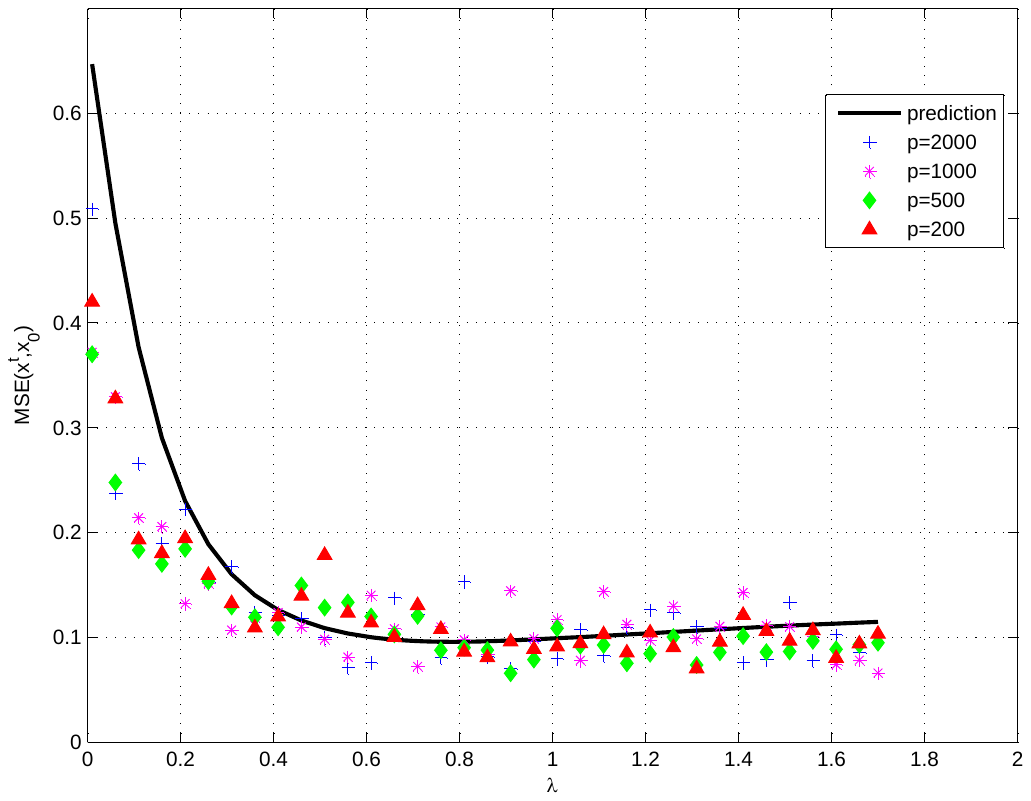}
\caption{AMSE$(x^t,x_0)$ compared to AMSE$(\hat{x},x_0)$ under various feature size and with different loss functions: least squares (left) and huber loss (right).}\label{fig:correct}
\end{figure}

\subsection{Robustness of RAMP with respect to the error distribution}\label{sec:error}
Further,   we know that using square loss to solve problem (\ref{eqDef}) is very sensitive with respect to the error distribution, which is the reason we release the loss function from the least squares loss to the general convex loss function satisfying Condition {\bf (R)}. We consider the robustness of the solution when the tail of error in model varies.

We assess the finite sample performance of RAMP through various models. We simulated data from the following model:
\[
Y = A^Tx_0 + W,
\]
where we generated $n = 640$ observations, $\delta = 0.64$, true parameter $x_0$ from $P(x_0 = 1) = 0.064, P(x_0 = -1) = 0.064$ and $P(x_0 = 0) = 0.872$, and each elements of $A$ satisfies $\mathcal{N} (0,1/n)$. We compare five scenarios for the error vector $w$. They are as follows: (a) light--tailed distribution: Normal $\mathcal{N}(0,0.2)$, Mixnormal $0.5\mathcal{N}(0,0.3)+0.5\mathcal{N}(0,1)$  and (b) heavy--tailed distribution: $t_8$, $t_4$, MixNormal $0.7\mathcal{N}(0,1) + 0.3 \mathcal{N}(0,3)$ and Cauchy$(0,1)$. The  Mixture of Normals distribution   generates samples from different normal distributions with corresponding probability and samples are centered to have mean zero.

\begin{figure}[h]
\centering
\includegraphics[trim = 55mm 60mm 10mm 30mm, width=0.4\textwidth, height=0.25\textheight]{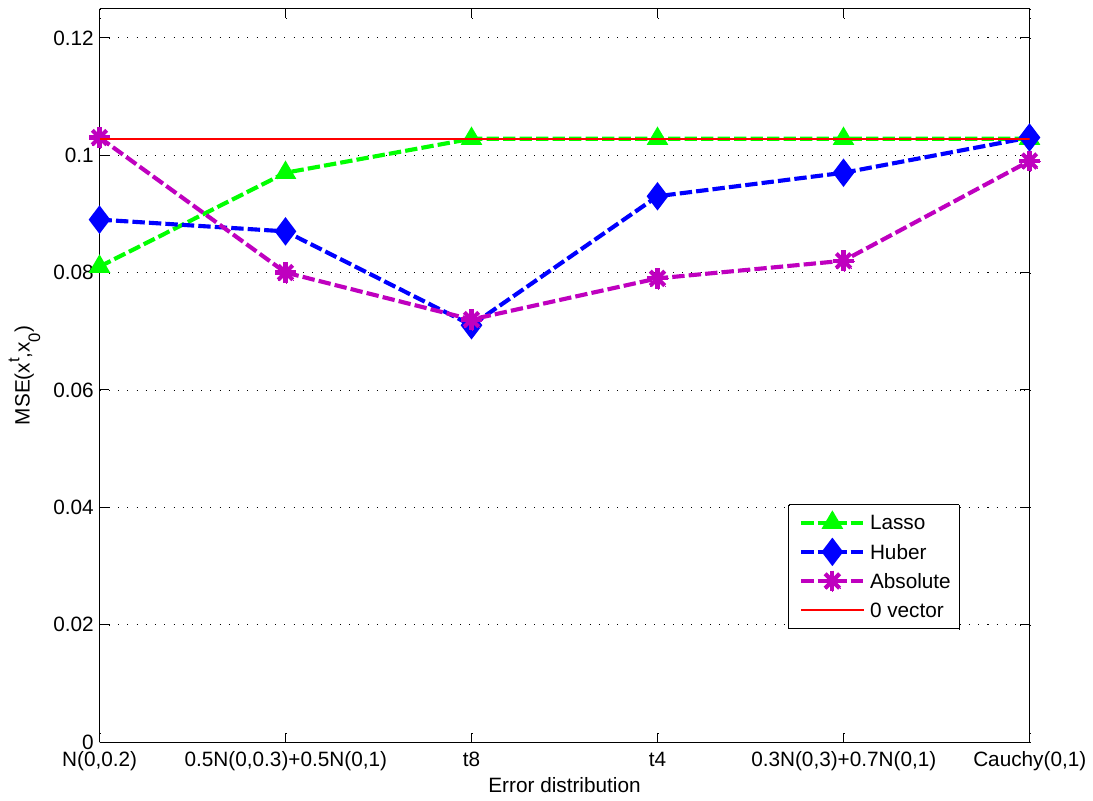}
\caption{AMSE$(x^t,x_0)$   under  various error distribution settings}
\label{fig:robust}
\end{figure}

Results of this experiment are presented in Figure \ref{fig:robust}. A few observations immediately follow.   The Lasso estimator is sensitive to the heavy tail error distribution whereas, the   Huber loss and the Least Absolute Deviation   loss perform better as the tail of the error distribution becomes heavier. Moreover, with larger tails the Least Absolute Deviation loss is clearly preferred over both the Huber and the Least Squares loss, whereas situation reverses when the tails are light. The Mixture of Normals errors are particularly difficult due to the bimodality of the error distribution. We see that  in both light and heavy tales cases of Mixture distribution, Huber Loss is preferred over the Least Squares loss.
  Lastly, as the tails becomes even heavier, all estimators face the problem  of  estimating the unknown parameter accurately. 

\subsection{Convergence property of RAMP with random design }\label{sec:design}
We proved that in case of the Gaussian  design matrix $A$  where each element has  mean 0 and  variance of $\frac{1}{n}$, the RAMP  algorithm recovers the penalized M-estimator in Theorem \ref{the:thethe}. We now release the restriction of the design matrix and generate $A$ from three different scenarios (we choose $\delta$ = 0.64 for all cases). First is the case of  $A_{ij}$ being i.i.d.  and following  $\mathcal{N}(0,1/n)$, whereas the last two are composed of the cases where $A$  is  random $\pm1$ matrix,  with each entry $A_{ij}$ being i.i.d  and such that $A_{ij}=\frac{1}{\sqrt{n}}$ or $A_{ij}=\frac{-1}{\sqrt{n}}$ with equal probability.

For each of the settings above, we plot the AMSE$(x^t,x_0)$ with respect to the $\lambda$ value. Average results  over $100$ repetitions are summarized in Figure \ref{fig:design}.
We find that the  two Binomial design  settings do  not change the line of AMSE$(x^t,x_0)$ and are very similar to the AMSE$(\hat{x}, x_0)$ with the Normal design. Even though we have not proved that the different design matrix does not effect the performance of the RAMP estimator,   because of  the central limit theorems effects we observe  the diminished influence in the results.  

\begin{figure}[h]
\hskip 65pt
\includegraphics[trim = 60mm 80mm 0mm 40mm, width=0.4\textwidth, height=0.25\textheight]{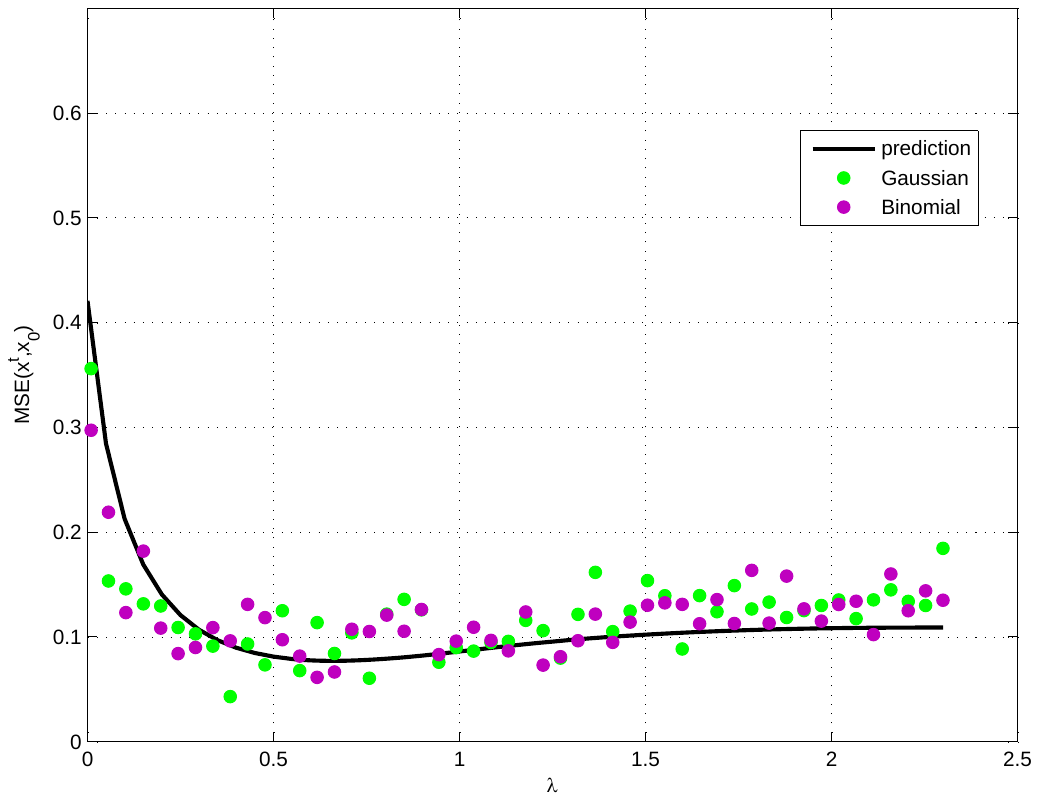}
\includegraphics[trim = 30mm 80mm 30mm 40mm, width=0.4\textwidth, height=0.25\textheight]{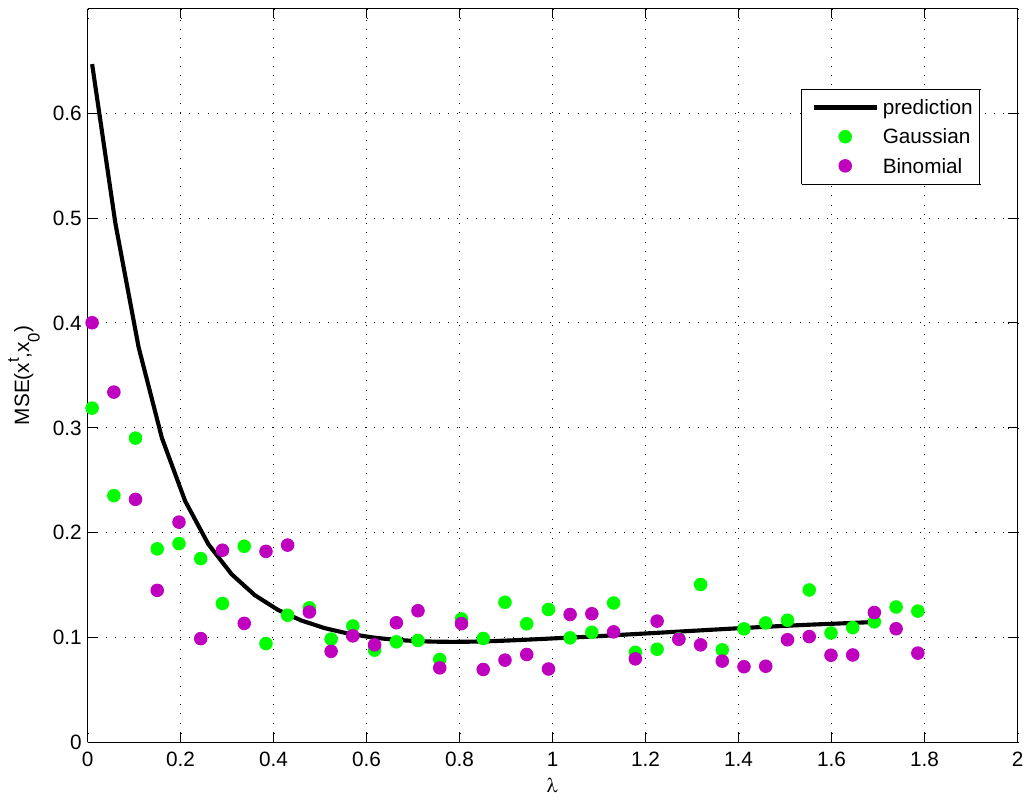}
\caption{AMSE$(x^t,x_0)$  under two design settings Gaussian design (green) and Binomial (purple) with least squares loss (left) and huber loss (right).}\label{fig:design}
\end{figure}

\subsection{Relative efficiency} \label{sec:RE1}
We use RAMP iteration to calculate the relative efficiency of the  Least square estimator versus the Least absolute estimator. It is known that the least square estimator is preferable in normal error assumption, but the least absolute estimator beats the least square estimator in double-exponential error assumption under classical low-dimensional setting.

In Table \ref{tab:Table1}, we fix $p = 50$ and discuss the comparison of relative efficiency between the low-dimensional case (where $p < n$) and the high-dimensional dense case (where $p \approx n$). We discuss the AMSE$(x^t, x_0)$ with  a different ratio of $\frac{p}{n}$ (10, 8, 3, 1.6, 1.4, 1.2) under two error settings (which are $N(0,0.2)$ and double exponential $(0,1)$). When we implement the equations \eqref{tau},\eqref{sigma} and \eqref{AMSE}, we consider $\eta$ function to be an identity function and $\omega$ is 1, because   neither the penalty nor the sparsity is needed.

\begin{table}[h]
\caption{Relative efficiency of Square Loss estimator w.r.t. Absolute Loss estimator under various low  dimensional  setting without sparsity}
\label{tab:Table1}
\begin{tabular*}{\linewidth}{ @{\extracolsep{\fill}} ll *{13}c @{}}
\toprule
  Relative Efficiency & \multicolumn{3}{c}{Least Squares} &
\multicolumn{3}{c}{Least Absolute Deviations}   \\
\midrule \midrule \addlinespace \\
\multicolumn{10}{c}{$p < n$  ,  with   fixed  $p =50$ and varying $n$}   \\
\addlinespace
   & $\delta =10$ & $\delta=8$  & $\delta =3$    & $\delta =10$ & $\delta=8$  & $\delta =3$      \\
\cmidrule{2-4} \cmidrule{5-7}  
\addlinespace \\
   Normal &   0.204 & 0.234 &0.308 & 0.395 & 0.439 & 0.568 \\ 
                           Laplace & 2.362&2.376   &  3.119 &  1.415&   1.792&1.578  \\ 
                                                                               \midrule \addlinespace \\
                                                    \multicolumn{10}{c}{$p \approx n$  ,  with   fixed  $p =50$ and varying $n$ }   \\
\addlinespace
                         & $\delta =1.6$ & $\delta =1.4$  & $\delta =1.2$    & $\delta =1.6$ & $\delta =1.4$  & $\delta =1.2$     \\
\cmidrule{2-4} \cmidrule{5-7} \cmidrule{8-10}  \addlinespace \\
  Normal &0.489 & 0. 643 & 1.102   & 0.946 & 0.962&1.192\\ 
                           Laplace & 5.544& 7.276  &  12.475 &  7.014 & 11.351  & 17.929 \\ 
                                                    \bottomrule
\end{tabular*}
\end{table}

\begin{table}[H]
\caption{Relative efficiency of penalized Square Loss estimator  w.r.t. penalized Absolute Loss estimator under various high dimensional     and sparsity setting}
\label{tab:Table2}
\begin{tabular*}{\linewidth}{ @{\extracolsep{\fill}} ll *{13}c @{}}
\toprule
  Relative Efficiency & \multicolumn{3}{c}{Least Squares} &
\multicolumn{3}{c}{Least Absolute Deviations}   \\
\midrule \midrule \addlinespace \\
\multicolumn{10}{c}{$p > n$ and  $s<n$,  with    fixed  $p=500$ and $\delta =0.64$ and varying $s/n$}   \\
\addlinespace
   & $\omega =0.05$ & $\omega =0.1$  & $\omega =0.2$    & $\omega =0.05$  & $\omega =0.1$  & $\omega =0.2$     \\
\cmidrule{2-4} \cmidrule{5-7}  
\addlinespace \\
   Normal & 0.042 & 0.0839 & 0.139 & 0.0458 & 0.113 & 0.183\\ 
                           Laplace &0.0437& 0.0914 & 0.192 & 0.0322 & 0.0745 & 0.177\\ 
                                                                               \midrule \addlinespace \\
                                                    \multicolumn{10}{c}{$p > n$ and  $s\approx n$,  with $p=500$ and fixed  $\delta =0.64$}   \\
\addlinespace
                         & $\omega =0.5$ & $\omega =0.55$  & $\omega =0.6$    & $\omega =0.5$  & $\omega =0.55$  & $\omega =0.6$   \\
\cmidrule{2-4} \cmidrule{5-7} \cmidrule{8-10}  \addlinespace \\
  Normal & 0.394 & 0.458 & 0.468  & 0.385 & 0.432 & 0.477\\ 
                           Laplace &0.522& 0.531 & 0.584 & 0.207 &0.245 & 0.289\\ 
                                                    \bottomrule
\end{tabular*}
\end{table}

 From the first two rows of Table \ref{tab:Table1}, we  see that in a Normal error setting, the Least Square estimator is preferable and the relative efficiency of the Least Square estimator w.r.t. the Least Deviation estimator is around $ {2}/{\pi}$. Further, we can see that in the Double exponential error setting, the Least Square estimator performs worse. This result matches the classical inference. From the last two rows, we can see that the Least Squares estimator is preferable no matter of the error distribution. This result is foreseen by \cite{DM13} and \cite{K13}.  
 
Remarkably, in Table \ref{tab:Table2}, we discuss a high-dimensional and sparse case ($p>n$). We fix $\delta = 0.64$ and $p=500$. This provides $n=320$.  For the  number of the non-zeros in true parameter, $s$,  we choose a variety  of options which  range from low-sparsity, $25$, to high sparsity, $300$.  From the first two rows of Table \ref{tab:Table2}, we see that in a Normal error setting, penalized Least Squares (P-LS) estimator is no longer preferred in all settings.  When the sparsity, $s$, is high and reaches $n$, penalized Least Absolute Deviations (P-LAD) estimator is preferred, whereas when the sparsity, $s$, is low,  P-LS estimator is preferred. However, from the last two rows, in the setting of the Laplace distribution, we see that P-LAD estimator is always preferred no matter of the size of $s$. This contradicts the findings of Table \ref{tab:Table1} and shows that model selection  affects the choice of the optimal loss function.


\section{Technical Proofs}

\subsection{Proofs for Section \ref{sec:2}}

\begin{proof}[Proof of Lemma \ref{lemma:lemma1}]
Let $(x,z)$ be a fixed point of the   RAMP algorithm iteration. Then the fixed point conditions at x read as
\[
  x=\eta(A^TG(z;b)+x;\theta)= \left\{\begin{array}{ll}
                      x+A^TG(z;b)-\theta, & \mbox{if $x+A^TG(z;b) >\theta$}\\
                       0,        &  \mbox{if $-\theta \leq x+A^TG(z;b) \leq \theta$}\\
                    x+A^TG(z;b)+\theta,&  \mbox{if $x+A^TG(z;b) < -\theta$}
                    \end{array}\right. .
\]
They imply that for all $x+A^TG(z;b) >\theta$, $x = x+A^TG(z;b)-\theta$, or in other terms that $A^TG(z;b)=\theta $. 
Similarly, $x+A^TG(z;b) <\theta$, $x = x+A^TG(z;b)+\theta$, or using different terms, that $A^TG(z;b)=-\theta $.
For the middle term, we observe that $x=0$, if and only if $-\theta < A^TG(z;b) < \theta $. 
 Hence,
\begin{align}\label{eq:eq12}
A^TG(z;b)=\theta v, 
\end{align}
where $v\in \mathbb{R}^p$ with each element 
$
v_i = \left\{\begin{array}{cc}
\mbox{sign} (x_i) & \text{if }  x_i \neq 0\\
(-1,1) & \text{if } x_i = 0
\end{array}
\right. .
$
Therefore, the correction term defined as  the average of the first derivative of $\eta(A^TG(z;b)+x;\theta)$, becomes:
\begin{align*}
\langle \partial_1\eta(A^TG(z;b)+x;\theta)\rangle =\langle\mathbbm{1}\{\lvert A^TG(z;b)\rvert\ \neq \theta\}\rangle 
                                                           =\langle\mathbbm{1}\{ x\neq 0\}\rangle 
                                                           =\frac{||x||_0}{p} = \omega.
\end{align*}
\end{proof}

\begin{proof}[Proof of Lemma \ref{the:thethe}]
The fixed point condition at $z$ reads
\begin{align}
z &= Y-Ax+\frac{1}{\delta}G(z;b)\langle \partial_1\eta(A^TG(z;b)+x;\theta)\rangle\nonumber.
  \end{align}
 Moreover, from Lemma \ref{lemma:lemma1} we conclude 
 $\langle \partial_1\eta(A^TG(z;b)+x;\theta)\rangle = \omega$, and hence  
  $
   z=Y-Ax+\frac{1}{\delta}\omega G(z;b)
   $. By definition of the rescaled effective score $G$, we conclude 
   $z=Y-Ax+\Phi(z;b),
$
which shows that $Y-Ax=z-\Phi(z;b)$.
Then, we have that the left hand side of the KKT condition becomes
\begin{align}\label{eq:eq20}
A^T\rho'(Y-Ax) &= A^T\rho'(z-\Phi(z;b)) \overset{(i)}{=}A^T\rho'(Prox(z,b))\nonumber\\
                                      &\overset{(ii)}{=}A^T\Phi(z;b)/{ b } 
                                       \overset{(iii)}{=} \frac{\theta v\omega}{\delta b}.
\end{align}
The equations (i) and (ii) are derived from the definition of $\Phi(z;b)$, equation (iii) is based upon the proof of Lemma \ref{lemma:lemma1}, equation (\ref{eq:eq12}). Hence,
\begin{align*}
A^TG(z;b)=A^T\Phi(z;b)\delta/\omega=\theta v
\end{align*}
so that $
A^T\Phi(z;b)/b=\frac{\theta v\omega}{\delta b}$.
Plugging $\theta= \frac{\lambda b\delta }{\omega}$ into equation (\ref{eq:eq20}), we have 
$
A^T\rho'(Y-Ax) = \lambda v.
$
\end{proof}

\begin{proof}[Proof of Lemma \ref{lem:lem22}]
This is an immediate application of state evolution  as defined in \cite{BM11}, which considers general recursions. Hence, it suffices to show  that the proposed algorithm is a special case of it.  
In the original notation of \cite{BM11}, the generalized recursions studied are
\begin{equation}\label{bb}
b^t = Aq^t -\lambda_tm^{t-1}
\end{equation}
\begin{equation}\label{hh}
h^{t+1}=A^Tm^t-\xi_tq^t
\end{equation}
where 
\begin{equation}\label{qf} 
q^t=f_t(h^t), \qquad 
m^t=g^t(b^t,w).
\end{equation}
The two scalars  $\xi_t$ and $\lambda_t$ are defined as 
\begin{equation}\label{xi}
\xi_t =\langle g'_t(b^t,w)\rangle 
\end{equation} and 
\begin{equation}\label{lambda}\lambda_t=\frac{1}{\delta}\langle f'_t(h^t)\rangle, \end{equation}where$\langle \cdot \rangle $ denotes an empirical mean over the entries in a vector and derivatives are with respect to the first argument. According to \cite{BM11}, the state evolution recursion involves two variables:
$
\bar{\tau}^2_t = \EE {g^2_t(\bar{\sigma}_tZ,W)} 
$ and
$
\bar{\sigma}_t=\frac{1}{\delta}\EE {f^2_t(\bar{\tau}_{t-1}Z)}.
$
To see that 
the RAMP algorithm in (\ref{alg:alg1}), (\ref{alg:alg2}) and (\ref{alg:alg3})   is a special case of this recursion,     we specify the above components of the general recursion to be
\begin{align}\label{h:h}
h^{t+1}&=x_0-A^TG(z^t;b_t)-x^t 
\\
\label{q:q}
q^t&=x^t-x_0
\\
\label{z:z}
z^t&=w-b^t
\\
\label{m}
m^t&=-G(z^{t};b_{t})
\\
\label{g}
g_t(s,w) &= -G(w-s;b_t)
\\
\label{f}
f_t(s)&=\eta(x_0-s;\theta)-x_0,
\end{align}
with the initial condition being $q^0=x_0$.
Now, we verify that  the simplification of the above series of equations \eqref{h:h}-\eqref{f} offer  the RAMP algorithm iterations. We discuss the first step of the algorithm and then the third, whereas we leave the discussion of the second step as the last. We observe
          \begin{align*}
          x^t&\overset{(\ref{q:q})}{=}q^t+x_0  \overset{(\ref{qf})}{=}f_t(h^t)+x_0 \overset{(\ref{f})}{=}\eta(x_0-h^t;\theta)-x_0+x_0  \\
                                                                  &\overset{(\ref{h:h})}{=}\eta(x_0-x_0+(A^TG(z^{t-1};b_{t-1}))+x^{t-1};\theta)  \\                                                                & =\eta(x^{t-1}+A^TG(z^{t-1};b_{t-1});\theta),
                                                                  \end{align*} which is the first step of our algorithm.
                                                                  Also, 
                                                                  \begin{align*}
z^t&\overset{(\ref{z:z})}{=}w-b^t
                                     \overset{(\ref{bb})}{=}w-Aq^t +\lambda_tm^{t-1}   
                                     \overset{(\ref{q:q})}{=}w-A(x^t-x_0)+\lambda_tm^{t-1}\\
                                     &\overset{(\ref{w})}{=}Y-Ax_0-A(x^t-x_0)+\lambda_tm^{t-1} 
                                     =Y-Ax^t+\lambda_tm^{t-1}\\
                                    &\overset{(\ref{m})(\ref{lambda})}{=}Y-Ax^t+\frac{1}{\delta}\langle f'_t(h^t)\rangle (-G(z^{t};b_{t}))\\
                                    &=Y-Ax^t+\frac{1}{\delta}G(z^{t};b_{t})\langle -\eta'(x_0-h^t;\theta)\rangle,   \qquad ( \text{Since }\langle f'_t(h^t)\rangle =\langle-\eta'(x_0-h^t;\theta)\rangle  )
                                                                  \end{align*}
                                                                  which is the third step of our algorithm.
Further, we need to show that $h^{t+1}$ in the above special recursion satisfies the equation of $h^{t+1}$ in general AMP, which means we need $h^{t+1}=A^Tm^t-\xi_tq^t=x_0-(A^TG(z^t;b_t))-x^t $. Therefore,
\begin{align*}
h^{t+1}&=A^Tm^t-\xi_tq^t \overset{(\ref{m})}{ =}-A^TG(z^{t};b_{t})-\xi_tq^t\\
                                                    & \overset{\eqref{q:q}}{=}-A^TG(z^{t};b_{t})-\xi_t(x^t-x_0 )  =x_0-(A^TG(z^t;b_t))-x^t. 
\end{align*}
This equation is only true when $\xi_t=1.$
                                   Moreover, by the definition of $G$,  we conclude that 
\[
\xi_t \overset{(\ref{xi})}{=} \langle G'(z^t;b_t)\rangle  = \frac{\delta}{\omega}\langle \Phi'(z^t;b_t)\rangle  = 1
\]
Therefore, we showed that the RAMP algorithm is a special case of general recursion, and we can conclude that  the Theorem 2 of \cite{BM11} applies and provides
\begin{align*}
\bar{\sigma}_t^2 &\overset{(\ref{sigma})}{=} \frac{1}{\delta}\EE (\eta(x_0-\bar{\tau}_{t-1}Z,\theta)-x_0))^2
\\
\bar{\tau}_{t}^2&\overset{(\ref{tau})}{=}\EE (G(W-\bar{\sigma}_tZ);b_t)^2.
                       \end{align*}
                       The proof is then completed by a simple observation that $Z$ and $-Z$ have the same distribution.
\end{proof}

\begin{proof}[Proof of Lemma \ref{lem:uniquetau}]

The statement of the lemma follows  if we successfully show that
(a) the total first derivative of $\mathbb{V}(\tau^2,b(\tau),\alpha \tau)$ is strictly positive for $\tau^2$ large enough;
(b) the function $\mathbb{V}$ is concave for all smooth loss functions $\rho$ and not for non-smooth loss functions $\rho$; and 
(c) the $\lim _{\tau \to \infty} \mathbb{V}'(\tau^2,b(\tau) ,\alpha \tau^2)$ is a strictly decreasing function of $\alpha$.

Part (a).\\
According to the definition of $\Phi$ and Condition ({\bf R}), we can represent 
$\frac{\partial \mathbb{V}(\tau^2,b,\theta)}{\partial  (\tau^2,\theta)} $
as 
\begin{align}
&\frac{\delta^2}{\omega^2} \frac{\partial  }{\partial  (\tau^2,\theta)} \EE [\Phi(W +\sigma  Z; b )]
\nonumber \\
&=
\frac{\delta^2}{\omega^2} \frac{\partial  }{\partial  (\tau^2,\theta)}  \left[ b \EE \left[(\upsilon_1 +\upsilon_2) (W +\sigma  Z  )\right] +  b \sum_{\nu=1}^k \alpha_\nu \PP \left\{ r_\nu \leq Prox (W +\sigma  Z ,b) \leq r_{\nu+1}\right\} \right]
\label{eq:derivation1}
\\
&=
\frac{\delta^2}{\omega^2} \left[ b \EE \left[\upsilon_1' (W +\sigma  Z  )  \frac{\partial \sigma }{\partial  (\tau^2,\theta)} Z \right] +  b   \sum_{\nu=1}^k \kappa_\nu \PP  ( q_\nu < W +\sigma  Z \leq q_{\nu+1})  + \right. \nonumber
\\ 
& \left.
+  b \sum_{\nu=1}^k \alpha_\nu  \EE_{x_0,Z
} \left( f (\bar r_{\nu+1})  \frac{\partial \bar r_{\nu+1} }{\partial  (\tau^2,\theta)}- f(\bar r_\nu)  \frac{\partial \bar r_{\nu} }{\partial  (\tau^2,\theta)} \right)     \right]
\nonumber
\end{align}
where $Prox'$ is derivative of the $Prox$ function with respect to its first argument, $f$ is the density of  $W$ and $\bar r_{\nu+1} $ is such that 
\[
  \bar r_{\nu+1} - \Phi(   \bar r_{\nu+1} + \sigma Z ;b) =  r_{\nu+1} -   \sigma Z.
\]
 By integrating the implicit relation above we obtain
 \begin{equation}\label{eq:b0}
 \frac{\partial  \bar r _{\nu+1}}{\partial  (\tau^2,\theta)}  = - Z  \frac{\partial  \sigma}{\partial  (\tau^2,\theta)}   \frac{\partial  \bar Prox (\bar r_{\nu+1} + \sigma Z;b)}{\partial  (\tau^2,\theta)} .
 \end{equation}
 Observe that $Prox$ is a  strongly convex function with bounded level sets. 
  \cite{BM12} derive  $\sigma$ to be  concave and for large $\tau^2 $ strictly increasing. Hence, $  \bar r _{\nu+1}+ \sigma Z$ can be made large and positive for large values of $\tau^2$.  In turn, $ \EE \frac{\partial  \bar r _{\nu+1}}{\partial  (\tau^2,\theta)} $ can be made strictly positive for large values of $\tau ^2$. Together with Condition {({\bf D})}  and convexity of $\rho$ we are ready to conclude that  $\frac{\partial  }{\partial  (\tau^2,\theta)} \EE [\Phi(W +\sigma  Z; b )]$ is strictly positive for large $\tau^2$.

%
By changing the order of differentiation and expectation (allowed by boundedness of functions considered given by Condition ({\bf R})), we obtain
that 
$b$ is defined as a solution to the equation 
$
  \partial_1 \EE [ \Phi (W + \sigma Z;b)] = \frac{\omega}{\delta}
$
(see Lemma \ref{lem:delta} for details).   More specifically,
\begin{align}\label{eq:b1}
  \left [ b \EE  \left[ (\upsilon_1+\upsilon_2) (W + \sigma Z )\frac{\partial \sigma (\tau^2,\theta)}{\partial \tau^2} Z \right]
    +b \sum_{\nu=1}^k \alpha_\nu  \EE_{x_0,Z
} \left( f (\bar r_{\nu+1})  \frac{\partial \bar r_{\nu+1} }{\partial   \tau^2 }- f(\bar r_\nu)  \frac{\partial \bar r_{\nu} }{\partial   \tau^2 } \right)   \right] = \frac{\omega}{\delta}.
\end{align}
Moreover,  as before, in Equation \eqref{eq:derivation1}
\begin{align} \label{eq:b2}
&\frac{\delta^2}{\omega^2} \frac{\partial  }{\partial  b } \EE [\Phi(W +\sigma  Z; b )] \nonumber
\\
&=
\frac{\delta^2}{\omega^2} \left[   \EE \left[ (\upsilon_1 +\upsilon_2 ) (W +\sigma  Z  )  \right] +    \sum_{\nu=1}^k \alpha_\nu  \EE_{x_0,Z
} \left( f (\bar r_{\nu+1})  \frac{\partial \bar r_{\nu+1} }{\partial   \tau^2 }- f(\bar r_\nu)  \frac{\partial \bar r_{\nu} }{\partial   \tau^2 } \right)     \right] \frac{\partial b (\tau^2)}{\partial \tau^2}. 
\end{align}
We focus on the last part of the above display. 
The total derivative of \eqref{eq:b1} provides the implicit 
equation for $\frac{\partial b}{ \partial \tau^2}$,
%
%
\begin{align*}\label{eq:b3}
&   b \EE  \left[ \upsilon_1' (W + \sigma Z )\frac{\partial^2 \sigma (\tau^2,\theta)}{\partial (\tau^2)^2} Z \right]  \nonumber
 +b \sum_{\nu=1}^k \kappa_\nu \PP   ( q_\nu < W +\sigma  Z \leq q_{\nu+1})\frac{\partial^2 \sigma (\tau^2,\theta)}{\partial (\tau^2)^2} Z\\
  &  +b \sum_{\nu=1}^k \alpha_\nu  \EE_{x_0,Z
} \left( f' (\bar r_{\nu+1})  \frac{\partial \bar r_{\nu+1} }{\partial   \tau^2 }- f'(\bar r_\nu)  \frac{\partial \bar r_{\nu} }{\partial   \tau^2 } + f  (\bar r_{\nu+1})  \frac{\partial^2 \bar r_{\nu+1} }{\partial   (\tau^2)^2 }- f(\bar r_\nu)  \frac{\partial^2 \bar r_{\nu} }{\partial   (\tau^2 )^2} \right)   \nonumber
\\
&+   \left [   \EE  \left[ (\upsilon_1 +\upsilon_2)(W + \sigma Z )\frac{\partial \sigma (\tau^2,\theta)}{\partial \tau^2} Z \right]
    + \sum_{\nu=1}^k \alpha_\nu  \EE_{x_0,Z
} \left( f (\bar r_{\nu+1})  \frac{\partial \bar r_{\nu+1} }{\partial   \tau^2 }- f(\bar r_\nu)  \frac{\partial \bar r_{\nu} }{\partial   \tau^2 } \right)   \right]   \frac{\partial b (\tau^2)}{\partial \tau^2}=0.
\end{align*}
Next, we observe that $\upsilon_1, \upsilon_2$ can be made positive for large $\tau^2$ and that $\upsilon_1'$ and $\upsilon_2'$  are positive. By \eqref{eq:b0} we can see that $ \frac{\partial \bar r_{\nu} }{\partial   \tau^2 }  >0$, $ \frac{\partial^2 \bar r_{\nu} }{\partial   (\tau^2 )^2} \geq 0 $,
  and $\frac{\partial^2 \bar r_{\nu} }{\partial   (\tau^2 )^2}  <\frac{\partial^2 \bar r_{\nu+1} }{\partial   (\tau^2 )^2} $ (curvature of a convex function decays away from the origin).
%
Moreover, \cite{BM12} prove  that $\sigma $ is strictly concave for $\alpha >0$. All of the above implies that $\frac{\partial \sigma (\tau^2,\theta)}{\partial \tau^2}  > 0$ for large $\tau^2$; $\frac{\partial^2 \sigma (\tau^2,\theta)}{\partial (\tau^2)^2} \geq 0$.
 Condition {({\bf D})} guarantees $f >0$. Hence, $\frac{\partial b}{ \partial \tau^2}$ can be made positive for large $\tau^2$.
 Next, it suffices to observe that 
 the total derivative of $\mathbb{V}(\tau^2,b,\theta)$ is given by the sum of the above marginal derivatives, all of which can be made positive.

Part(b).\\
Careful inspection of the second derivative of  $\mathbb{V}(\tau^2,b,\theta)$ provides details (by the same arguments above) that the second derivative is  negative, i.e., that the function $\mathbb V$ is concave for all smooth $\rho$ and not necessarily negative for all non-smooth loss functions $\rho$. We show the analysis for one of the marginals as the analysis for the rest is done equivalently. 
\begin{align*}
&\frac{\omega^2}{\delta^2}  \frac{\partial^2 \mathbb{V}(\tau^2,b,\theta)}{\partial (\tau^2)^2} 
\\
&= 
\frac{\partial   }{\partial  \tau^2 } \left[ b \EE \left[(\upsilon_1' +\upsilon_2')(W +\sigma  Z  )  \frac{\partial \sigma }{\partial   \tau^2 } Z \right] +  b \sum_{\nu=1}^k \alpha_\nu  \EE_{x_0,Z
} \left( f (\bar r_{\nu+1})  \frac{\partial \bar r_{\nu+1} }{\partial   \tau^2 }- f(\bar r_\nu)  \frac{\partial \bar r_{\nu} }{\partial   \tau^2 } \right)     \right]
\\
&= \underbrace{b \EE \left[    \upsilon_1^{''} (W +\sigma  Z; b )   Z^2  \left(\frac{\partial \sigma (\tau^2,\theta)}{\partial \tau^2}\right)^2 +  (\upsilon_1' +\upsilon_2')(W +\sigma  Z  )   Z    \frac{\partial^2 \sigma (\tau^2,\theta)}{\partial ( \tau^2)^2}  \right]  }_{T_1}
\\
&+
\underbrace{b \sum_{\nu=1}^k \alpha_\nu  \EE_{x_0,Z
} \left( f' (\bar r_{\nu+1})  \left(\frac{\partial \bar r_{\nu+1} }{\partial   \tau^2 } \right)^2- f'(\bar r_\nu) \left(\frac{\partial \bar r_{\nu } }{\partial   \tau^2 } \right)^2   +  f (\bar r_{\nu+1})  \frac{\partial^2 \bar r_{\nu+1} }{\partial   (\tau^2)^2 }- f(\bar r_\nu)  \frac{\partial^2 \bar r_{\nu} }{\partial   (\tau^2 )^2} \right)}_{T_2}
\end{align*}
 Next, we show that the above display is negative for all smooth $\rho$.
 Observe that for all smooth losses $\rho$, $T_2=0$ and otherwise $T_2 \neq 0$.
 Hence, for the smooth losses, it suffices to show that $T_1 \leq 0$.
  Condition {({\bf R})} provides that $\EE \upsilon_1''$ and $\upsilon_1'$ is negative. Furthermore, $Z$ has a symmetric density   and    $\sigma$ is concave \citep{BM11}; hence, $T_1 <0$.

  Let us know focus on non-smooth loss functions.
  As $f$ is a continuous density,  $f(\bar r_{\nu+1})  < f(\bar r_\nu)$ for all $\bar r_{\nu },\bar r_{\nu+1} \geq 0$ and $f(\bar r_{\nu+1})  > f(\bar r_\nu)$ otherwise. Moreover,  for symmetric densities $f'(\bar r_{\nu+1}) >0$ for all $\bar r_{\nu+1} <0$. Moreover,  $f'(\bar r_{\nu+1}) >f'(\bar r_{\nu }) $  for all $\bar r_{\nu+1} > \bar r_{\nu}$ and $\bar r_{\nu+1},\bar r_{\nu} <0$ . Opposite inequalities will hold on the positive axis with $f'(\bar r_{\nu+1}) <0$ for all $\bar r_{\nu+1} >0$. Additionally, as $\bar r_{\nu}$ is a proxy for a $Prox^{-1}$, it is concave with a negative second derivative ($Prox$ is a convex function). Therefore, the marginal derivative above is necessarily negative.  Hence the sign of $T_2$ will alternate between negative and positive.

Part(c).\\
For part (c), the result of \cite{BM12} provides that 
\[
\lim _{\tau \to \infty} \sigma'(\tau^2,\alpha \tau) = f(\alpha).
\]
Moreover, they show  that  $\sigma$ is strictly  concave for $\alpha >0$.  Hence, $\sigma$ will converge to some $\sigma_{\min}$ when $\tau \to \infty$.

Hence, ${\partial \mathbb{V}(\tau^2,b(\tau),\alpha(\tau))} $ will converge to 
$$ (\delta/\omega)^2f(\alpha)  \EE 2  \left[ \Phi (W +\sigma_{\min} Z;b) \partial_1\Phi (W +\sigma_{\min} Z;b) \right]\left[1- \frac{\partial _{11}\EE \left[\ \Phi (W + \sigma_{\min} Z;b) Z  \right]}{\partial _{21} \EE\left[ \Phi (W + \sigma_{\min} Z;b)\right]} \right], $$
with $\partial_1\Phi = \upsilon_1' + \upsilon_2'$ and $\partial _{11} \EE \Phi $ follows the same formula as $\frac{\partial^2 \mathbb{V}(\tau^2,b,\theta)}{\partial (\tau^2)^2} $ does (see above). Furthermore, $\partial _{21} \EE \Phi $ denotes $\partial_1$ of the function on the right hand side of \eqref{eq:b2}.
Moreover, \cite{BM12} show that $f(\alpha)$ is decreasing function of $\alpha$. Hence,  the above limit is as well, by observing that the remaining terms are independent of $\alpha$. 
%
%

%
\end{proof}

\begin{proof}[Proof of Lemma \ref{lem:delta}]
This proof relies on Lemma \ref{lem:lem22} and  a simple modification of Theorem 2 of \cite{BM11}. This theorem provides a state evolution equation  for a general recursion algorithm. As Lemma \ref{lem:lem22} establishes a connections between our algorithm and general recursion, the proof is then a simple application of Theorem 2 of \cite{BM11}, with a simple relaxation of its conditions.

Let $\bar \tau_t$ and $\bar \sigma_t$ be defined by recursion \eqref{tau}-\eqref{sigma}.
By Lemma \ref{lem:lem22} and with  $b_i^t$   defined therein \eqref{bb}, Theorem 2 of \cite{BM11} states 
\begin{align}\label{eq:clt}
\lim_{n \to \infty} \frac{1}{n} \sum_{i=1}^n \psi(b_i^t,W_i) = \EE \left[ \psi(\bar \sigma_tZ,W)\right]
\end{align}
for any pseudo-Lipschitz function $\psi:\RR^2\to \RR$ of order $k$ and for all $W_i$ with bounded $2k-2$ moments. Careful inspection of the proof of Lemma 5 of \cite{BM11}  shows that if $\psi$  is a function that is uniformly bounded, the restriction on the moments of $W_i$ is unnecessary.  A    version of Hoeffding's  inequality  suffices, as applied to   independent and not-necessarily equally distributed random variables (see Theorem 12.1 in \cite{B06}).

Next, we split the analysis into two cases: $  \Phi$ is differentiable and $\Phi$ is not differentiable. For the first case, it suffices to observe  
 that by Lemma \ref{lem:lem22} we have $b_i^t = W_i-z_i^t$, with $z_i^t$ defined in \eqref{alg:alg3}. Next, we choose $\psi$ to be 
\[
\psi(s,t) = \partial_1 \Phi(t-s;b).
\]
Then, $\psi(b_i^t,W_i) = \partial_1 \Phi(W_i -W_i+z_i^t;b)$ and by Condition ({\bf R}) $\psi$ is a uniformly bounded function. 
Thus, application of the result above provides
\begin{align*}\label{eq:clt2}
\lim_{n \to \infty} \frac{1}{n} \sum_{i=1}^n \partial_1 \Phi(z_i^t;b) = \EE \left[ \partial_1 \Phi( W- \bar \sigma_t Z;b)\right].
\end{align*}
The proof then follows by observing that the right hand side is equal to $\omega/\delta$ by \eqref{alg:alg2}.

Next, we discuss the case of  non-differentiable losses $\rho$. 
Let $h$ be a bandwidth parameter of an estimator of $\nu(b)$.
We define 
\[
S_n(h) =\sum_{i=1}^n \left[ \Phi(z_i^t+n^{-1/2}h;b) - \Phi(z_i^t;b)\right]
\]
for $h \in [0,C]$ for some constant $C$: $0 <C <\infty$. We set 
\[
S_n^o(h)=S_n(h)  - \EE S_n(h), \mbox{ for } h \in [0,C].
\]
Moreover, by Condition~$\mathbf{(R)}$ (i)
\[ 
\Phi(z_i^t;b) = b v_1(Prox(z_i^t,b)) +b v_2(Prox(z_i^t,b)).
\]
Absolutely-continuous term $\nu_1$ can be handled as the  above case; hence, without loss of generality we can assume it is equal to zero.  Hence,
\[
S_n (h) = \sum_{i=1}^n  \sum_{\nu=1}^{k-1}  b   \left[  \mathbbm{1}\left\{ Prox( z_i^t +n^{-1/2}h ,b) \in (r_\nu,r_{\nu+1}) \right\}  -  \mathbbm{1}\left\{ Prox( z_i^t   ,b) \in (r_\nu,r_{\nu+1}) \right\} \right]
\]
and 
\[
\EE S_n (h) = b \sum_{\nu=1}^{k-1} \left[  \PP \left\{ Prox(z_i^t+n^{-1/2}h,b) \in (r_\nu,r_{\nu+1}) \right\} -  \PP \left\{ Prox(z_i^t,b) \in (r_\nu,r_{\nu+1}) \right\}\right].
\]
As $ Prox(z,b)  = z-\Phi(z;b)$, we know the term above can be further written as 
\[
\EE S_n (h) = \sum_{i=1}^n b \sum_{\nu=1}^{k-1} \left[  \PP \left\{  z_i^t+n^{-1/2}h - \Phi(z_i^t+n^{-1/2}h,b) \in (r_\nu,r_{\nu+1}) \right\} -  \PP \left\{ z_i^t-\Phi(z_i^t,b) \in (r_\nu,r_{\nu+1}) \right\}\right].
\]
Then, by the same arguments as for \eqref{eq:clt} we 
obtain
\[
n^{-1/2} S_n^o(h) \to \mathcal{N}(0,\gamma^2(h)),  \mbox{ for }  h \in [0,C],
\]
in distribution,
where for each $h \in [0,C]$, and
\[
\gamma^2(h) =  b \sum_{\nu=1}^{k-1} \alpha_\nu  \lim_{n \to \infty} \frac{1}{n}  \sum_{i=1}^n h \left[ f_{z_i^t-\Phi(z_i^t,b)}  ( r_{\nu+1})  - f_{z_i^t-\Phi(z_i^t,b)}  ( r_{\nu }) \right].
\]
Then, by the arguments of \eqref{eq:clt} we conclude
\[
\gamma^2(h) =  b  h \sum_{\nu=1}^{k-1} \alpha_\nu \left[ f_{W - \bar \sigma_t Z}  ( r_{\nu+1})  - f_{W - \bar \sigma_t Z}  ( r_{\nu }) \right].
\]
The right hand side of the equality above is finite by the Condition~$\mathbf{(R)}$.
To establish a uniform statement, we need to establish the compactness or tightness of the sequence $n^{-1/2} S_n(h)$ for $h \in [0,C]$. This follows by noticing that the sequence is a sequence of differences of two, univariate, empirical distribution functions, both of which weakly converge to a Wiener function (see Lemma 5.5.1 in \cite{JS89}).
Hence,
\begin{equation}\label{eq:clt1}
\sup_{|h| \leq C}   n^{-1/2}\sum_{i=1}^n \left[ \Phi(z_i^t+h;b) - \Phi(z_i^t;b)+ h b \gamma^*\right]  =O_P(n^{-\tau})
\end{equation}
where $\tau=1/2$ for continuous $\psi$ and $\tau=1/4$ for discontinuous $\psi$. In the display above
$ \gamma^*=\sum_{\nu=1}^{k } \left( \alpha_\nu - \alpha_{\nu-1}\right)  f_{W - \bar \sigma_t Z}  ( r_{\nu}) $.
By the definition $\omega/\delta$ is the derivative of a consistent estimator of $\nu(b)=\partial_1 \EE \Phi(z^t;b_t)$. Because of  the equation above, we see that 
\[
\omega/\delta =b \sum_{\nu=1}^{k } \left( \alpha_\nu - \alpha_{\nu-1}\right)  f_{W - \bar \sigma_t Z}  ( r_{\nu}), 
\]
for all consistent estimators of $\nu(b)$ with a bandwidth choice of $h \to 0$ and $ nh \to \infty$.
\end{proof}

\subsection{Proofs for Section \ref{sec:amse}}

\begin{proof}[Proof of Theorem \ref{the:the2}]
The proof is split into two parts. In the first  step, we show that the proposed algorithm belongs to the class of generalized recursions as defined in \cite{BM11}. The result is presented in Lemma \ref{lem:lem22}.

In the second step, we utilize  conditioning technique  and the result of  Theorem 2 of \cite{BM11} designed for generalized recursions.
For an appropriate sequence of vectors  $h_i^t$  of generalized recursions and a $x_0$ the true regression coefficient, they show
\begin{equation}\label{lim:lim}
\underset{p\to\infty}{\operatorname{lim }}\frac{1}{p}\sum_{i=1}^{p}\psi(h_i^{t+1},x_{0,i})= \EE \left[\psi (\bar{\tau}^*_tZ,x_0) \right]
\end{equation}
for a pseudo-Lipschitz function $\psi$. 
We now proceed to identify $x^{t}$ for  a suitable $h_i^t$ of the proposed RAMP algorithm.
By definition of RAMP,
\begin{align}\label{x:x}
x^{t+1}&\overset{(i)}{=}\eta(x^t+A^TG(z^t;b_t);\theta) \overset{(ii)}{=}\eta(x_0-x_0+x^t+A^TG(z^t;b_t);\theta)\overset{(iii)}{=}\eta(x_0-h^{t+1}),
\end{align}
where equation (i) is because of the iteration RAMP,   the equation (ii) is plus and minus a same term and the equation (iii) is the special choice of $h^{t+1}$ in equation(\ref{h:h}).
Therefore, combining $x^t$ in equation (\ref{x:x}) and equation (\ref{lim:lim}), we obtain
\begin{align*}
\underset{p\to\infty}{\operatorname{lim }}\frac{1}{p}\sum_{i=1}^{p}\psi(x_i^t,x_{i,0})&=\underset{p\to\infty}{\operatorname{lim }}\frac{1}{p}\sum_{i=1}^{p}\psi(\eta(x_{0,i}-h_i^t;\theta),x_{0,i})=\EE [\psi(\eta(x_0-\bar{\tau}^*_tZ;\theta),x_0)]
\end{align*}
\end{proof}

\begin{proof}[Proof of Theorem \ref{the:the1}]
In order to prove this result we designed a series of Lemmas $4-15$ provided in the Appendix . The
main part of the proof is provided by the results of Lemma \ref{mainthem}. In the next steps we 
  apply Lemma \ref{mainthem} to  the specific choice of vectors $x = x^t$ and $r = |\hat{x}-x^t|$. We show there exist  constants $c_1,...,c_5 > 0$, such that  for each $\epsilon > 0$ and some iteration $t$,  Conditions $(C1)-(C6)$ of  Lemma \ref{mainthem} hold with probability going to 1 as $p \rightarrow \infty.$\\
  
Condition $(C1)$. We need to show $\||x^t-\hat{x}\||_2\leq c_1\sqrt{p}$.
Lemma \ref{lem:lem22}  proves that the RAMP algorithm  is a special case of a general iterative and recursive scheme, as defined in \cite{BM11}. From \eqref{lim:lim} we choose $\psi(a,b) = a^2$ and obtain 
\[
\underset{t\to \infty}{\operatorname{lim }}\underset{p\to\infty}{\operatorname{lim }}\frac{||x^t||^2}{p} = \EE \{\eta(x_0 + \bar{\tau}^* Z;\theta^*)\}^2< \infty.
\]
 Moreover, we observe that
$
\frac{||\hat{x}||_2}{p} < \infty
$
by assumptions of the Theorem. 

Condition $(C2)$. By the definition of $\hat x$ as the minimizer of the  $\mathcal{L}$, we conclude that $\mathcal{L}(\hat{x}) < \mathcal{L}(x )$ for any $x \neq \hat x$  and this applies for $x=x^t$.  

Condition $(C3)$. We need to show $||sg(\mathcal{L},x^t)||_2\leq\epsilon\sqrt{p}$. By the definition of the RAMP iteration
\begin{equation*}
x^t=\left\{\begin{array}{cc}
A^T G(z^{t-1};b_{t-1}) + x^{t-1} + \theta_{t-1},&\mbox{if } A^T G(z^{t-1};b_{t-1}) + x^{t-1} \geq\theta_{t-1}\\
A^TG(z^{t-1};b_{t-1}) + x^{t-1} - \theta_{t-1},& \mbox{if } A^T G(z^{t-1};b_{t-1}) + x^{t-1} \leq-\theta_{t-1}\\
0& \mbox{otherwise}
\end{array}
\right..
\end{equation*}
This indicates that  when $x^t = 0$
\[
\frac{|A^T G(z^{t-1};b_{t-1}) + x^{t-1}|}{ \theta_{t-1}}\leq 1,
\]
and that in cases of $x^t \neq 0$
\[
A^TG(z^{t-1};b_{t-1})+x^{t-1} = x^t+\mbox{sign}(x^t)\theta_{t-1}.
\]
Therefore, the subgradient $sg(\mathcal{L}, x^t) $ must satisfy
\begin{equation}\label{eq:sub1}
sg(\mathcal{L}, x^t) \equiv \left\{\begin{array}{cc}
\lambda \mbox{sign}(x^t) - A^T\rho'(Y-Ax^t),& \operatorname{if} x^t \neq 0\\
\lambda \frac{A^T G(z^{t-1};b_{t-1}) + x^{t-1}}{\theta_{t-1}} - A^T\rho'(Y-Ax^t),& \operatorname{if} x^t = 0
\end{array}
\right.  .
\end{equation}
Moreover, by equation \eqref{alg:alg3} and Lemma 1
\[
Y-Ax^t = z^t - \Phi(z^{t-1};b_{t-1}).
\]
Then,  
\begin{eqnarray}
A^T\rho'(Y-Ax^t) &=& A^T\rho'(z^t - \Phi(z^{t-1};b_{t-1})) \nonumber
\\ \label{eq:temp25}
& =& A^T\rho'(Prox(z^t,b_t) + \Phi(z^t;b_t)-\Phi(z^{t-1};b_{t-1})),
\end{eqnarray}
where we used the fact that $z^t - \Phi(z^{t-1};b_{t-1}) = Prox(z^t,b_t)$.
Adding equation \eqref{eq:temp25} to \eqref{eq:sub1} and the expression of $sg(\mathcal{L}, x^t)$, we conclude
\begin{equation}
sg(\mathcal{L}, x^t) = \lambda s^t - A^T\rho'(Prox(z^t,b_t) + \Phi(z^t;b_t)-\Phi(z^{t-1};b_{t-1})),
\end{equation}
where 
\begin{equation*}
s^t = \left\{\begin{array}{cc}
\mbox{sign}(x^t) &, \mbox{if } x^t \neq0\\
\frac{A^TG(z^{t-1};b_{t-1})+x^{t-1}}{\theta_{t-1}} &, \mbox{if }  x^t = 0
\end{array}
\right. .
\end{equation*}
Now, we rewrite $sg(\mathcal{L}, x^t)$ as follows
\begin{eqnarray*}
sg(\mathcal{L}, x^t)
& =&\frac{1}{\theta_{t-1}} \left[\lambda\theta_{t-1}s^t - \frac{\lambda\delta b}{\omega}A^T\rho'(Prox(z^t,b_t) + \Phi(z^t;b_t)-\Phi(z^{t-1};b_{t-1}))\right]
 \\
 &
 +&\frac{1}{\theta_{t-1}} \left[ \frac{\lambda\delta b}{\omega}-\theta_{t-1}\right]A^T\rho'(Prox(z^t,b_t) + \Phi(z^t;b_t)-\Phi(z^{t-1};b_{t-1})).
\end{eqnarray*}
Then, by the non-negativity of $\theta_{t-1}$ and triangular inequality
\begin{eqnarray*}
\frac{1}{\sqrt{p}} \left\| sg(\mathcal{L},x^t) \right\|_2
&\leq&  \underbrace{ \frac{\lambda}{\theta_{t-1}\sqrt{p}} \left\| \lambda\theta_{t-1}s^t - \frac{\lambda\delta b}{\omega}A^T\rho'(Prox(z^t,b_t) + \Phi(z^t;b_t)-\Phi(z^{t-1};b_{t-1}))  \right\|_2 } _{A}
\\
&&
+
\underbrace{\frac{|\frac{\lambda\delta b}{\omega}-\theta_{t-1}|}{\theta_{t-1}}\left\| A^T\rho'(Prox(z^t,b_t) + \Phi(z^t;b_t)-\Phi(z^{t-1};b_{t-1}))\right\|_2}_{B}. 
\end{eqnarray*}
We consider the bound of B first. 

Observe that $Prox(z^t,b_t)  = z^t - \Phi(z^t;b_t)$. Then, utilizing \eqref{eq:clt} and \eqref{eq:clt1}, we observe that there exists a $0<q<\sqrt{p}$ such that $\|Prox(z^t,b_t) + \Phi(z^t;b_t)-\Phi(z^{t-1};b_{t-1}) \|_2 \leq q$.
We define $M \equiv \underset{\|z\|_2 \leq q}{\operatorname{sup }} v(z)$, where 
$v(z) = v_1'(z) + v_2'(z)$. Then, by Condition ({\bf R}) (i)-(ii) and (iv) we know that $M <\infty$.
 Then, by Taylor expansion and Triangle inequality, we conclude
\begin{align*}
B &\leq\frac{|\frac{\lambda\delta b}{\omega} -\theta_{t-1}|}{\theta_{t-1}} \frac{1}{\sqrt{p}} \biggl[
\|
A^T\rho'(Prox(z^t,b_t))\|_2+ M \|A\|_2 \| \Phi(z^t;b_t)-\Phi(z^{t-1};b_{t-1})\|_2 \biggl].
\end{align*}
Moreover,
\begin{eqnarray*}
\frac{|\frac{\lambda\delta b}{\omega} -\theta_{t-1}|}{\theta_{t-1}} \frac{1}{\sqrt{p}} ||A^T\rho'(Prox(z^t,b_t))||_2 &=&\frac{|\frac{\lambda\delta b}{\omega} -\theta_{t-1}|}{\theta_{t-1}b_t} \frac{1}{\sqrt{p}} ||A^T\Phi(z^t;b_t)||_2 \\
&\leq& \frac{|\frac{\lambda\delta b}{\omega} -\theta_{t-1}|}{\theta_{t-1}b_t} \frac{1}{\sqrt{p}} \sigma_{max}(A)||\Phi(z^t;b_t)||_2 
\end{eqnarray*}
Next, we observe that the   state evolution (by Lemma 3 and Theorem 1) guarantees, 
\[
\underset{p\to \infty}{\operatorname{lim }}\frac{||\Phi(z^t;b_t)||_2}{p} < \infty.
\]
Moreover, $\sigma_{max}(A)$ is almost surely bounded as $p\to \infty$ \citep{BY93}. Hence, we conclude
$$\underset{t\to \infty}{\operatorname{lim }}\underset{p\to \infty}{\operatorname{lim }} \frac{\sigma_{max}(A)\| \Phi(z^t;b_t) - \Phi(z)\|_2}{\sqrt{p}} = 0. $$ 
Furthermore, using Lemma \ref{kkt} we obtain
\[
\underset{t\to \infty}{\operatorname{lim }}\underset{p\to \infty}{\operatorname{lim }}\frac{|\frac{\lambda\delta b_t}{\omega} -\theta_{t-1}|}{\theta_{t-1}b_t} = \frac{\frac{\lambda \delta b^*}{\omega}-\theta^*}{\theta^*b^*} = 0.
\]
Therefore, B converges to 0 when $p\to \infty $.\\

Now we consider A. From equation  \eqref{eq:sub1}, we conclude that 
\[
\theta_{t-1}s^t - \frac{\delta b^t}{\omega}A^T\rho'(Prox(z^t,b_t)) = \theta_{t-1}s^t s^t - \frac{\delta }{\omega}A^T\Phi(z^t;b_t) = x^t - x^{t-1}.
\]
Plugging into A, we obtain
\begin{eqnarray*}
A &\leq& \frac{\lambda}{\theta_{t-1}\sqrt{p}} \| x^t-x^{t-1} \|_2 + \frac{\lambda \delta b_t}{\omega \theta_{t-1}\sqrt{p}}M\sigma_{max}(A)  \left\|  \Phi(z^t;b_t)-\Phi(z^{t-1};b_{t-1}) \right\|_2.
\end{eqnarray*}
The convergence of the second term is by the convergence of the term B  and the first term is  converging to $0$ by the convergence of the  RAMP algorithm -- that is  the result of Theorem \ref{the:the2} holds.  Therefore, A converges to 0 when $p \to \infty$. This finishes the proof of Condition $(C3)$. 

Condition $(C4)$. This result follows from Lemma \ref{lem:cond4} provided in the Appendix. 

Condition $(C5)$.  Let $A\in \RR^{n\times p}$ be a matrix with i.i.d. entries such that $\EE \{A_{ij}\} = 0$, $\EE \{A_{ij}^2\} = 1/n$, and $n = p\delta.$ Let $\sigma_{max}(A)$ be the largest singular value of A and $\sigma_{min}(A)$ be its smallest non-zero singular value. Then,   \cite{BY93} provide a general result that claims
\begin{align}
\underset{p\to \infty}{\operatorname{lim }} \sigma_{max}(A) = \frac{1+\sqrt{\delta}}{\sqrt{\delta}},
\qquad
\underset{p\to \infty}{\operatorname{lim }}\sigma_{min}(A) = \frac{1-\sqrt{\delta}}{\sqrt{\delta}}.
\end{align}

Condition $(C6)$. Assumption $\bf(R)$ is guaranteeing the validity of $(C6)$.

  Conditions $(C1)$-$(C6)$ are checked and the proof is completed.
\end{proof}


\subsection{Proofs for Section \ref{sec:RE}}

 \begin{proof}[Proof of Theorem \ref{thm:information}]
 For shorter statements, in the   proof of statements (i) and (ii) we use abbreviated notation $\Phi$ for the bivariate function $\Phi (z;b)$. We deviate from this notation in the proof of statement (iii) where $\Phi$ denotes cumulative distribution function of standard normal.
Let $I(F_W)$ be a well defined information matrix of the errors, $W_i$.
If the distribution of the errors $W_i$ is a convolution $D=F_W \circ N(0,\sigma^2)$, then,
\[
\EE_D \Phi'=\omega/\delta.
\]
Observe that $\Phi'$ should be interpreted according to the Lemma \ref{lem:delta}.
 Let the score function for the location of $D$ be denoted with $L_D$. Then, the information matrix of $D$ can be represented as $I(D) = \EE[L_D^2]$ and $\EE_D \Phi' =\EE_G \Phi L_D$.  In turn, simple Cauchy-Swartz inequality provides 
 \[
\tau_t^2=\frac{\omega}{\delta} \EE_G \Phi^2 \geq \frac{\omega}{\delta} \frac{|\EE_G \Phi L_D|^2 }{\EE[L_D^2]} = \frac{\omega}{\delta} \frac{(\EE_D \Phi')^2}{I(D)} = \frac{\omega}{\delta}  \frac{1}{I( F_W \circ N(0,\sigma_t^2))}.
 \]
By Lemma 3.5 of \cite{DM13}, the lower bound can be further reduced to 
\[
\tau_t^2\geq \frac{\omega}{\delta}  \frac{1 + \sigma_t^2 I(F_W) }{I( F_W ) }.
\]
The proof is finalized by obtaining a lower bound of $\sigma_t$.
\[
{\sigma}_t^2  = \frac{1}{\delta} \mathbb{E}\left[\eta(x_0+\tau_{t-1} Z,\theta)-x_0)\right]^2.
\]
For $\theta=\alpha \tau_{t-1}$, Proposition 1.3 of \cite{BM12} shows that $\sigma_t^2$ is a strictly concave function for $\alpha > \alpha_{\min} >0$ and $x_0 \neq 0$ and an increasing  function of  $\tau^2$. 
Hence, $\sigma_t^2 > \tau_{t-1}^2$ for small $\tau_{t-1}^2$ and  $\sigma_t^2 < \tau_{t-1}^2$ for large $\tau_{t-1}^2$.
Hence,
\[
\tau_t^2\geq \frac{\omega}{\delta}  \frac{1 + \tau_{t-1}^2 I(F_W) }{I( F_W ) } \geq \frac{\omega}{\delta}  \frac{1 + \omega/\delta}{I( F_W ) }.
\]
Iterating previous equation $k$ times, we obtain that for $t >k$
\[
\tau_t^2\geq \frac{\omega}{\delta}  \frac{1 + \omega/\delta + (\omega/\delta)^2 + \cdots + (\omega/\delta)^k}{I( F_W ) }.
\]
When $k \to \infty$, $\tau_t^2 \to {\tau^*}^2$, we obtain
\[
 {\tau^*}^2 \geq  \frac{\omega/\delta}{1 - \omega/\delta}\frac{ 1}{I( F_W ) }  = \frac{s}{n-s}\frac{ 1}{I( F_W ) }. 
\]

Part $(iii)$.
Utilizing the scale-invariance property  of the soft-thresholding function $\eta$, we obtain that 
\begin{align*}
\nu_1(\tau) &=\alpha^2 \PP \left( \left| Z+ \frac{x_0}{\tau}\right| \geq \alpha  \right) -2\alpha \EE\left[ Z \mbox{sign}(Z) \mathbbm{1} \left\{ \left| Z+ \frac{x_0}{\tau}\right| \geq \alpha \right\}\right] + \EE \left[ Z^2 \mathbbm{1} \left\{ \left| Z+ \frac{x_0}{\tau}\right| \geq \alpha \right\}\right],
\\
\nu_2(\tau) &=  \EE\left[ {x_0^2}  \mathbbm{1} \left\{ \left| Z+ \frac{x_0}{\tau}\right| \leq \alpha \right\} \right].
\end{align*}
Let us first focus on the second component, i.e., $\nu_2(\tau)$.  The derivative of $\nu_2(\tau)$ is 
\[
\frac{\partial \nu_2(\tau) }{ \partial \tau} = \EE_{x_0} \left[ \frac{x_0^3}{\tau^2} \left(\phi( \alpha -\frac{x_0}{\tau}) -  \phi(-\alpha -\frac{x_0}{\tau}) \right)\right].
\]
By observing that the last term on the RHS is non-negative for all $x_0 >0$ and negative for all $x_0 <0$, we conclude that $\nu_2(\tau)$ is an increasing function.

We conclude the proof with the analysis of the first term, $\nu_1(\tau)$. The displays  above imply that 
the first and the last term  of $\nu_1(\tau)$ together lead to $ \EE \left[  (Z^2 + \alpha^2) \mathbbm{1} \left\{ \left| Z+ \frac{x_0}{\tau}\right| \geq \alpha \right\}\right]$, whereas the middle term can be written as 
 \[
  2 \alpha \EE \left[ Z \mathbbm{1} \left\{  Z \leq \alpha -   \frac{x_0}{\tau} \right\}\right] + 2\alpha\EE \left[ Z \mathbbm{1} \left\{  Z \leq -\alpha -   \frac{x_0}{\tau} \right\}\right].
\]
By Stein's lemma we know  that the previous expression is equal to 
$
  2 \alpha  \EE_{x_0} \left[  \phi( \alpha -   \frac{x_0}{\tau}) - \phi(-\alpha -   \frac{x_0}{\tau})   \right] .
$
Furthermore, utilizing the  variance  computation  of a truncated random variable, 
 conditional on $x_0$,  it is easy to check that 
 \[
 \EE \left[ Z^2 \mathbbm{1} \left\{ Z+ \frac{x_0}{\tau}  \leq \alpha \right\}\right] =  1 - \left( \alpha -   \frac{x_0}{\tau}\right) \phi\left( \alpha -   \frac{x_0}{\tau}\right).
 \]
The rest of terms can be computed similarly. Combining all of  the above we obtain 
\begin{align*}
\nu_1(\tau) 
&= \EE_{x_0} \left[  \alpha^2 +1 - \alpha^2  \left[ \Phi(\alpha -   \frac{x_0}{\tau})  - \Phi(-\alpha -   \frac{x_0}{\tau})\right]  \right. \\ 
&  \left. -   \alpha    \left[  \phi( \alpha -   \frac{x_0}{\tau}) + \phi(-\alpha -   \frac{x_0}{\tau})   \right]  - \frac{x_0}{\tau} \left[  \phi( \alpha -   \frac{x_0}{\tau}) - \phi(-\alpha -   \frac{x_0}{\tau})   \right] \right].
\end{align*}
Evaluating the derivative of $\nu_1(\tau)$, we obtain 
\[
\frac{\partial \nu_1(\tau) }{ \partial \tau} = \EE_{x_0} \left[ \frac{x_0}{\tau^2} \left( 1- \frac{x_0^2}{\tau^2}\right) \left( \phi(\alpha - \frac{x_0}{\tau}) - \phi(- \alpha - \frac{x_0}{\tau}) \right) \right].
\]
Hence, for small $\tau^2$ the expression above is negative and for large values of $\tau^2$ it is positive. It follows that, $\nu_1(\tau)$ is a convex function of $\tau^2$.

\end{proof}

\begin{proof}[Proof of Lemma \ref{lem:prep1}]
 
Notice that in sparse, high dimensional setting,  the distribution of the $x_0$ can be represented as a convex combination of the Dirac measure at 0 and a measure that doesn't have mass at zero. Let us denote with $\Delta$ and $U$ two random variables, each having the two measures above. 
Let 
\[
 \Psi_{\alpha}(\tau ) = \frac{1- \omega}{\delta} \EE \eta^2 (Z ,\alpha) + \frac{\omega}{\delta} 
\EE  \left[ \eta\left(\frac{U}{\tau }+Z;\alpha\right)- \frac{U}{\tau }\right]^2. \]

First, 
we prove that whenever $\sigma_{\mbox{\tiny $W$}}^2 \to 0$ then $\tau_{\mbox{\tiny P-LAD}}^2 \to 0$ as long as $\lim_{\tau \to 0}\Psi_\alpha(\tau) \neq 0$.
To accomplish this, let's prove that $\lim_{\tau \to 0}\Psi_\alpha(\tau) \neq 0$ and  look at the relationship between $\tau_{\mbox{\tiny P-LAD}}$ and $\sigma_{\mbox{\tiny $W$}}$.

Notice that by the result of Theorem 4 of \cite{Z15}, we conclude
\[
\lim_{\tau \to 0}\Psi_\alpha(\tau) = \frac{\omega}{\delta},
\]
which is different from $0$ whenever $s \neq 0$.

Observe that whenever $\sigma_{\mbox{\tiny $W$}}^2 \to 0$, it holds that  $Y \to \sigma_{\mbox{\tiny P-LAD}}^2 Z$ and $f (\mbox{\footnotesize $W$};  \tau^2_{\mbox{\tiny P-LAD}}) \to 0$. In this case 
\begin{equation}\label{eq:t4}
\tau_{\mbox{\tiny P-LAD}}^2  \left( 1- \Psi_\alpha^{-1}   \tilde g(\tau^2_{\mbox{\tiny P-LAD}}) \right)  =b^2 \frac{\PP \left( \left|  \tau_{\mbox{\tiny P-LAD}}^2 \Psi_\alpha Z \right| >b  \right)}{ \PP^2 \left(  \left| \tau_{\mbox{\tiny P-LAD}}^2 \Psi_\alpha Z\right| \leq b \right)},
\end{equation}
where 
\[
\tilde g(\tau^2_{\mbox{\tiny P-LAD}}) = \EE_Z \left[ Z^2 \left(F_W (b - \tau_{\mbox{\tiny P-LAD}} \Psi_\alpha^{1/2}  Z) - F_W(-b - \tau_{\mbox{\tiny P-LAD}}  \Psi_\alpha^{1/2} Z ) \right)\right] /\PP^2 \left(  \left| \tau_{\mbox{\tiny P-LAD}}^2 \Psi_\alpha Z\right| \leq b \right).
\]
Hence,
 $\tilde g(0) =\EE_Z \left[ Z^2 \left(F_W (b  ) - F_W(-b   ) \right)\right]   = F_W (b  ) - F_W(-b   )  < \infty$.
In turn, by plugging in $\tau_{\mbox{\tiny P-LAD}}  =0$ it satisfies both sides of the equation \eqref{eq:t4}.
\end{proof}

\begin{proof} [Proof of Lemma \ref{lem:Plad}]

Notice that in sparse, high dimensional setting,  the distribution of the $x_0$ can be represented as a convex combination of the Dirac measure at 0 and a measure that doesn't have mass at zero. Let us denote with $\Delta$ and $U$ two random variables, each having the two measures above. 
Let 
\[
\Psi_{\alpha} =\Psi_{\alpha}(\tau_{\mbox{\tiny P-LAD}}) = \frac{1- \omega}{\delta} \EE \eta^2 (Z ,\alpha) + \frac{\omega}{\delta} 
\EE  \left[ \eta\left(\frac{U}{\tau_{\mbox{\tiny P-LAD}}}+Z;\alpha\right)- \frac{U}{\tau_{\mbox{\tiny P-LAD}}}\right]^2. \]
We first discuss the P-LAD estimator.
By the state-evolution recursion, \eqref{sigma} 
\begin{equation}\label{eq:t3}
\tau_{\mbox{\tiny P-LAD}}^2 \Psi_\alpha =\sigma_{\mbox{\tiny P-LAD}}^2.
\end{equation}
Let $Y = W + \sigma_{\mbox{\tiny P-LAD}}^2 Z$. According to \eqref{eq:tau-tau},
\begin{equation}\label{eq:t1}
\tau_{\mbox{\tiny P-LAD}}^2 = \frac{\EE[Y^2 \mathbbm{1}{|Y| \leq b}] + b^2 \PP(|Y| >b)}{ \PP^2 (|Y| \leq b)}.
\end{equation}

Next observe that 
$
\EE[Y^2 \mathbbm{1}{|Y| \leq b}]  = \EE [W^2\mathbbm{1}{|Y| \leq b}] +\sigma_{\mbox{\tiny P-LAD}}^2\EE[Z \mathbbm{1}{|Y| \leq b}] 
$; 
 moreover,
$
\EE[Y^2 \mathbbm{1}{|Y| \leq b}]  =  \sigma_{\mbox{\tiny $W$}}^2 - \EE[ W^2\mathbbm{1}{|Y|  > b}] +\sigma_{\mbox{\tiny P-LAD}}^2\EE[Z \mathbbm{1}{|Y| \leq b}] .
$
Plugging into \eqref{eq:t1} we obtain
\begin{equation}\label{eq:t2}
\tau_{\mbox{\tiny P-LAD}}^2 = \sigma_{\mbox{\tiny P-LAD}}^2 \frac{\EE[Z \mathbbm{1}{|Y| \leq b}] }{ \PP^2 (|Y| \leq b)} + \frac{\sigma_{\mbox{\tiny $W$}}^2 - \EE[ W^2\mathbbm{1}{|Y|  > b}] }{ \PP^2 (|Y| \leq b)} + \xi(b)
\end{equation}
for 
\[
\xi(b) = b^2 \frac{\PP \left( \left| W +\sigma_{\mbox{\tiny P-LAD}}^2 Z \right| >b  \right)}{ \PP^2 \left(  \left| W +\sigma_{\mbox{\tiny P-LAD}}^2 Z\right| \leq b \right)}.
\]
Let 
\[
g(\tau^2_{\mbox{\tiny P-LAD}}) =\frac{\EE[Z \mathbbm{1}{|Y| \leq b}] }{ \PP^2 (|Y| \leq b)}, \qquad f (\mbox{\footnotesize $W$};  \tau^2_{\mbox{\tiny P-LAD}})  = \frac{\sigma_{\mbox{\tiny $W$}}^2 - \EE[ W^2\mathbbm{1}{|Y|  > b}] }{ \PP^2 (|Y| \leq b)},
\]
then
\begin{equation}\label{eq:t4}
\tau_{\mbox{\tiny P-LAD}}^2 = \sigma_{\mbox{\tiny P-LAD}}^2 g(\tau^2_{\mbox{\tiny P-LAD}})  +f (\mbox{\footnotesize $W$};  \tau^2_{\mbox{\tiny P-LAD}}) + \xi(b).
\end{equation}
Substituting \eqref{eq:t4} in \eqref{eq:t3} we obtain
\begin{equation}\label{eq:t5}
\frac{\sigma_{\mbox{\tiny P-LAD}}^2}{\sigma_{\mbox{\tiny $W$}}^2} = \frac{\Psi_{\alpha}}{1- g(\tau^2_{\mbox{\tiny P-LAD}}) \Psi_{\alpha}} \left[ \frac{f (\mbox{\footnotesize $W$};  \tau^2_{\mbox{\tiny P-LAD}}) }{\sigma_{\mbox{\tiny $W$}}^2}  + \frac{\xi(b) }{\sigma_{\mbox{\tiny $W$}}^2} \right].
\end{equation}

By Stein's lemma and some algebra we arrive at the representation of $g(\tau^2_{\mbox{\tiny P-LAD}}) $ and $f (\mbox{\footnotesize $W$};  \tau^2_{\mbox{\tiny P-LAD}}) $, as 
\begin{align*}
g(\tau^2_{\mbox{\tiny P-LAD}}) &= \EE_Z \left[ Z^2 \left(F_W (b - \sigma_{\mbox{\tiny P-LAD}}Z) - F_W(-b - \sigma_{\mbox{\tiny P-LAD}}Z ) \right)\right] / \PP^2 (|Y| \leq b),
\\
f (\mbox{\footnotesize $W$};  \tau^2_{\mbox{\tiny P-LAD}})  &= \EE_{W} \left[ W^2 \left(\Phi \left(\frac{b-W}{\sigma_{\mbox{\tiny P-LAD}}}\right)-\Phi\left(\frac{-b-W}{\sigma_{\mbox{\tiny P-LAD}}}\right) \right)\right]/ \PP^2 (|Y| \leq b).
\end{align*}

Let us first focus on the case of $\sigma_{\mbox{\tiny $W$}}^2 \to 0$.  By Lemma \ref{lem:prep1} we conclude that $\tau_{\mbox{\tiny P-LAD}}^2 \to 0$ and $\sigma_{\mbox{\tiny P-LAD}}^2 \to 0$. Hence,
\[
\lim _{\sigma_{\mbox{\tiny $W$}}^2 \to 0}  \frac{\sigma_{\mbox{\tiny P-LAD}}^2}{\sigma_{\mbox{\tiny $W$}}^2}  
=
\lim _{\tau \to 0, \sigma_{\mbox{\tiny $W$}}^2 \to 0} \frac{\Psi_{\alpha}(\tau)}{1- g(\tau ) \Psi_{\alpha}(\tau)} \left[ \frac{f (\mbox{\footnotesize $W$};  \tau ) }{\sigma_{\mbox{\tiny $W$}}^2}  + \frac{\xi(b) }{\sigma_{\mbox{\tiny $W$}}^2} \right].
\]
We proceed to show that the last term in the display above is converging to $\infty$.
Observe that whenever $\sigma_{\mbox{\tiny $W$}}^2 \to 0$,  it holds that $Y \to \sigma_{\mbox{\tiny P-LAD}}^2 Z$  and 
\[
\xi(b) \to  b^2 \frac{\PP \left( \left|  \tau ^2 \Psi_\alpha(\tau)  Z \right| >b  \right)}{ \PP^2 \left(  \left| \tau ^2 \Psi_\alpha(\tau)  Z\right| \leq b \right)}.
\]
Furthermore, with $\sigma_{\mbox{\tiny P-LAD}} \to 0$ and $b>0$, it holds that $\xi(b)\to 0$. For $\phi$ denoting the density of the standard normal, the application of  Lohpital's rules guarantees 
\begin{align*}
 \lim _{\sigma \to 0, \sigma_{\mbox{\tiny $W$}}^2 \to 0}\frac{  \xi(b)   }{\sigma_{\mbox{\tiny $W$}}^2 }
&= b^2
\lim _{\sigma \to 0, \sigma_{\mbox{\tiny $W$}}^2 \to 0} \frac{  \phi (b /  \sigma) \sigma^{-2} + \phi (-b /  \sigma)  \sigma^{-2}  }{4\sigma_{\mbox{\tiny $W$}} },
\end{align*}
which implies  $\frac{\xi(b) }{\sigma_{\mbox{\tiny $W$}}^2}  \to \infty$ as $\sigma_{\mbox{\tiny $W$}} \to 0$.
%
%
%
%

We finish the proof by discussing the P-LS estimator. By Lemma \ref{lem:lem22} we see that the special case of the RAMP algorithm, when the loss function $\rho(x) =(x)^2$  is the approximate message passing algorithm of \cite{BM12}. Hence, results that apply to the algorithm in \cite{BM12} apply. In particular, a recent work \cite{Z15} discusses the properties of $\lim _{ \sigma_{\mbox{\tiny $W$}}^2 \to 0} \frac{\bar\sigma_{\mbox{\tiny P-LS}}^2}{\sigma_{\mbox{\tiny $W$}}^2}$ in their Theorem 7.  

\end{proof}

\begin{proof}[Proof of Lemma \ref{lem:Plad2}]

We will use the notation defined in the proof of Lemma \ref{lem:Plad}.
We first discuss the Penalized LAD estimator. 
Based on the representation proved in Lemma \ref{lem:Plad}
\[
\lim _{\sigma_{\mbox{\tiny $W$}}^2 \to \infty}  \frac{\sigma_{\mbox{\tiny P-LAD}}^2}{\sigma_{\mbox{\tiny $W$}}^2}  
=
\lim _{\tau \to \infty, \sigma_{\mbox{\tiny $W$}}^2 \to \infty} \frac{\Psi_{\alpha}(\tau)}{1- g(\tau ) \Psi_{\alpha}(\tau)} \left[ \frac{f (\mbox{\footnotesize $W$};  \tau ) }{\sigma_{\mbox{\tiny $W$}}^2}  + \frac{\xi(b) }{\sigma_{\mbox{\tiny $W$}}^2} \right].
\]
It suffices to discuss the limiting properties of the first, second and the third term in the right hand side above.
Let us discuss the last term first. 
Observe that we can rewrite 
\begin{align*}
\lim _{\tau \to \infty, \sigma_{\mbox{\tiny $W$}}^2 \to \infty} \frac{\xi(b) }{\sigma_{\mbox{\tiny $W$}}^2} 
&=
\lim _{\sigma \to \infty} \frac{\xi(b) }{\sigma^2}
=
\lim _{\sigma \to \infty} \frac{b^2 \frac{\PP \left( \left| W +\sigma_{\mbox{\tiny P-LAD}}^2 Z \right| >b  \right)}{ \PP^2 \left(  \left| W +\sigma_{\mbox{\tiny P-LAD}}^2 Z\right| \leq b \right)} }{ \frac{\EE \Phi^2 ( W +\sigma_{\mbox{\tiny P-LAD}}^2 Z ;b)}{ \PP^2 \left(  \left| W +\sigma_{\mbox{\tiny P-LAD}}^2 Z\right| \leq b \right)}}
\\
&=
b^2 \lim _{\sigma \to \infty} \frac{b^2 {\PP \left( \left| W +\sigma_{\mbox{\tiny P-LAD}}^2 Z \right| >b  \right)} }{ {\EE \Phi^2 ( W +\sigma_{\mbox{\tiny P-LAD}}^2 Z ;b)}  } =1,
\end{align*}
where in the last step we used the fact that when $\tau \to \infty$, $W +\sigma_{\mbox{\tiny P-LAD}}^2 Z  \to \infty$
\[
{\EE \Phi^2 ( \infty;b)} = b^2 \lim_{\sigma \to \infty}\mathbbm{1} \{ W +\sigma^2 Z  \geq b\} =b^2.
\]

Next, we discuss the limit of $\Psi_\alpha(\tau)$.  Corollary 6 of \cite{Z15} guarantees that $\lim_{\tau \to \infty} \Psi_\alpha(\tau) = \EE \eta^2(Z ; \alpha)/\delta$, that is, $\Psi_{\alpha}(\infty) = \Gamma/\delta$.

In the following, we analyze the limit of   
\[
g(\tau) =\frac{ \EE_Z \left[ Z^2 F_W (b -\sigma Z) - Z^2 F_W(-b - \sigma Z)\right] }{ \PP^2 \left( |W + \sigma Z |\leq b\right)}
\]
as $\tau \to \infty$.
In view of the fact that, both the numerator and denominator of $g(\tau)$
converge to $0$ when $\tau \to \infty$, we   use the L'H\~{o}pital's rule in determining its limit.
Therefore, 
\[
\lim_{\tau \to \infty} g(\tau) 
=
\lim_{\tau \to \infty}  \frac{ \EE_Z \left[ -Z^3 f_W (b -\sigma Z) + Z^3 f_W(-b - \sigma Z)\right] }{2 \PP \left( W + \sigma Z \leq b\right) (F_W(b-\sigma Z) +1 - F_W(-b-\sigma Z) )}.
\]
Moreover, the last expression still needs L'H\~{o}pital's rule. Hence,
\[
\lim_{\tau \to \infty} g(\tau) 
=
\lim_{\tau \to \infty}  \frac{ \EE_Z \left[  Z^4 f_W' (b -\sigma Z) - Z^4 f_W'(-b - \sigma Z)\right] }{2  (F_W(b-\sigma Z) +1 - F_W(-b-\sigma Z) )}=0.
\]
The proof is finalized by the analysis of  $ {f (\mbox{\footnotesize $W$};  \tau ) }/{\sigma_{\mbox{\tiny $W$}}^2}$, when $\sigma_{\mbox{\tiny $W$}}^2 \to \infty$ and $\tau \to \infty$.
We begin with the following representation of ${f (\mbox{\footnotesize $W$};  \tau ) }$,
\[
 {f (\mbox{\footnotesize $W$};  \tau ) }=\frac{\EE _ W \left[ W^2 \Phi_Z(\frac{b-W}{\sigma}) + W^2  - W^2\Phi_Z(\frac{-b-W}{\sigma}) \right]  } { \PP^2 \left( |W + \sigma Z |\leq b\right)}.
\]
We observe that in the limit  when $\tau \to \infty$, of the above expression  takes the form $0/0$; hence, we apply the L'H\~{o}pital's rule to obtain
\begin{align*}
&\lim_{\sigma \to \infty, \sigma_{\mbox{\tiny $W$}}^2 \to \infty}  \frac{f (\mbox{\footnotesize $W$};  \tau ) } {\sigma_{\mbox{\tiny $W$}}^2}\\
&=\lim_{\sigma \to \infty, \sigma_{\mbox{\tiny $W$}}^2 \to \infty} 
\frac{  \left( \EE _ W \left[ -W^2 \phi_Z(\frac{b-W}{\sigma}) (b-W)/\sigma^2  - W^2\phi_Z(\frac{-b-W}{\sigma} ) (b+W) /\sigma^2 \right] \right) / 2 \sigma_{\mbox{\tiny $W$}}}{2 \PP \left(| W + \sigma Z |\leq b\right)  \EE_W \left[-\phi_Z(\frac{b-W}{\sigma}) (b-W) /\sigma^2  - \phi_Z(\frac{-b-W}{\sigma})  (b+W)/ \sigma^2\right]}
\\
&=\lim_{\sigma \to \infty, \sigma_{\mbox{\tiny $W$}}^2 \to \infty} \frac{1}{4}\frac{ \EE _ W \left[  W^2 \phi_Z(\frac{b-W}{\sigma}) (b-W)^2   - W^2\phi_Z(\frac{-b-W}{\sigma} ) (b+W)^2   \right]  /\sigma}{\EE^2_W \left[-\phi_Z(\frac{b-W}{\sigma}) (b-W)  - \phi_Z(\frac{-b-W}{\sigma})  (b+W) \right]}
+o(1)
\\
=
&\lim_{\sigma \to \infty, \sigma_{\mbox{\tiny $W$}}^2 \to \infty} \frac{1}{64 b^2 }\frac{ \EE _ W \left[  W^2 \phi_Z(\frac{b-W}{\sigma}) (b-W)^2   - W^2\phi_Z(\frac{-b-W}{\sigma} ) (b+W)^2   \right]  /\sigma}{\EE^2_Z \left[-\sigma Z f_W(\sigma Z)      \right]}
+o(1)
\end{align*}
where in the last step we used the change of variables to go from $\EE_W$ to $\EE_Z$. The last expression converges to zero as both $\sigma \to \infty, \sigma_{\mbox{\tiny $W$}}^2 \to \infty$.

\end{proof}







\subsection{Auxiliary Results}

This section gathers results used throughout the proofs. They are of secondary interest, so we present them in this Appendix section.
 
\begin{lemma}\label{mainthem}
Let $r$ and $x$ be vectors in $\mathbb{R}^p$, $\mathcal{L}$   defined in Problem (\ref{eqDef}) and $sg(\mathcal{L},x)\in \partial\mathcal{L}(x)$   the subgradient of $\mathcal{L}$ with respect to $x$. For any $c_1,...c_5 >0$, if the following Conditions 1-5 hold, then there exists a function $\xi(\epsilon,c_1,...,c_5)\to 0$ as $\epsilon\to 0$ such that $||r||_2\leq\sqrt{p} \xi(\epsilon, c_1,...,c_5)$. 

The conditions are:
\begin{itemize}
\item[(C1)] $\| r \|_ 2 \leq c_1 \sqrt{p}$;
\item[(C2)] $\mathcal{L}(x + r) \leq \mathcal{L}(x)$;  
\item[(C3)] There exists a $sg(\mathcal{L},x)\in \partial \mathcal{L}(x)$ with $||sg(\mathcal{L},x)||_2 \leq \sqrt{p}\epsilon$;
\item[(C4)]Let $v \equiv \frac{1}{\lambda}[\sum_{i=1}^{p}\rho'(Y_i-A_i^Tx)A_i+sg(\mathcal{L},x)] \in\partial||x||_1$, {and $S(c_2)\equiv \{i\in[p]:|v_i|\geq1-c_2\}$. Then, for any $S'\subseteq [p], |S'|\leq c_3p$, we have $\sigma_{min}( \{A\}_{S(c_2)\cup S'})\geq c_4$};
\item[(C5)] The maximum and minimum non-zero singular value of A satisfy $\frac{1}{c_5}\leq {\sigma}_{min}( A)^2\leq\sigma_{max}( A)^2\leq c_5$;
\item[(C6)] For all such vectors $r$, the loss function $\rho$ satisfies  $E _i  =   \mathbb{E}(v_1'( W_i ))    \geq k_1 $ for a constant $k_1>0$.
\end{itemize}
\end{lemma}


\begin{proof}[Proof of Lemma \ref{mainthem}]
The proof follows the   strategy of  Lemma 3.1. of \cite{BM12},  with nontrivial adaptation to a  class of general loss functions. \\

Let $S = supp(x)\subseteq[p]$, where $supp(x) \equiv\{i|x_i\neq 0\}$ and $[p] = \{1,2,...,p\}$ and let $\bar S$ be its complement.
Let $r$ be the vector that satisfies Conditions $(C1)$ and $(C2)$, i.e., it is such that $\|r\|_2 \leq \sqrt{p}$ and  ${\mathcal{L}(x+r)-\mathcal{L}(x)}{} \geq 0$.
Observe that we can decompose the Lasso penalty   as follows
\begin{align}\label{eq:temp2}
\| x+r\|_1 - \|x\|_1 = \|x_S+r_S\|_1-\|x_S\|_1+\|r_{\bar{S}}\|_1,
\end{align}
as $x_S = x$ and $r = r_S+r_{\bar{S}}$.
%
%
%

Let us define a vector $v$ as 
\begin{equation}\label{eq:v}
v  \equiv \frac{1}{\lambda}[\sum_{i=1}^{p}\rho'(Y_i-A_i^Tx)A_i+sg(\mathcal{L},x)] 
\end{equation}
By observing that the  subgradients of $\mathcal{L}(x)$ satisfy $sg(\mathcal{L},x) = \lambda \partial ||x||_1 - \sum_{i=1}^n\rho'(Y_i-A_i^Tx)A_i$,  we obtain that $v_S = \partial||x_S||_1$.
Moreover, by adding and subtracting $ \langle v,r\rangle$ 
\begin{equation}\label{eq:vb}
\|r_{\bar{S}}\|_1 \geq   - p\langle \partial \|x_S\|_1,r_S\rangle   + \left( \|r_{\bar{S}}\|_1 - p\langle v_{\bar{S}}, r_{\bar{S}}\rangle ) + p \langle v,r\rangle\right)
\end{equation}
 where  $\langle u,v\rangle \equiv\frac{1}{m} \sum_{i=1}^mu_iv_i$, denotes the scalar product for $u,v \in\mathbb{R}^m$.

By observing that ${\mathcal{L}(x+r)-\mathcal{L}(x)}{} \geq 0$ and by plugging in all of the above inequalities, we conclude 
\begin{eqnarray}\label{ieq:1}
0&\overset{(iii)}{\geq}&\lambda \left(\frac{\|x_S+r_S\|_1-\|x_S\|_1}{p}-\langle \partial \|x_S\|_1,r_S\rangle \right) + \lambda \left(\frac{\|r_{\bar{S}}\|_1}{p} - \langle v_{\bar{S}}, r_{\bar{S}}\rangle \right) \\
&+& \lambda \langle v,r\rangle  - \Delta_n
\end{eqnarray}
where $(iii)$ follows from plugging equations \eqref{eq:temp2} and  (\ref{eq:vb})  in ${\mathcal{L}(x+r)-\mathcal{L}(x)}{} \geq 0$ and 
where 
\[
p \Delta_n = \sum_{i=1}^n \left[ \rho(Y_i-A_i^Tx-A_i^Tr) -  \rho(Y_i-A_i^Tx) \right].
\]

Next, we observe that
\begin{equation}\label{eq:va}
 \lambda \langle v,r\rangle =  \langle sg(\mathcal{L},x),r\rangle + p^{-1}  \sum_{i=1}^n\rho'(Y_i-A_i^Tx)(A_i^Tr)  .
\end{equation}

Let $\gamma_n$ be a sequence of positive numbers.
We define the following event
\begin{equation}
\mathcal{E}_n = \left\{ \left|\frac{\sum_{i=1}^n\rho'(Y_i-A_i^Tx)(A_i^Tr) }{p} \right| \leq \gamma_n : \forall \|r\|_2 \leq \sqrt{p}\right\}.
\end{equation}
%
%

Then, conditionally on $\mathcal{E}   $ we have 

\begin{eqnarray}\label{ieq:1b}
\gamma_n&\overset{(iii)}{\geq}& \lambda\left(\frac{\|x_S+r_S\|_1-\|x_S\|_1}{p}-\langle \partial \|x_S\|_1,r_S\rangle \right) + \lambda \left(\frac{\|r_{\bar{S}}\|_1}{p} 
- \langle v_{\bar{S}}, r_{\bar{S}}\rangle \right) \\
&+& \lambda \langle sg(\mathcal{L},x),r\rangle   - \Delta_n.
\end{eqnarray}

We discuss the last term first. 
We rewrite $ p\Delta_n$ as 
\[
 p \Delta_n = \mathbb{V}_{n}(r) + n \mathbb{E} {\rm v}_1 (r),
\]
where 
$
\mathbb{V}_n (r) = {\sum_{i=1}^n [{\rm v}_i(r) - \mathbb{E} {\rm v}_i (r)]},
$
with
$
{\rm v}_i(r) = \rho(Y_i-A_i^Tx-A_i^Tr) -  \rho(Y_i-A_i^Tx).
$

Let $\eta_n$ be a sequence of positive numbers.
Then, we consider the following event 
\[
\mathcal{V}_n = \left\{  | \mathbb{V}_n (r)| \leq \eta_n : \|r \|_2^2 \leq p \right\}.
\]

Conditioning on this event, the inequality (\ref{ieq:1}) becomes 

\begin{eqnarray}\label{ine:3}
\gamma_n + \eta_n &\overset{(iii)}{\geq}& \lambda\left(\frac{\|x_S+r_S\|_1-\|x_S\|_1}{p}-\langle \partial \|x_S\|_1,r_S\rangle \right)  \\
&&\nonumber+ \lambda \left(\frac{\|r_{\bar{S}}\|_1}{p} 
- \langle v_{\bar{S}}, r_{\bar{S}}\rangle \right) + \lambda \langle sg(\mathcal{L},x),r\rangle   +  \frac{n}{p} \mathbb{E} {\rm v}_1 (r).
\end{eqnarray}

%
%
%
Moreover, Cauchy Schwartz Inequality tells that:
\[
-\frac{||sg(\mathcal{L},x)||_2||r||_2}{p}\leq -\langle sg(\mathcal{L},x), r\rangle \leq\frac{||sg(\mathcal{L},x)||_2||r||_2}{p}.
\]
Using Conditions $(C1)$ and $(C3$, inequality (\ref{ine:3}) becomes
 \begin{eqnarray}\label{ine:4}
&&\lambda \left(\frac{\| x_S+r_S \| _1- \| x_S \| _1}{p}-\langle \mbox{sign}(x_S),r_S\rangle \right) + \lambda \left(\frac{ \| r_{\bar{S}} \|_1}{p} - \langle v_{\bar{S}}, r_{\bar{S}}\rangle \right)+\frac{n}{p} \mathbb{E} {\rm v}_1 (r) \\
& \leq &c_1\epsilon + \gamma_n + \eta_n.
\end{eqnarray}
The first two terms of the above right hand side are non-negative (proven by arguments identical to the  Lemma 3.1 in \cite{BM12}). For   the last term  we employ results of Lemma \ref{lem:uniform-rho} to obtain
\begin{eqnarray}\label{eq:70}
\frac{n}{p}    \mathbb{E}[  \rho(Y_i-A_i^Tx-A_i^Tr) -  \rho(Y_i-A_i^Tx)]
\geq  - \frac{n}{p}    \mathbb{E}[     \psi (W_i) A_i^T r]   
+ \frac{n}{p} \kappa [A_i^T r]^2
-
 o_P(1) .
\end{eqnarray}
In the display above, the first term disappears; for the second  one 
$$2 \kappa = \EE \left[ v_1' (W_i) +v_2'(W_i) \right] + \gamma,$$ for $\gamma$ defined in Lemma \ref{lem:uniform-rho}. According to Lemma \ref{lem:uniform-rho} and Condition (C6), we conclude that $\gamma$   is strictly positive. Hence, $\kappa >0$.
Therefore, there exists a constant $C>0$ such that 
\begin{eqnarray}\label{eq:eq3}
 \frac{1}{p} C  \| A r\|_2^2 \leq c_1\epsilon + \gamma_n + \eta_n  + o_P(1):=\xi_1(\epsilon).
\end{eqnarray}

To complete the proof we need to show that $\xi_1(\epsilon) \to 0$   and then employ arguments similar to  Lemma 3.1 in \cite{BM12}. This can be done by effectively bounding the size of the events $\mathcal{E}_n$ and $\mathcal{V}_n$.

The size of $\eta_n$ can be found by choosing appropriate sequence $u_n$ of Lemma \ref{lemma:bound1}.
For $u_n =  \sqrt{ (\log p)^2/ (p n )}$ we obtain that $\eta_n =n u_n = (\log p)\sqrt{n}/\sqrt{p}$ is sufficient to guarantee that $P(\mathcal{V}_n) \geq 1- \exp\{-2\log p/\kappa^2\}$.

Similarly, the size of $\gamma_n$ can be found by choosing appropriate sequence $u_n$ of Lemma \ref{lemma:bound2}.
For $u_n =  \sqrt{ (\log p)^2/ (p n  )}$ we obtain that $\eta_n =n u_n = (\log p) \sqrt{n}/\sqrt{p}$ is sufficient to guarantee that $P(\mathcal{E}_n) \geq 1- \exp\{-2\log p/\kappa^2\}$.

\end{proof}

 \begin{lemma}\label{lemma:bound1}
Let $|\rho' (u)| \leq \kappa$ for all $u \in \mathbb{R}$ and  some constant $\kappa <\infty$. Then, for all vectors $\rb$, such that $\| \rb\|_2 \leq \sqrt{p}$ and for any sequence of positive numbers $u_n \geq 0$ we have
\begin{equation}
\mathbb{P} \left( \left| \ {\sum_{i=1}^n   {\rm v}_i(\rb) - \mathbb{E} {\rm v}_i (\rb) }\right| \geq n p u_n \right) \leq \exp \left\{ - 2\frac{n^2 p u_n^2}{ \kappa^2 \log p }\right\},
\end{equation}
for $
{\rm v}_i(\rb) = \rho(Y_i-A_i^Tx_0-A_i^T\rb) -  \rho(Y_i-A_i^Tx_0).
$

\end{lemma}

\begin{proof}[Proof of Lemma \ref{lemma:bound1}]
Let $
\mathbb{V}_n (\rb) = {\sum_{i=1}^n [{\rm v}_i(\rb) - \mathbb{E} {\rm v}_i (\rb)]}.$
We begin by observing
\[
p^{-1} \mathbb{V}_n (\rb) \leq  {\sum_{i=1}^n p^{-1} \left| {\rm v}_i(\rb) - \mathbb{E} {\rm v}_i (\rb) \right|},
\]
for 
$
{\rm v}_i(\rb) = \rho(Y_i-A_i^Tx_0-A_i^T\rb) -  \rho(Y_i-A_i^Tx).
$
Then, by a Taylor expansion of the loss function $\rho$ around, we conclude 
\[
|\rho(Y_i-A_i^Tx_0-A_i^T\rb) -  \rho(Y_i-A_i^Tx_0)| \leq   |H_i(c) A_i^T \rb| 
\]
for 
$H_i(c) = \sup_{|u| \leq c} \rho'(W_i - u)  $. 
By Hoelder's  inequality  we conclude
\[
| {\rm v}_i(\rb) - \EE {\rm v}_i(\rb) |  \leq   \left| H_i(c) - \EE H_i(c)\right| \left| \langle A_i^T \rb \rangle  \right| .
\]
We proceed to bound each term in the RHS above, independently. 
For the first term, we observe that for a positive, bounded constant $\kappa$, the boundedness of the sub-gradient provides $\left| H_i(c) - \EE H_i(c)\right| \leq \kappa$.
For the second term, as $A_{ij}$ are Gaussian with variance $1/n$, by the weighted Bernstein inequality 
\begin{align}\label{eq:temp1}
\PP\left(  1/p \left| \langle A_i^T \rb \rangle   \right|  \geq a_n\right) 
&\leq 
\PP\left(  \sum_{j=1}^p \left| A_{ij} r_j \right|  \geq pa_n\right) 
\nonumber\\ 
&
\leq \exp \left\{ - \frac{p^2 a_n^2 }{4  \sum_{j=1}^p A_{ij}^2 r_{ij}^2 + 2C p \max_{j} |A_{ij}| /3} \right\}
\nonumber\\ 
&\leq
\exp \left \{- \frac{p^2 n a_n^2}{4\|\rb\|_2^2 }\right\}.
\end{align}
For all $r$ such that $\| \rb\|_2 \leq \sqrt{p} c_1$, the right hand side is smaller than $ \exp\{-  {p} n a_n^2/4 c_1\}$. Hence,  a choice of $a_n = \sqrt{\log p/ (n  {p})}$ leads that
\[
p^{-1}| {\rm v}_i(\rb) - \EE {\rm v}_i(\rb) |  =O_P\left(\sqrt{ \frac{ \log p}{np}}\right).
\]
This, in turn, guarantees that $\frac{1}{p} \mathbb{V}_n (\rb) $
 is a sum of $n$ terms, each of which is $o_P(1)$. 
 By Hoefding's inequality for bounded random variables, for any positive sequence of $u_n$  \[
\mathbb{P} \left(\frac{1}{p} |\mathbb{V}_n (\rb) | \geq   n u_n  \right)  
\leq 
\exp \left\{ - 2\frac{  n^2 u_n^2}{ \kappa^2  \sum_{i=1}^n  \frac{ \log p}{np} }\right\}
\leq
\exp \left\{ - 2\frac{  p n^2 u_n^2}{ \kappa^2 \log p }\right\}.
\]
\end{proof}

\begin{lemma}\label{lemma:bound2}
Let $|\rho' (u)| \leq \kappa$ for all $u \in \mathbb{R}$ and  some constant $\kappa <\infty$.
Then, for a positive sequence of $u_n \geq 0$ we have
\[
\PP\Bigl( \left| \langle \bigtriangledown \mathcal{L}(x_0) ,\rb\rangle  \right|\geq  u_n  \Bigl) \leq \exp\left\{-  \frac{ n p u_n^2}{ 2 \kappa^2 {(\log p)  }}\right\}.
\]

\end{lemma}

\begin{proof}[Proof of Lemma \ref{lemma:bound2}]
Let
\[
\bigtriangledown \mathcal{L}(x_0) = \sum_{i=1}^n\rho'(Y_i-A_i^Tx_0)A_i^T
\]
and observe that $\EE \bigtriangledown \mathcal{L}(x_0) =0$ by the vanishing property of the true score function $\EE \rho'(Y_i-A_i^Tx_0) =0$. Hence,
\[
\langle \bigtriangledown \mathcal{L}(x_0) ,r\rangle = p^{-1}\sum_{i=1}^n\rho'(Y_i-A_i^Tx_0)(A_i^T\rb) 
\]
and is such that $\EE \langle \bigtriangledown \mathcal{L}(x_0) ,\rb\rangle =0$.
By a triangular inequality and the bounded sub-gradient assumption 
\[
\left| \langle \bigtriangledown \mathcal{L}(x_0) ,\rb\rangle  \right| \leq \kappa  \sum_{i=1}^n q_i(\rb)
\]
with
\[
q_i(\rb) =  p^{-1} \sum_{j=1}^p  |A_{ij} r_j|.
\]
Then, $\EE q_i(\rb)=0$ as $A$ is a mean zero design matrix. 
From Lemma \ref{lemma:bound1}, equation \eqref{eq:temp1}, we conclude that
%
\[
\PP\left( \sum_{i=1}^n q_i(\rb) \geq  u_n  \right) \leq \exp\left\{-  \frac{ u_n^2}{ 2 {(\log p)  }/{(n p)}}\right\}.
\]
\end{proof}

\begin{lemma}\label{lem:uniform-rho}
Consider the model \eqref{w} with Conditions {\rm({\bf R})}, {\rm({\bf D})} and {\rm({\bf A})} satisfied. Let $\rb$ be a vector in $\RR^p$ such that $\| \rb\|_2^2 \leq Cp$ for a constant $C$: $0 < C <\infty$. Then,
\begin{equation}\label{eq:rho1}
\sup_{\| \rb\|_2 \leq \sqrt{p}}p^{-1} \left| \sum_{i=1}^n \left[\rho(Y_i -A_i^T x - A_i^T \rb) -  \rho(Y_i -A_i^T x )\right] + \rb^T \sum_{i=1}^n A_i \rho'(Y_i - A_i^T x) -\gamma \rb^T \sum_{i=1}^n A_i^T A_i \rb \right| =o_P(1),
\end{equation}
as $n$ and $p \to \infty$, with $2\gamma = \EE v_1'(W)+ \sum_{\nu=1}^{k} (\alpha_{\nu} - \alpha_{\nu-1})f_W(r_\nu)$. \end{lemma}

\begin{proof}[Proof of Lemma \ref{lem:uniform-rho}]
It suffices to prove 
\begin{equation}\label{eq:rho2}
\sup_{\| \rb\|_2 \leq \sqrt{p}} \left\| \sum_{i=1}^n A_i\left[\psi(Y_i -A_i^T x - A_i^T \rb) -  \psi(Y_i -A_i^T x )\right] + \gamma   \sum_{i=1}^n A_i^T A_i \rb \right\|_{\infty}  =O_P(\sqrt{p \log p/n})
\end{equation}
where $2\gamma = \EE v_1'(W)+ \sum_{\nu=1}^{k} (\alpha_{\nu} - \alpha_{\nu-1})f_W(r_\nu)$, together with $| \sum_{i=1}^n A_{ij} \Psi(W_i) |=O_P(1)$, for all $j=1,\dots,p$.
The above, in turn, implies through integration over $r$ the statement \eqref{eq:rho1}.  

Let $j=1,\dots,p$. We first argue that $| \sum_{i=1}^n A_{ij} \Psi(W_i) |=O_P(1)$: by Condition ${\mathbf (D)}$, $A_{ij} =O_P(\sqrt{1/n})$ and by Condition ${\mathbf (R)}$, $| n^{-1/2} \sum_{i=1}^n  \Psi(W_i)|=O_P(1)$ (bounded random variables no matter of the size of $W_i$ -- consequence of Theorem 12.1 \cite{B06}).

Next, we prove \eqref{eq:rho2}. For that end, define a stochastic process 
\[
S_n(r) = \sum_{i=1}^n A_i\left[\psi(Y_i -A_i^T x - A_i^T \rb) -  \psi(Y_i -A_i^T x )\right]
\]
for $r \in [-C\sqrt{p},C\sqrt{p}]^{p}$, for some $C$: $0<C<\infty$.
We let $\psi=v_1 +v_2$ and denote the absolute continuous and step-function components by $v_1$ and $v_2$, respectively.

Case I: $\psi=v_2$ (i.e., $v_1 =0$). 

Without loss of generality, we assume that there is a single jump-point. We set $v_2(y)$ to be $0$ or $1$ according to $y$ being $\leq 0$ or $>0$. By the vector structure in \eqref{eq:rho2}, it suffices to show that for each coordinate of $S_n(r)$ the uniform asymptotic linearity result holds for  $r \in [-C\sqrt{p},C\sqrt{p}]^{p}$. To simplify the notation, we consider only the first coordinate and drop the subscript $1$ in $S_{n1}(r)$:
\[
S_n^0(r) = S_n(r) - \EE S_n(r),
\]
where $\EE S_n(r) = \sum_{i=1}^n A_{i1} \left[ F_W(0) -F_W ( A_i^T \rb)  \right]$. By Taylor expression, we have 
$F_W(0) -F_W ( A_i^T \rb) =f_w(0)A_i^T \rb + f'_W(\xi) [A_i^T \rb]^2$, for $\xi \in (0, A_i^T \rb)$. Moreover, by \eqref{eq:temp1}, $|A_i^T \rb | = O_P(\sqrt{p \log p/n})$ and by  Condition ${\mathbf (D)}$, $A_{ij} =O_P(\sqrt{1/n})$ . Therefore, 
\[
\left| \sum_{i=1}^n A_{i1} f'_W(\xi) [A_i^T \rb]^2 \right| =O_P( (p \log p) /n  ).
\]
Hence, by Hoefding's inequality, Theorem 12.1 of \cite{B06}, we have 
\[
|S_n^0(r)| =O_P(\sqrt{p \log p/n}),
\]
for $\rb \in [-C\sqrt{p},C\sqrt{p}]^{p}$. To prove uniform asymptotic linearity we resort to the known weak convergence properties of the empirical cumulative distribution functions to the Brownian motion \citep{JS89} or by uniform  decompositions of the work of \cite{BC11}.

Case II: $\psi=v_1$ (i.e., $v_2 =0$). 
 Note that for every $\rb \in [-C\sqrt{p},C\sqrt{p}]^{p}$, by a second-order Taylor's expansion,
 \[
 \psi(Y_i -A_i^T x - A_i^T \rb) -  \psi(Y_i -A_i^T x )  = v_1'(Y_i -A_i^T x ) [ - A_i^T \rb]   +R
 \]
 where the remainder term 
 $$
 R = 1/2 \int_{Y_i -A_i^T x}^{Y_i -A_i^T x - A_i^T \rb} (Y_i -A_i^T x - A_i^T \rb - t)^2 v_1^{''}(t)dt \leq \frac{1}{2!}[A_i^T \rb]^2
 $$ 
 as $v_1^{''}(t) \leq C$ for all $t \in [Y_i -A_i^T x,Y_i -A_i^T x - A_i^T \rb ]$ and by \eqref{eq:temp1},
 is of the order $O_P(p  (\log p /n) )$.

 Now, it can be easily shown that for any $r_1$ and $r_2$ of distinct points
 \begin{align}
 \mbox{Var} \left(S_n(r_1) - S_n(r_2) \right) 
 &\leq \sum_{i=1}^n A_{i1}^2 \EE \left[ \psi(Y_i -A_i^T x - A_i^T \rb) -  \psi(Y_i -A_i^T x ) \right]^2
 \\
 & \leq K \| r_1 -r_2\|^2_2
 \end{align}
 uniformly in $r_1,r_2$, for a constant $K$: $0 < K <\infty$.
 Also, the boundedness of $v_1'$ we have 
 \begin{align}
\EE \left[S_n(r_1) - S_n(r_2) - \sum_{i=1}^n A_{i1} A_i^T (r_1 -r_2)\right]
 &\leq \sum_{i=1}^n A_{i1}^2 \EE \left[ \psi(Y_i -A_i^T x - A_i^T \rb) -  \psi(Y_i -A_i^T x ) \right]^2
 \\
 & \leq K_1 \| r_1 -r_2\|_2
 \end{align}
 uniformly in $r_1,r_2$, for a constant $K_1$: $0 < K_1 <\infty$.
 With all of the above we conclude 
  \begin{align}
 S_n(r_1) - S_n(r_2) - \sum_{i=1}^n A_{i1} A_i^T (r_1 -r_2) 
  = O_P(\sqrt{p}).
 \end{align}
To prove the compactness, we shell consider increments of $S_n(r)$ over small blocks.
For $r_2 >r_1$, the increments of $S_n(\cdot)$ over the block $B=B(r_1,r_2)$ is 
\[
S_n(B)= S_n(r_2)-S_n(r_1) = \sum_{i=1}^n A_i \psi_i(W_i;B)
\]
for $i=1,\dots, n$ and 
\[
\psi_i(W_i;B) =\psi_i(W_i - A_i^T r_2) - \psi_i(W_i - A_i^T r_1).
\]
From \eqref{eq:temp1}, $A_i^T r_2$  and $A_i^T r_1$ are of the order of $O_P(\sqrt{p \log p /n})$ and $\psi$ is a   bounded function, we have $\psi_i(W_i;B) =O_P(\sqrt{p \log p /n})$. Moreover, $\psi_i(W_i;B) =0$ if any of the arguments lay in the same interval. Hence, 
\[
\sup \left\{ |S_n(B (-K,r_2))|: -K \leq r_2 \leq K\right\} \leq \sum_{i=1}^n |A_{i1}| K \left(\| A_i r_1\|_1 + \| A_i r_2\|_1 \right) I_i
\]
 where $I_i$ are independent, non-negative indicator variables with 
 \[
 \EE I_i \leq K_1 \left(\| A_i r_1\|_1 + \| A_i r_2\|_1 \right)
 \]
 for a constant $K_1$: $0 < K_1 < \infty$.
 Hence,
 \[
\mbox{Var} \left\{ \sup \left\{ |S_n(B (-K,r_2))|: -K \leq r_2 \leq K\right\} \right\}  = O(\sqrt{p \log p/n}).
\]
\end{proof} 
 
The following lemma is a simple modification  of Lemma 5.3 of \cite{BM12}; hence, we omit the proof.
\begin{lemma}\label{lem:temp1}
Let $S\subseteq [p]$ be measurable on the $\sigma-$algebra $\sigma_t$ generated by $\{z^0,...,z^{t-1}\}$ and $\{x^0 + A^TG(z^0,b_0),...,x^t + A^TG(z^t;b_t)\}$; assume $|S|\leq p(\delta-c)$ for some $c> 0$. Then,  there exists $a_1 = a_1(c)>0$ and $a_2 = a_2(c,t)>0$, such that $\underset{S'}{min}\{\sigma_{min}(A_{S\cup S'}):S'\subseteq[p],|S'|\leq a_1p\}\geq a_2$ with probability converging to 1 as $p\to \infty.$
\end{lemma}
We apply this lemma to a specific choice of the set $S$. Defining
\[
v^t\equiv \frac{1}{\theta_{t-1}}(x^{t-1}+A^TG(z^{t-1};b_{t-1})-x^t).
\]

\begin{lemma}\label{lem:temp2}
Fix $\gamma\in(0,1)$ and let   $S_t(\gamma)\equiv \{i\in[p]:|v^t_i|\geq 1-\gamma\}$ for $\gamma\in(0,1)$.  For any $\xi>0$ there exists $t_*(\xi,\gamma)$ such that for all $t_2>t_1>t_*$,
\[
\underset{p\to\infty}{lim}\PP\{|S_{t_2} \backslash S_{t_1}|\geq p\xi\} = 0.
\]
\end{lemma}

Proof of the Lemma \ref{lem:temp2} follows exact steps as Lemma 3.5 in \cite{BM12}. The change is in the definition of the appropriate set $S_t(\gamma)$; hence, we omit the proof.

The Lemma \ref{lem:temp1} and Lemma \ref{lem:temp2} imply the following important result.

\begin{lemma} \label{lem:cond4}
There exist constants $\gamma_1\in (0,1)$, $\gamma_2 = a_1(c)/2,\gamma_3 = a_2(c,t_{min}) >0 $ and $ t_{min} <\infty$ such that, for any $t\geq t_{min}$,
\[
\min\{\sigma_{min}(A_{S_t(\gamma_1)\cup S'}): S'\subseteq [p],|S'|\leq\gamma_2p\}\geq \gamma_3,
\]
with probability converging to 1 when $p \to \infty$.
\end{lemma}
\begin{proof}[Proof of Lemma \ref{lem:cond4}]
Observe that the $\sigma$ algebra $\sigma_t$, contains$\{x^0,...,x^t\}$ by design of the RAMP algorithms. Therefore, it contains the vector $v^t$. 
By Lemma \ref{lem:delta}, the empirical distribution of $(x_0-A^TG(z^{t-1},b_t)-x^{t-1},x_0)$ converges weakly to $(\bar{\tau}_{t-1}Z,x_0)$. Now we need to check if $S_t(\gamma)\leq p(\delta-c)$
\begin{eqnarray}
\underset{p\to\infty}{lim}\frac{|S_t(\gamma)|}{p} &=&\underset{p\to\infty}{lim}\frac{1}{p}\sum_{i=1}^{p} \mathbbm{1}_{\{\frac{1}{\theta_{t-1}}|x_i^{t-1}+[A^TG(z^{t-1};b_{t-1})]_i-x_i^{t}|\geq 1-\gamma\}}\nonumber\\
&=&\underset{p\to\infty}{lim}\frac{1}{p}\mathbbm{1}_{\{\frac{1}{\theta_{t-1}}|x_0-h^t-\eta(x_0-h^t,\theta_{t-1})|\geq 1-\gamma\}}\nonumber\\
&=&\PP \left\{\frac{1}{\theta_{t-1}}|x_0+\bar{\tau}_{t-1}Z-\eta(x_0+\tau_{t-1}Z,\theta_{t-1})|\geq 1-\gamma\right\} .\label{eq:solve}
\end{eqnarray}
Because
\[
|x_0+\bar{\tau}_{t-1}Z-\eta(x_0+\tau_{t-1}Z,\theta_{t-1})| = \left\{\begin{array}{cc}
\theta_{t-1} & |x_0+\bar{\tau}_{t-1}Z|\geq \theta_{t-1}\\
|x_0+\bar{\tau}_{t-1}Z| & others
\end{array}
\right. ,
\]
from the equation \eqref{eq:solve}, we conclude
\[
\underset{p\to\infty}{lim}\frac{|S_t(\gamma)|}{p} = \EE\{\eta'(x_0+\bar{\tau}_{t-1}Z,\theta_{t-1})\}+\PP \left\{(1-\gamma)\leq\frac{1}{\theta_{t-1}}|x_0+\bar{\tau}_{t-1}Z|\leq 1\right\}.
\]
The fact that $\omega < \delta$, the first term will be strictly smaller than $\delta$ for large enough t. And the second term converges to 0. Therefore, we can choose constants $\gamma_1\to(0,1)$ and $c>0$ such that
\[
\underset{p\to\infty}{lim}\PP \left\{|S_t(\gamma_1)|<p(\delta-c)\right\} = 1.
\]
for all t larger than some $t_{min}$. For any $t\geq t_{min}$, apply  Lemma \ref{lem:temp1} for some $a_1$ and $a_2$. Fix $c>0$ and $a_1$. Let $\xi= a_1/2$ in Lemma \ref{lem:temp2}, $t_{min} = max(t_{min}, t_*(a_1/2,\gamma_1))$.  We have 
\[
\min\{\sigma_{min}(A_{S_t(\gamma_1)\cup S'}): S'\subseteq [p],|S'|\leq a_1p\}\geq a_2,
\]
together with
$
\underset{p\to\infty}{lim}\PP\{|S_t \backslash S_{t_{min}}| \geq pa_1/2\} = 0.
$
\end{proof}

\begin{proof}[Proof of Theorem \ref{thm:thm10}]
The result of Theorem \ref{thm:thm10}  follows   the same arguments as those of Theorem 1.5 of  \cite{BM12}; we observe  that   ${||x^{t+1}||_2^2}/{p}$, ${||\hat{x}||^2_2}/{p}$ are bounded and   that 
\[
\underset{p\to\infty}{lim}\frac{1}{p}\sum_{i = 1}^p\psi(\hat{x_i},x_{0,i}) =\underset{t\to\infty}{lim} \underset{p\to\infty}{lim}\sum_{i = 1}^p\psi(x_i^{t+1},x_{0,i}).
\]
 By Theorem \ref{the:the1}, we have  ${||x^{t+1}||_2^2}/{p}$ is bounded. Moreover,  an upper bound on ${||\hat{x}||^2_2}/{p}$ is guaranteed by the conditions.
\end{proof}

\end{document}